\documentclass{amsart}
\usepackage[toc,page]{appendix}
\usepackage{amsfonts}
\usepackage{amsmath}
\usepackage{setspace}
\usepackage{amsthm}
\usepackage{amssymb}
\usepackage[all]{xy}
\usepackage{color}
\usepackage{tikz}
\usepackage{mathrsfs}
\usetikzlibrary{arrows.meta}
\usetikzlibrary{patterns}
\usetikzlibrary{matrix, arrows, decorations.pathmorphing}
\usepackage{calc, xkeyval, ifthen, todonotes}
\usepackage[colorlinks,citecolor=blue]{hyperref}
\usepackage[margin=1.25in]{geometry}

\definecolor{job}{RGB}{200,65,0}
\definecolor{maitreyee}{RGB}{25,2,255}
\definecolor{jacob}{RGB}{221,7,223}
\definecolor{kaveh}{RGB}{255,255,255}

\newcommand{\bDiamond}{\mathbin{\Diamond}}

\newcommand{\textdef}[1]{\emph{#1}}

\makeatletter
\newcommand\bigDiamond{\mathop{\mathpalette\bigDi@mond\relax}}
\newcommand\bigDi@mond[2]{%
  \vcenter{\hbox{\m@th
    \scalebox{\ifx#1\displaystyle 2\else1.2\fi}{$#1\Diamond$}%
  }}%
}
\newcommand\bigLozenge{\mathop{\mathpalette\bigL@zenge\relax}}
\newcommand\bigL@zenge[2]{%
  \vcenter{\hbox{\m@th
    \scalebox{\ifx#1\displaystyle 2\else1.2\fi}{$#1\blacklozenge$}%
  }}%
}
\makeatother

\DeclareMathOperator{\Hom}{Hom}

\newcommand{\e}{\varepsilon}
\newcommand{\R}{\mathbb R}

\newcommand{\Z}{\mathbb Z}

\newcommand{\modu}{\operatorname{mod}}
\newcommand{\add}{\operatorname{add}}
\newcommand{\intc}{\underline{c}}
\newcommand{\Ext}{\operatorname{Ext}}
\newcommand{\Ind}{\operatorname{Ind}}
\newcommand{\quilt}{\mathfrak Q}
\newcommand{\nset}{[\mathbf{n},{<}]}
\newcommand{\Db}{\mathcal D^b}
\newcommand{\udim}{\underline{\operatorname{dim}}}
\newcommand{\ClosedHeart}{\mathcal C_{\mathcal Z}}
\newcommand{\TT}{\mathbf T}
\newcommand{\UZc}{\mathbb U_{\mathcal Z, \underline{c}}}
\newcommand{\HZc}{\mathbb H_{\mathcal Z, \underline{c}}}
\newcommand{\rep}{\mathop{\text{rep}}}
\newcommand{\Fix}{\mathop{\text{Fix}}}
\newcommand{\repPlus}{\add(\mathcal I^+)}
\newcommand{\gZ}{g_{\mathcal Z}}

\newtheorem{theorem}{Theorem}[section]

\newtheorem{lemma}[theorem]{Lemma}
\newtheorem{proposition}[theorem]{Proposition}
\theoremstyle{definition}
\newtheorem{example}[theorem]{Example}
\newtheorem{definition}[theorem]{Definition}
\newtheorem*{introdef}{Definition}

\newtheorem{remark}[theorem]{Remark}
\newtheorem{notation}[theorem]{Notation}
\newtheorem{question}[theorem]{Question}
\newtheorem*{introremark}{Remark}

\newtheorem{thm}{Theorem}

\begin{document}

\title{A continuous associahedron of type $A$}
\author[Kulkarni]{Maitreyee C. Kulkarni}
\address{Department of Mathematics, North Carolina State University, Raleigh, NC, USA.}
\email{mckulkar@ncsu.edu}

\author[Matherne]{Jacob P. Matherne}
\address{Department of Mathematics, North Carolina State University, Raleigh, NC, USA.}
\email{jpmather@ncsu.edu}

\author[Mousavand]{Kaveh Mousavand}
\address{Representation Theory and Algebraic Combinatorics Unit, Okinawa Institute of Science and Technology (OIST), Japan.}
\email{mousavand.kaveh@gmail.com}

\author[Rock]{Job D. Rock}
\address{Algebra group, Department of Mathematics, KU Leuven, Leuven, Belgium;
Mathematics:\ Analysis, Logic and Discrete Mathematics, UGent, Gent, Belgium}
\email{jobdaisie.rock@kuleuven.be}

\thanks{MK and JM received support from the Max Planck Institute for Mathematics in Bonn, Germany. MK, JM, and JR also received support from the Hausdorff Research Institute for Mathematics. JM also received support from the Deutsche Forschungsgemeinschaft (DFG) under Germany's Excellence Strategy - GZ 2047/1, Projekt-ID 390685813, as well as from a Simons Foundation Travel Support for Mathematicians Award MPS-TSM00007970. In the final stage of this work, KM was supported by Early-Career Scientist JSPS Kakenhi grant number 24K16908.}

\begin{abstract}
    Taking a representation-theoretic viewpoint, we construct a continuous associahedron motivated by the realization of the generalized associahedron in the physical setting.
    We show that our associahedron shares important properties with the generalized associahedron of type $A$.
    Our continuous associahedron is convex and manifests a cluster theory:
    the points which correspond to the clusters are on its boundary, and the edges that correspond to mutations are given by intersections of hyperplanes.
    This requires development of several methods that are continuous analogues of discrete methods.
    We conclude the paper by showing that there is a sequence of embeddings of type $A$ generalized associahedra into our continuous associahedron.
    
\end{abstract}

\maketitle
\setcounter{tocdepth}{1}
\tableofcontents
	\section{Introduction}
	    The classical $(n-2)$-dimensional associahedron is a convex polytope whose vertices are binary bracketings of words using $n$ symbols, and whose edges correspond to one application of associativity.  It was first discovered by Tamari while studying general questions about associativity in algebras \cite{T51} and rediscovered under the name of Stasheff polytope in the context of homotopy theory \cite{S63}. 
	    The associahedron captures the combinatorial structure of a variety of objects throughout mathematics, including triangulations of polygons, operads and homotopy theory, real moduli spaces, and cluster structures \cite{FZ01, CFZ02}.  For details of these constructions and for other realizations of associahedra, we point to \cite{CSZ15} and the references therein.
	    
	    The purpose of this paper is to introduce a continuous version of the classical associahedron guided by recent advances in cluster categories and particle physics. 
	    Our motivation for such a construction is two-fold:
	    \begin{itemize}
	        \item From \cite{FZ01,CFZ02}, the associahedron can be viewed as the ``cluster polytope" which captures the combinatorics of type $A$ cluster algebras.
	        In each of \cite{IT15} and \cite{IRT22}, the authors use representation theory of type $A$ quivers with infinitely many vertices to give an analogue for a cluster structure.
	        However, neither of these settings has a known cluster algebra.
	        In this paper, we strengthen the analogy between the continuous and finite dimensional cases by developing an analogue of the cluster polytope in the continuous setting.
	        
	        \item In \cite{ABHY18}, where the scattering amplitude of a certain quantum field theory is treated, the authors realized the corresponding amplituhedron as the generalized associahedron of type $A$. 
	        Moreover, a continuous associahedron of type $A$ appears in \cite{AHST22}, where it is obtained as an inverse limit of the generalized associahedra for $A_n$, as $n$ goes to infinity. 
	        In contrast, we
	        start from a category with a continuous cluster structure, allowing us to treat time as a continuous phenomenon from the beginning and consequently give a continuous associahedron without taking limits.
	    \end{itemize}
	    As in \cite{AHST22}, our construction of the continuous associahedron is motivated by the representation-theoretic techniques developed in \cite{B-MDMTY24}.
		
	\subsection{The continuous associahedron}
	
	Fix an algebraically closed field $k$ of characteristic $0$.  The central category considered in \cite{IT15} is the Krull--Schmidt triangulated category $\mathcal{D}$ (defined in Section~\ref{sec:the category D}), whose indecomposable objects are the points $(x,y)$ in $\R \times (-\frac{\pi}{2},\frac{\pi}{2})$, and whose shift functor is given by $(x,y)[1] := (x+\pi,-y)$ for every indecomposable object $(x,y)$. 
	Choosing a zigzag $\mathcal{Z}$ in $\mathcal{D}$ (Definition~\ref{def:zigzag}) determines a $t$-structure $(\mathcal{D}^{\le 0},\mathcal{D}^{\ge0})$ on $\mathcal{D}$ (Section~\ref{sec:t-structures}).  One may ``quilt" off of $\mathcal{Z}$ (Definition~\ref{def:quilt of zigzag}), just as one knits off of the projective slice in $\mathcal{D}^b(A_n)$.  The indecomposable objects of $\mathcal{D}^{\le 0}$ are those obtained by quilting off $\mathcal{Z}$, and the indecomposable objects of $\mathcal{D}^{\ge 0}$ are those obtained by inverse quilting off $\mathcal{Z}[1]$.  We write $\mathcal{D}^\heartsuit := \mathcal{D}^{\le0} \cap \mathcal{D}^{\ge0}$ for the heart of this $t$-structure and note that it plays an analogous role to $\rep(A_n)$ in our continuous story.  See Section~\ref{sec:additional properties of heart} for more on this philosophy. 
    
    The setting for the continuous associahedron is the category
	$
	\ClosedHeart = \add\left(\Ind(\mathcal D^\heartsuit) \sqcup \Ind(\mathcal Z[1]) \right)
	$.
	The category $\ClosedHeart$ has a continuous analogue of the mesh relations appearing in the Auslander--Reiten theory of algebras of type $A$:
	this allows us to define continuous deformed mesh relations (Definition \ref{def:contintuous deformed mesh relations}) using a function $\intc:\Ind(\ClosedHeart)\to \R_{>0}$ analogous to the construction in \cite{B-MDMTY24}.  For the remainder of the introduction, we fix a zigzag $\mathcal{Z}$ and such a function $\intc$.
	
	\begin{introdef}[Definition \ref{def:continuous associahedron}]
	The continuous associahedron $\UZc$ is the subset of $\prod_{\Ind(\ClosedHeart)}\R$ consisting of nonnegative solutions of the continuous deformed mesh relations with respect to $\intc$.
	\end{introdef}
	
	\begin{thm}[Theorem \ref{thm:convex}]\label{thm:intro:convex}
		The continuous associahedron $\UZc$ is convex in the sense that any line segment in $\prod_{\Ind(\ClosedHeart)}\R$ whose endpoints are in $\UZc$ is entirely contained in $\UZc$.
	\end{thm}

\subsection{The continuous associahedron as a cluster ``polytope''}

    Using the triangulated structure in $\mathcal D$, we define compatibility of a pair of indecomposables in $\ClosedHeart$ (Definition \ref{def:compatibility}).
    A $\TT$-cluster $\mathcal T$ (Definition \ref{def:T-cluster}) is a maximal collection of pairwise compatible indecomposables in $\ClosedHeart$.
    We define an exchange relation $\mathcal T\to (\mathcal T\setminus\{X\})\cup\{Y\}$, called $\TT$-mutation (Definition \ref{def:algebra mutation}), that replaces exactly one indecomposable $X \in \mathcal{T}$ with a new indecomposable $Y\notin \mathcal{T}$ whenever possible.
    We say a solution $\Phi$ corresponds to a $\TT$-cluster $\mathcal T$ if $\Phi(X)=0$ for all $X\in\mathcal T$.  Given a $\TT$-cluster $\mathcal{T}$, it is not known whether a solution corresponding to $\mathcal{T}$ exists, and if it exists, whether or not it is unique.
	We refer the reader to Question \ref{quest:open problem} and the discussion preceding it for further details.
	
	\begin{introremark}
	    A key difference to the cluster structure in \cite{B-MDMTY24} is that one is not necessarily able to exchange an arbitrary object in a $\TT$-cluster (see Example \ref{xmp:mutation}).
	    Additionally, our category $\ClosedHeart$ is not a cluster category in the sense of \cite{BMRRT} because $\ClosedHeart$ is not an orbit category of $\mathcal D$.
	    We do, however, have a cluster theory as in \cite[Definition 5.1.1]{IRT22}.
	\end{introremark}
	
	Cluster polytopes were introduced by Fomin and Zelevinsky \cite{FZ01} and their polytopality was proved by Chapoton, Fomin, and Zelevinsky \cite{CFZ02}.
	In particular, the vertices correspond to clusters and the edges to mutation.
	We show similar behavior in $\UZc$.
	To each indecomposable $X$ in $\ClosedHeart$, we associate the hyperplane $\mathbb H_X$ in $\prod_{\Ind(\ClosedHeart)}\R$ given by setting the $X$-coordinate to 0. Additionally, using Theorem \ref{thm:intro:convex}, we define a point $X$ to be on the boundary of $\UZc$ if there exists a nontrivial line segment (a line segment with distinct endpoints) ending at $X$ that cannot be extended inside $\UZc$ such that $X$ is no longer an endpoint (Definition \ref{def:boundary}).
	
	\begin{thm}[Theorems \ref{thm:clusters are extremal} and \ref{thm:mutation on associahedron}]
	    If $\mathcal{T}$ is a $\TT$-cluster, then a solution corresponding to $\mathcal{T}$ is on the boundary of $\UZc$. 
		Let $\mathcal T$ and $\mathcal T'$ be two $\TT$-clusters, each of which uniquely correspond to two respective points $\Phi$ and $\Phi'$ on the boundary of $\UZc$.
		Then there is a $\TT$-mutation $\mathcal T\to \mathcal T'$ if and only if there is an edge in $\UZc$ connecting $\Phi$ to $\Phi'$, which is given by $\bigcap_{X\in\mathcal T\cap\mathcal T'} \mathbb H_X$.
	\end{thm}
	
	In Section~\ref{sec:finite embeddings}, we embed the finite dimensional generalized associahedra $\mathbb{U}_{n,\intc}$ given in \cite{B-MDMTY24}, for all $n$, into the continuous associahedron $\UZc$ in a way that preserves the cluster structures in the sense of the following theorem.

	\begin{thm}[Theorem \ref{thm:big finite embedding}]
		There is an infinite sequence of embeddings
		\begin{displaymath}
			\mathbb U_{2,\intc}\hookrightarrow\mathbb U_{3,\intc}\hookrightarrow\cdots\hookrightarrow\mathbb U_{n,\intc}\hookrightarrow\mathbb U_{n+1,\intc}\hookrightarrow\cdots  \UZc.
		\end{displaymath}
		For $\mathbb U_{n,\intc}\hookrightarrow\mathbb U_{n+1,\intc}$, cluster vertices are taken to cluster vertices and mutation edges are taken to mutation edges.
		For $\mathbb U_{n,\intc}\hookrightarrow \UZc$, cluster vertices are taken to solutions corresponding to $\TT$-clusters and mutation edges are taken to edges corresponding to $\TT$-mutations.
	\end{thm}
	
	\subsection{Outline}
	In Section~\ref{sec:Amplituhedron for Dynkin Quivers}, we review the representation-theoretic construction of the finite dimensional associahedron in \cite{B-MDMTY24}.  In the rest of Section~\ref{sec:finiteandd}, we introduce the category $\mathcal{D}$ and prove some basic properties.  Section~\ref{Sec:Continuous Deformed Mesh Relations} is dedicated to introducing the continuous deformed mesh relations using a process called ``quilting", which is analogous to knitting in the finite dimensional case.
	We develop connections to representations of quivers in Section \ref{sec:rep theory}, where we prove that the hearts of certain $t$-structures have several of the same properties as finitely generated representations of type $A_n$ quivers.
	In Section \ref{sec:T-clusters}, we introduce $\TT$-clusters and $\TT$-mutation.
	Our last section, Section \ref{sec:associahedron}, is devoted to the definition of the associahedron $\UZc$ and the proofs of our main results.
	
	\subsection{Setting and notation}
	We work over an algebraically closed field $k$ of characteristic $0$ (for example, the field of complex numbers).
	For a finite dimensional $k$-algebra $\Lambda$, by $\modu \Lambda$ we denote the category of all finite dimensional left $\Lambda$-modules and $\Ind(\Lambda)$ denotes the set of all isomorphism classes of indecomposable objects in $\modu \Lambda$. 
    To $\modu \Lambda$, we associate its Auslander--Reiten quiver $\Gamma_\Lambda$, where the set of vertices of $\Gamma_\Lambda$ is in bijection with $\Ind(\Lambda)$ and the arrows between two vertices $X$ and $Y$ in $\Gamma_\Lambda$ correspond to the irreducible morphisms in $\Hom_\Lambda(X,Y)$. We denote by $\mathcal{D}^b(\Lambda)$ the bounded derived category of $\Lambda$, and consider the full additive subcategory $\mathcal{C}$ of $\mathcal{D}^b(\Lambda)$ generated by
    $\Ind(\Lambda) \sqcup \{P_i[1] \mid i\in Q_0\}$, where $P_i[1]$ is the shift of the projective module $P_i$ in $\mathcal{D}^b(\Lambda)$.
    For a detailed treatment of the representation theory of finite dimensional algebras, and for all undefined terms, we point to \cite{ASS06}. In this paper, we write $\mathcal D^b(A_n):=\mathcal D^b(\Lambda)$, where $\Lambda=kA_n$ is the path algebra of the type $A_n$ quiver.
	
    \subsection*{Acknowledgements} This research was part of the Junior Trimester Program at the Hausdorff Research Institute for Mathematics (HIM) in Bonn, Germany. MK, JM, and JR would like to thank the HIM for the financial support and the stimulating working environment.
    The authors would like to thank Hugh Thomas for insightful conversations and numerous helpful comments on an earlier draft of this manuscript.
    Also, the authors would like to thank Nima Arkani-Hamed and Giulio Salvatori for stimulating discussions. 
    Finally, we thank the anonymous referees for helpful suggestions and references---these have greatly improved the quality of the manuscript.
	
\section{Finite case and category $\mathcal D$}\label{sec:finiteandd}
		
\subsection{Amplituhedron for Dynkin quivers}\label{sec:Amplituhedron for Dynkin Quivers}
The amplituhedron studied by physicists in \cite{ABHY18} was constructed  representation theoretically in \cite{B-MDMTY24}.  The main goal of this paper is to introduce an analogue of this construction in the continuous setting.  To set the scene, we use this section to briefly recall the construction in \cite{B-MDMTY24}, preferring to stick to an example rather than inundating the reader with technical details.

\begin{example}\label{xmp:Running example}
Let $Q$ be the quiver given by the following orientation of $A_5$:
\begin{displaymath}
\xymatrix{1 & 2 \ar[l]_{\alpha}  & 3 \ar[l]_{\beta}  \ar[r]^{\gamma}  & 4& 5   \ar[l]_{\delta} }
\end{displaymath}
where $\{3,5\}$ is the set of sources and $\{1,4\}$ consists of sinks in $Q_0$.  

The algebra $\Lambda=kQ$ is finite dimensional and $B:=\{e_1,e_2,e_3,e_4,e_5, \alpha, \beta, \gamma, \delta, \alpha\beta \}$ is a $k$-basis for $\Lambda$, where $e_i$ denotes the trivial path at vertex $i$, for each $1\leq i \leq 5$.
Moreover, the multiplication in $\Lambda$ is given by the composition of directed paths, whenever possible, and zero otherwise. For example $0\neq \alpha\beta \in B$, but $\beta\alpha=0$ and $\gamma\beta=0$.

In this case, $\Gamma_\Lambda$ is the subquiver of the quiver in Figure \ref{fig:augmented} determined by the solid vertices (those of the form $\bullet$) and the arrows between them.  In particular, the leftmost copy of $Q^{op}$ in $\Gamma$, which appears in bold, has the projective indecomposable modules $P_i:= \Lambda e_i$ as its vertices, where $1\leq i \leq 5$.
Moreover, the additive subcategory $\mathcal{C}$ of $\mathcal{D}^b(kA_{5})$ is generated by
$\Ind(kA_5) \sqcup \{P_i[1] \mid 1\leq i \leq 5\}$, where $P_i[1]$ is the shift of the projective module $P_i$ in $\mathcal{D}^b(kA_5)$. The Auslander--Reiten quiver $\Gamma_\mathcal{C}$ is obtained by extending $\Gamma_\Lambda$, as further explained below.
\end{example}

In Figure~\ref{fig:augmented}, $\Gamma_\mathcal{C}$ is the quiver consisting of both solid and star vertices, together with all arrows between them.
The vertices depicted by $\circ$ and the dotted arrows in Figure \ref{fig:augmented} are called \textdef{virtual}, for the reason we soon explain in the following.
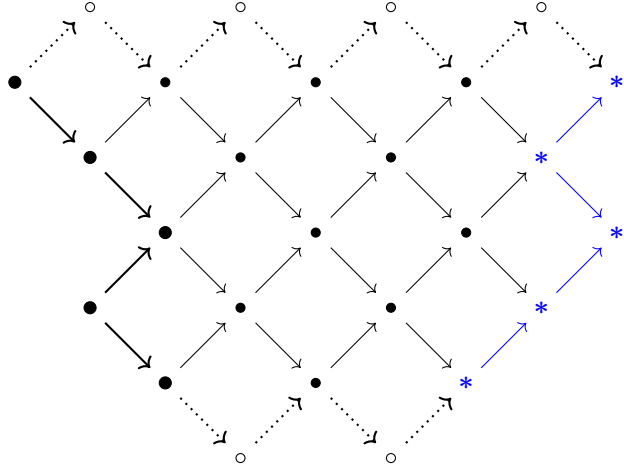
\begin{figure}
\begin{center}
\begin{tikzpicture}
\filldraw (0,4) circle[radius=.8mm];
\draw[->, dotted, thick] (0.2,4.2) -- (0.8,4.8);
\draw[->, thick] (0.2,3.8) -- (0.8,3.2);
\filldraw (1,3) circle[radius=.8mm];
\draw[->] (1.2, 3.2) -- (1.8,3.8);
\draw[->, thick] (1.2,2.8) -- (1.8,2.2);
\filldraw (2,2) circle[radius=.8mm];
\draw[->] (2.2, 2.2) -- (2.8,2.8);
\draw[->] (2.2, 1.8) -- (2.8,1.2);
\filldraw (1,1) circle[radius=.8mm];
\draw[->, thick] (1.2,1.2) -- (1.8,1.8);
\draw[->, thick] (1.2,0.8) -- (1.8,0.2);
\filldraw (2,0) circle[radius=.8mm];
\draw[->] (2.2,0.2) -- (2.8,0.8);
\draw[->, dotted, thick] (2.2,-0.2) -- (2.8,-0.8);

\foreach \x in {2,4,6} 
	\filldraw (\x, 4) circle[radius=.6mm];
\foreach \x in {3,5} 
{
	\filldraw (\x, 3) circle[radius=.6mm];
	\filldraw (\x, 1) circle[radius=.6mm];
	\filldraw (\x+1,2) circle[radius=.6mm];
}
\filldraw (4,0) circle[radius=.6mm]; 

\foreach \x in {1,3,5,7}
	\filldraw[fill=white] (\x,5) circle[radius=.6mm];
\foreach \x in {3,5}
	\filldraw[fill=white] (\x,-1) circle[radius=.6mm];
	
\draw[blue] (8,4) node {$\boldsymbol{*}$};
\draw[blue] (7,3) node {$\boldsymbol{*}$};
\draw[blue, ->] (7.2,3.2) -- (7.8,3.8);
\draw[blue, ->] (7.2,2.8) -- (7.8,2.2);
\draw[blue] (8,2) node {$\boldsymbol{*}$};
\draw[blue] (7,1) node {$\boldsymbol{*}$};
\draw[blue,->] (7.2,1.2) -- (7.8,1.8);
\draw[blue] (6,0) node {$\boldsymbol{*}$};
\draw[blue,->] (6.2,0.2) -- (6.8,0.8);

\foreach \x in {1,3,5,7}
	\draw[->, dotted, thick] (\x+0.2,4.8) -- (\x+0.8,4.2);
\foreach \x in {3,5}
	\draw[->, dotted, thick] (\x+0.2,-0.8) -- (\x+0.8,-0.2);
	
\foreach \x in {2,4,6}
{
	\draw[->, dotted, thick] (\x+0.2, 4.2) -- (\x+0.8, 4.8);
	\draw[->] (\x+0.2, 3.8) -- (\x+0.8, 3.2);
}
\foreach \x in {3,5}
{
	\draw[->] (\x+0.2, 3.2) -- (\x+0.8,3.8);
	\draw[->] (\x+0.2, 2.8) -- (\x+0.8,2.2);

	\draw[->] (\x+1.2, 2.2) -- (\x+1.8,2.8);
	\draw[->] (\x+1.2, 1.8) -- (\x+1.8,1.2);
	
	\draw[->] (\x+0.2, 1.2) -- (\x+0.8,1.8);
	\draw[->] (\x+0.2, 0.8) -- (\x+0.8,0.2);
}
\draw[->] (4.2,0.2) -- (4.8,0.8);
\draw[->, dotted, thick] (4.2,-0.2) -- (4.8,-0.8);
\end{tikzpicture}
\stepcounter{equation}
\caption{Augmented Auslander--Reiten quiver for $1\leftarrow 2\leftarrow 3\rightarrow 4\leftarrow 5$.}
\label{fig:augmented}
\end{center}
\end{figure}

Following the same notation from Example \ref{xmp:Running example}, if $Q$ is a simply-laced Dynkin quiver and $\Lambda=kQ$, the leftmost copy of $Q^{op}$ in $\Gamma_\Lambda$ is called the \textdef{projective slice}, while the rightmost copy of $Q^{op}$ in $\Gamma_{\mathcal{C}}$ is called the \textdef{shifted projective slice}. 
This is because for all $i\in Q_0$, the vertices of these copies of $Q^{op}$ respectively correspond to the projective indecomposable modules $P_i=\Lambda e_i$ and their shifts $P_i[1]$ in $\mathcal{D}^b(\Lambda)$.
For each $P_i$, the dimension vector $\underline{\dim}(P_i) \in \mathbb{Z}^{Q_0}_{\geq 0}$ is given as follows: For any $j \in Q_0$, the $j$th coordinate of $\underline{\dim}(P_i)$ is the number of paths from $i$ to $j$ in $Q$.
For each quiver $Q$ treated in this work, and each pair of vertices $i$ and $j$ in $Q$, there is at most one directed path from $i$ to $j$. Hence, for a fixed vertex $i$, the coordinates of the dimension vector $\underline{\dim}(P_i)$ in $Q_0^{\mathbb{Z}_{\geq 0}}$ always belong to $\{0,1\}$, whereas for an arbitrary quiver $Q$ the coordinates of $\underline{\dim}(P_i)$ can be arbitrarily large. In particular, we always have the $i$th coordinate of $\underline{\dim}(P_i)$ is $1$.
Since the Auslander--Reiten theory induces mesh relations on $\Gamma_{\mathcal{C}}$, we can start from the projective slice and dimension vectors of projective modules and use the mesh relations to determine the dimension vector $\underline{\dim}(X)$, for each vertex $X$ of $\Gamma_{\mathcal{C}}$. 

It is well known that each vertex of $\Gamma_{\mathcal{C}}$ corresponds to a cluster variable of the cluster algebra associated to the simply-laced Dynkin quiver $Q$.  Recall that two cluster variables are said to be \textdef{compatible} if there is some cluster which contains both of them. Consequently, two elements of $\Gamma_{\mathcal{C}}$ are called compatible if the corresponding cluster variables are such.

The case that generalizes to our continuous setting is when we take $\Lambda = kQ$ for $Q$ a type $A_n$ quiver (for $n \in \mathbb{Z}_{>1}$).
Here, we consider
$\mathcal{C}:=\add \left(\Ind(\Lambda) \sqcup \big\{P_i[1] \mid i\in Q_0 \big\}\right)$ as a full subcategory of $\mathcal{D}^b(\Lambda)$, and we extend the Auslander--Reiten quiver $\Gamma_{\mathcal{C}}$ to a larger quiver $\widetilde{\Gamma}_A$ which plays a prominent role in our studies.
For an inductive construction of $\widetilde{\Gamma}_A$ compatible with our continuous setting, we fix the following particular configuration of $\Gamma_{\mathcal{C}}$ as the initial step:

\begin{enumerate}
\item[($0$)]
Suppose $\Gamma_{\mathcal{C}}^0$ is a realization of $\Gamma_{\mathcal{C}}$ in $\mathbb{R}^2$ such that each arrow of $\Gamma_{\mathcal{C}}$ is of unit length which makes a $\pm 45$-degree angle with the horizon, and the Auslander--Reiten translation is a horizontal shift to the left.

It follows that any mesh relation in $\Gamma_{\mathcal{C}}^0$ forms a diamond or an isosceles triangle. Now that the initial step of our inductive construction is described, we extend $\Gamma_{\mathcal{C}}^0$ to the desired translation quiver $\widetilde{\Gamma}_\Lambda$. In particular, provided $\Gamma_{\mathcal{C}}^i$ is constructed, the next step is as follows:

\item[($i+1$)] In $\Gamma_{\mathcal{C}}^i$, if $x$ is a leftmost vertex of odd degree, add an arrow $\delta_x$ of unit length outgoing from $x$ and pointing to the East such that $\delta_x$ is orthogonal to all arrows of $\Gamma_{\mathcal{C}}^i$ connected to $x$. Denote the resulting quiver by $\Gamma_{\mathcal{C}}^{i+1}$.
\end{enumerate}

We say $\delta_x$ is a \textdef{virtual arrow} of $\Gamma_{\mathcal{C}}^{i+1}$. Moreover, if $e(\delta_x)$ does not belong to $\Gamma_{\mathcal{C}}$, it is called a \textdef{virtual vertex} of $\Gamma_{\mathcal{C}}^{i+1}.$ Consequently, $\Gamma_{\mathcal{C}}^{i+1}$ is said to be the \textdef{virtual extension} of $\Gamma_{\mathcal{C}}^{i}$ by $\delta_x$.
Finally, if $\Gamma_{\mathcal{C}}^{i+1}$ has no vertex of odd degree, set $\widetilde{\Gamma}_\Lambda:=\Gamma_{\mathcal{C}}^{i+1}$ and call it the \textdef{augmented Auslander--Reiten quiver} of $A$. From the construction it is easy to see that there is a unique $i \in \mathbb{Z}$ with $\widetilde{\Gamma}_\Lambda=\Gamma_{\mathcal{C}}^{i+1}$.

We observe that $\widetilde{\Gamma}_\Lambda$ can be tiled by squares (symmetric diamonds) of the same size. Moreover, the projective slice (resp. the shifted projective slice) in $\Gamma_{\mathcal{C}}$ determines the leftmost (resp. rightmost) border of $\widetilde{\Gamma}_\Lambda$. All arrows in $\widetilde{\Gamma}_\Lambda \setminus \Gamma_{\mathcal{C}}$ are virtual, but there are virtual arrows whose endpoints belong to $\Gamma_{\mathcal{C}}$, meaning that $e(\delta_x)$ is not necessarily virtual. In fact, the top and bottom rows of $\widetilde{\Gamma}_\Lambda$ consists of virtual vertices and each virtual vertex of $\widetilde{\Gamma}_\Lambda$ belongs to exactly one of these two rows. 

For instance, if $Q$ is the quiver from Example \ref{xmp:Running example}, then in Figure \ref{fig:augmented} we can see the explicit construction of the augmented quiver $\widetilde{\Gamma}_\Lambda$ for $\Lambda=kQ$ via the above algorithms. In this case, we iterate the second step of the algorithm twelve times, each time creating a new virtual arrow shown by dotted arrows. The virtual vertices are depicted by circle (nonsolid) vertices, which form the top and bottom rows of vertices in $\widetilde{\Gamma}_\Lambda$.

Recall that each vertex of $\Gamma_{\mathcal{C}}$ comes with its own dimension vector. Further, to any virtual vertex $v$ of $\widetilde{\Gamma}_\Lambda$ we associate the \textdef{virtual dimension vector}, being $\underline{\dim}(v)=(0,0, \ldots,0) \in \mathbb{Z}^{Q_0}$. 
For each diamond tile $\bDiamond$ in $\widetilde{\Gamma}_\Lambda$, let $L_{\Diamond}, R_{\Diamond}, U_{\Diamond}$, and $D_{\Diamond}$ respectively denote the left, right, top and bottom vertices of $\bDiamond$.
Then, thanks to the Auslander--Reiten translation in $\Gamma_{\mathcal{C}}$ and the induced mesh relations, the following identity holds for any diamond $\bDiamond$ in $\widetilde{\Gamma}_\Lambda$:
$$\underline{\dim}(L_{\Diamond}) + \underline{\dim}(R_{\Diamond}) =  \underline{\dim}(U_{\Diamond}) + \underline{\dim}(D_{\Diamond}).$$

From the above identity and a straightforward computation, one obtains similar equations for any rectangular area that entirely lies in the quiver $\widetilde{\Gamma}_\Lambda$. In particular, every such rectangle is tiled with the symmetric diamonds. Via the cancellations induced by these diamonds, an analogous equation holds for the rectangle, where $L_{\Diamond}, R_{\Diamond}, U_{\Diamond}$, and $D_{\Diamond}$ should be replaced by the appropriate corners of the rectangle. 

\begin{remark}\label{rmk:finite compatibility}
The notion of compatibility of vertices in $\Gamma_{\mathcal{C}}$ also has a more homological incarnation and can be phrased in terms of extensions between the corresponding indecomposable objects in $\mathcal{C}$. In particular, as shown in \cite{BMRRT06}, two vertices $v$ and $w$ of $\Gamma_{\mathcal{C}}$, with the corresponding indecomposable objects $M_v$ and $M_w$ in $\mathcal{C}$, are compatible if and only if $\Ext^1_{\mathcal{C}}(M_v,M_w)=\Ext^1_{\mathcal{C}}(M_w,M_v)=0$.
As described above, we can view $v$ and $w$ as vertices of $\widetilde{\Gamma}_\Lambda$.
This implies that $M_v$ and $M_w$ are incompatible if and only if there exists a rectangle that fully lies in $\widetilde{\Gamma}_\Lambda$ whose left and right corners are $v$ and $w$ (or $w$ and $v$).
\end{remark}

Suppose $\mathcal{I}^{+}$\label{tag:I +} denotes the set of vertices in $\widetilde{\Gamma}_\Lambda$ associated to $\Ind(\Lambda)$. The dimension vectors for all such vertices are nonzero with nonnegative coordinates. This justifies the choice of notation $\mathcal{I}^{+}$.
From the construction of $\widetilde{\Gamma}_\Lambda$, it follows that any vertex in $\widetilde{\Gamma}_\Lambda$ which does not belong to $\mathcal{I}^{+}$ is either a virtual vertex or it is associated to $P_i[1]$, for some $i\in Q_0$. In the former case the associated dimension vector is virtual and thus the zero vector, but in the latter case the associated dimension vector is nonzero with nonpositive coordinates. In particular, the set of vertices of $\widetilde{\Gamma}_\Lambda$ is a disjoint union of the form $\mathcal{I}^{+}\sqcup \mathcal{I}^{v} \sqcup \mathcal{I}^{[1]}$, where $\mathcal{I}^{v}$ denotes the set of virtual vertices, and $\mathcal{I}^{[1]}$ consists of vertices associated to the shifted projective indecomposable objects $P_i[1]$ in $\mathcal{C}$, for all $i\in Q_0$.
For the dimension vectors associated to the vertices of $\widetilde{\Gamma}_\Lambda$, 
we often refer to those corresponding to $\mathcal{I}^{+}$, $\mathcal{I}^{v}$ and $\mathcal{I}^{[1]}$ respectively as the positive, zero and negative dimension vectors.

In addition to the set of dimension vectors, the collection of $g$-vectors corresponding to the positive and negative vertices of $\widetilde{\Gamma}_\Lambda$ satisfy all the mesh relations. It is well known that the $g$-vectors of the indecomposable modules in the projective slice, being the leftmost boundary of $\widetilde{\Gamma}_\Lambda$, are given by the standard basis of $\mathbb{R}^{Q_0}$. Therefore, if to each vertex in $\mathcal{I}^v$ we again associate the zero vector, one can use the mesh relations in $\widetilde{\Gamma}_\Lambda$ to compute the $g$-vectors of all the remaining indecomposable modules in $\widetilde{\Gamma}_\Lambda$. This new set of vectors has played a pivotal role in the generalization of the setting and results in \cite{ABHY18}. This is because the physical phenomena studied by the authors resulted in a system of equations which could be realized as the deformed mesh relations that hold in the augmented quiver $\widetilde{\Gamma}_\Lambda$, where $\Lambda=kA_n$ for the linearly oriented quiver $A_n$. In particular, for each diamond $\bDiamond$ in $\widetilde{\Gamma}_\Lambda$, the \textdef{deformed mesh relation} is of the form 
\begin{equation}\label{defmesh}
    \underline{\dim}(L_{\Diamond}) + \underline{\dim}(R_{\Diamond}) =  \underline{\dim}(U_{\Diamond}) + \underline{\dim}(D_{\Diamond})+ c_{\Diamond}
\end{equation}

\noindent where $c_{\Diamond}$ is a nonnegative real value associated to ${\Diamond}$. We observe that for each rectangular area that entirely lies in the quiver $\widetilde{\Gamma}_\Lambda$, the deformed mesh relations imply a similar equation, while the constant $c_{\Diamond}$ must be replaced by the sum of all $c_{\Diamond}$'s for the diamonds ${\Diamond}$ that tile the rectangle. 
Furthermore, it is obvious that if $c_{\Diamond}=0$, for every diamond ${\Diamond}$ in $\widetilde{\Gamma}_\Lambda$, we get the old system of equations induced by the mesh relations.

In retrospect, in \cite{ABHY18} the authors considered a polytopal realization of the space induced by the deformed mesh relations coming from $\widetilde{\Gamma}_\Lambda$, where $\Lambda=kA_n$ is given by a linearly oriented quiver $A_n$. This phenomenon was in fact viewed as a geometric description of the scattering amplitudes for bi-adjoint $\phi^3$ scalar theory and plays an analogous role to that of a particular semialgebraic set in a Grassmannian, called the  \emph{amplituhedron}, which encodes scattering amplitudes for $\mathcal{N}=4$ super Yang--Mills theory. Due to the analogous nature of the aforementioned polytopal realization in \cite{ABHY18}, we adopt this terminology and henceforth refer to that as the amplituhedron for bi-adjoint $\phi^3$ scalar theory.
This realization of the amplituhedron was generalized to all simply-laced Dynkin quivers in terms of \textdef{generalized associahedra}---the polyhedra whose normal fan is given by the $g$-vectors (see \cite{B-MDMTY24}). This last realization occurs in a kinematic space $\mathbb V=\prod_{\mathcal I}\mathbb R$, where $\mathcal{I}:=\mathcal{I}^{+}\sqcup \mathcal{I}^{[1]}$.
Provided $\underline{c}=(c_{\Diamond})_{\Diamond \in \widetilde{\Gamma}_\Lambda}$ is a collection of positive integers, in \cite{B-MDMTY24} the authors consider an $n$-dimensional affine space $\mathbb E_{\underline{c}}$
inside $\mathbb V$ determined by \underline{c}. Here $n$ is the number of vertices of the quiver. 
In fact, $\mathbb E_{\underline{c}}$ is
induced by the $\underline{c}$-deformed mesh relations of the form 
$$p_{i,j} + p_{i+1,j} = c_{ij} + \sum_{(i,j) \rightarrow (i',j') \rightarrow (i+1,j)} p_{i',j'},$$
where $p_{i,j}$ is the coordinate function on $\mathbb V$ indexed by $(i,j)$ in $\mathcal{I}$.
Let $\mathbb U_{\underline{c}}$ denote the intersection of
the positive orthant in $\mathbb V$ with $\mathbb E_{\underline{c}}$.
In \cite{B-MDMTY24},
the authors further use these $p_{i,j}$ and the $g$-vectors to define a projection of $\mathbb V$ onto the $n$-dimensional real space. In particular, they consider $\mathbb A_{\underline{c}}=\pi(\mathbb U_{\underline{c}})$, where $\pi:\mathbb V\rightarrow \mathbb R^n$, whose $k$th
coordinate is given by $p_{ij}$ for $(i,j)\in \mathcal{I}^{[1]}$ such that $g(i,j)=-e_k$. Here $e_k$ denotes the $k$th standard bases of $\mathbb R^n$. More specifically, they show the following theorem. For further details, see \cite{B-MDMTY24}.

\begin{theorem}\label{th-one} 
\begin{enumerate}
\item Each facet of $\mathbb U_{\underline{c}}$ is defined by the vanishing of exactly one
coordinate of $\mathbb V$. Moreover, the vertices of $\mathbb U_{\underline{c}}$  correspond to clusters.
\item The faces of $\mathbb A_{\underline{c}}$
correspond bijectively to compatible sets in $\mathcal I$.
\end{enumerate}\end{theorem}

In our treatment of the more general setting which will be discussed in the following sections, $\underline{c}$ need not be an integer vector, meaning that for each diamond $\bDiamond$ in $\widetilde{\Gamma}_\Lambda$, we only assume $c_{\bDiamond}$ is a nonnegative real value.

\begin{remark}
Since the dimensions of the kinematic space $\mathbb V=\mathbb \prod_{\mathcal I} R$ and the subspace $\mathbb U_{\underline{c}}$ grow fast, visualization of the amplituhedron $\mathbb A_{\underline{c}}$ is possible only for small Dynkin quivers. 
In Figure \ref{fig:ABHY illustration} we consider the case $Q=A_2$. In particular, $\mathbb A_{\underline{c}}$ is the ordinary associahedron for $\underline{c}\in \mathbb{R}^3_{>0}$, where the $X_{ij}$ denote the (affine-)linear forms on $\mathbb{R}^{\mathcal I}$, indexed by the vertices of $\Gamma_{\mathcal{C}}$, as labelled below:

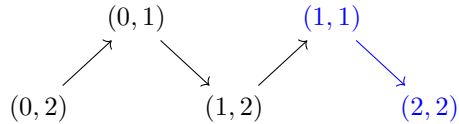
\[\begin{tikzpicture}[xscale=1.3, yscale=1]
    \draw (-3,0) node (02)   {$(0,2)$};
    \draw (-2,1.2) node (01) {$(0,1)$};
    \draw (-1,0) node (12)   {$(1,2)$};
    \draw (0,1.2) node (11)  {${\color{blue}(1,1)}$};
    \draw (1,0) node  (22)   {${\color{blue}(2,2)}$};
    \draw [->] (02) -- (01);
    \draw [->] (01) -- (12);
    \draw [->] (12) -- (11);
    \draw [blue,->] (11) -- (22);
\end{tikzpicture} \]
We note that each choice of $\underline{c} \in \mathbb{R}^3_{\geq 0}$ yields different deformed mesh relations and the corresponding system of equations gives rise to different generalized associahedra $\mathbb A_{\underline{c}}$ in $\mathbb{R}^2$.  Regardless of the choice of $\intc$, we remark that all the generalized associahedra $\mathbb A_{\underline{c}}$ have two pairs of parallel facets, which are $\{X_{01}, X_{22}\}$ and $\{X_{02}, X_{12}\}$.
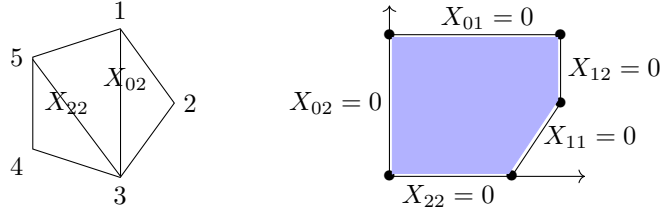
\begin{figure}
    \centering
\begin{tikzpicture}[scale=1.3]
   \newdimen\Ra
   \Ra=0.8cm
   \draw (0:\Ra)
      \foreach \x in {72,144,...,360} {  -- (\x:\Ra) }
               -- cycle (360:\Ra) node[right] {2}
               -- cycle (288:\Ra) node[below] {3}
               -- cycle (216:\Ra) node[below left] {4}
               -- cycle (144:\Ra) node[left] {5}
               -- cycle  (72:\Ra) node[above] {1};
               
    \draw [-] (0.25,0.75) --(0.25,-0.75);
\node at (0.3,0.25) {$X_{02}$};
    \draw [-] (-0.65,0.45) --(0.25,-0.75);
\node at (-0.3,0) {$X_{22}$};

    \draw [->] (3,-0.75) --(5,-0.75);
    \node at (3,-0.75) {$\bullet$};
    \node at (4.25,-0.75) {$\bullet$};
\node at (3.6,-0.95) {$X_{22}=0$};
    \draw [->] (3,-0.75) --(3,1);
    \node at (3,0.7) {$\bullet$};
    \node at (4.75,0.7) {$\bullet$};
\node at (2.45,0) {$X_{02}=0$};
    \draw [-] (4.25,-0.75) --(4.75,0);
\node at (5.05,-0.35) {$X_{11}=0$};
    \draw [-] (4.75,0) --(4.75,0.7);
    \node at (4.75,0) {$\bullet$};
\node at (5.3,0.35) {$X_{12}=0$};
    \draw [-] (3,0.7) --(4.75,0.7);
\node at (4,0.87) {$X_{01}=0$};

\filldraw[blue!30,-] (3.03,-0.73)--(3.03,0.67)--(4.72,0.67)--(4.72,0.03)--(4.22,-0.72)--(3.03,-0.72);

\end{tikzpicture}
\stepcounter{equation}
    \caption{\small{For $n=2$, the kinematic space is $\mathbb{R}^5$. The associahedron is realized inside a $2$-dimensional affine plane determined by a system of inequalities in terms of the mesh relations. The facets correspond to diagonals of the pentagon, and vertices correspond to its triangulations.}}
    \label{fig:ABHY illustration}
\end{figure}		

\end{remark}

		\subsection{The category $\mathcal D$}\label{sec:the category D}
		We now introduce a continuous version of $\Db(A_n)$, which we denote by $\mathcal D$.
		This category is triangulated equivalent to the category $\mathcal D_\pi$ defined by Igusa and Todorov \cite{IT15}.
		Moreover, these categories are isomorphic: they are equivalent and their objects (not just isomorphism classes) are in bijection.
		
		\subsubsection{Objects and morphisms}
		The indecomposable objects of $\mathcal D$ are the points in the set $\R\times (-\frac{\pi}{2}, \frac{\pi}{2})$, and each object in $\mathcal D$ is a finite (possibly empty) direct sum of indecomposable objects.
		
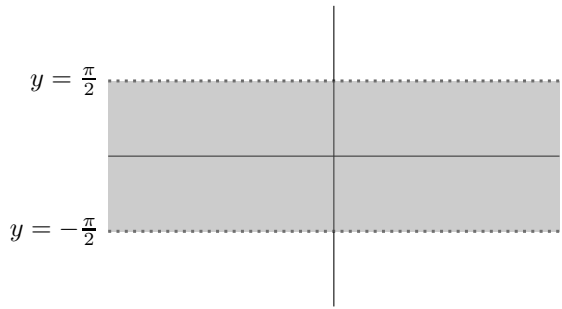
\begin{displaymath}
			\begin{tikzpicture}
				\draw (-3,0) -- (3,0);
				\draw (0,-2) -- (0,2);
				\draw[line width=.4mm, draw opacity = .6, dotted] (-3,-1) -- (3,-1);
				\draw[line width=.4mm, draw opacity = .6, dotted] (-3,1) -- (3,1);
				\filldraw[fill opacity =.4, draw opacity =0, fill=gray] (-3,-1) --  (-3,1) -- (3,1) -- (3,-1) -- (-3,-1);
				\draw (-3,1) node[anchor=east] {$y=\frac{\pi}{2}$};
				\draw (-3,-1) node[anchor=east] {$y=-\frac{\pi}{2}$};
			\end{tikzpicture}	
		\end{displaymath}
		For each point $(x,y)$, define the set $H(x,y)$ in the following way.
		First, consider the rectangle determined by the points
		\begin{displaymath}
			(x,y),\quad
			\left(x +\frac{\pi}{2}-y, \frac{\pi}{2}\right), \quad 
			\left(x+\frac{\pi}{2}+y,-\frac{\pi}{2}\right), \quad
			\text{and} \quad
			(x+\pi,-y).
		\end{displaymath}
		The set $H(x,y)$ is the interior of this rectangle together with the left boundary, but without the points $(x +\frac{\pi}{2}-y, \frac{\pi}{2})$ and 	$(x+\frac{\pi}{2}+y,-\frac{\pi}{2})$.
		\begin{figure}
		\begin{displaymath}
			\begin{tikzpicture}
				\draw (-3,0) -- (3,0);
				\draw (0,-2) -- (0,2);
				\draw[line width=.4mm, draw opacity = .6, dotted] (-3,-1) -- (3,-1);
				\draw[line width=.4mm, draw opacity = .6, dotted] (-3,1) -- (3,1);
				\filldraw[fill opacity =.4, draw opacity =0, fill=gray] (-3,-1) --  (-3,1) -- (3,1) -- (3,-1) -- (-3,-1);
				\draw (-3,1) node[anchor=east] {$y=\frac{\pi}{2}$};
				\draw (-3,-1) node[anchor=east] {$y=-\frac{\pi}{2}$};
				\filldraw (-1,-0.5) circle[radius=.3ex];
				\draw (-1,-0.5) node[anchor=east] {$(x,y)$};
				\draw[thick] (0.5,1) -- (-1,-0.5) -- (-0.5,-1);
				\draw[thick, dotted] (0.5,1) -- (1,0.5) -- (-0.5,-1);
				\draw (-0.5,-1) node[anchor=north] {$(x+\frac{\pi}{2}+y,-\frac{\pi}{2})$};
				\draw (0.5,1) node[anchor=south] {$(x+\frac{\pi}{2}-y,\frac{\pi}{2})$};
				\draw (1,0.5) node[anchor=west] {$(x+\pi,-y)$};
				\filldraw[fill opacity=.2, draw opacity=0] (0.5,1) -- (-1,-0.5) -- (-0.5,-1) -- (1,0.5) -- (0.5,1);
			\end{tikzpicture}	
		\end{displaymath}
		\caption{Morphisms between indecomposable objects in $\mathcal{D}$.}\label{fig:morphs}
		\end{figure}
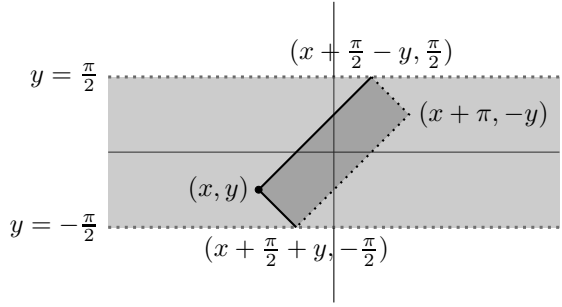
		One may also consider $H(x,y)$ as being defined by the beams emitted from $(x,y)$ with slopes $\pm 1$.  See Figure~\ref{fig:morphs} for an illustration of the set $H(x,y)$.
				
		We define $\Hom$ between indecomposable objects as follows:
		\begin{displaymath}
			\Hom((x,y),(x',y')) = \begin{cases}
		 	k & \text{if } (x',y')\in H(x,y) \\
		 	0 & \text{otherwise}.
		 	\end{cases}
		\end{displaymath}
		When $(x,y)=(x',y')$, the identity morphism on $(x,y)$ is $1\in k$.
		The composition of morphisms $f:(x_1,y_1)\to (x_2,y_2)$ and $g:(x_2,y_2)\to(x_3,y_3)$ is given by multiplication of $f$ and $g$ inside $k$ whenever $\Hom((x_1,y_1),(x_3,y_3))$ is not 0.
		Hom sets and composition of morphisms for arbitrary objects is given by extending the structure bilinearly.
		
		\subsubsection{Triangulated structure}\label{sec:triangulated structure of D}
		Let $(x,y)$ be an indecomposable object in $\mathcal D$.
		We define the shift of $(x,y)$ by
		\begin{displaymath}
			(x,y)[1] := (x+\pi,-y).
		\end{displaymath}
		The shift of a sum of indecomposables is defined to be the sum of the shift of each of the indecomposables.
		
		The minimal distinguished triangles in $\mathcal D$ are of the form
		\begin{displaymath}
			\xymatrix{
				(x,y) \ar[r] & E \ar[r] & (x',y') \ar[r] & (x+\pi,-y),
			}
		\end{displaymath}
	where the case $E=(x,y)\oplus (x',y')$ is called a \textdef{trivial} triangle. 
	
	Among these minimal distinguished triangles, we are particularly interested in the nontrivial ones. This is because, as in Definition~\ref{def:tilting rectangle} and Section~\ref{sec:T-clusters}, we use them to discuss the notions of tilting rectangles (Definition~\ref{def:tilting rectangle}) and compatibility (Definition~\ref{def:compatibility}).
	Moreover, the rest of the distinguished triangles in $\mathcal D$ can be constructed from this collection of minimal distinguished triangles (see \cite{IT15} for details).
		
		The other possibility occurs when $(x',y')$ is an element of $\overline{H(x,y)}\cap \left(\R \times(-\frac{\pi}{2},\frac{\pi}{2})\right)$, where the overline denotes the closure of $H(x,y)$ in $\R^2$ with the standard Euclidean topology.
		In this case, consider the following two points:
		\begin{align*}
			T&=\left( \frac{x+x'-y+y'}{2}, \frac{-x+x'+y'+y}{2} \right) \\
			B&=\left( \frac{x+x'+y-y'}{2}, \frac{\,\,\,\,\,x-x'+y'+y}{2} \right).
		\end{align*}
		If the $y$-coordinate of $T$ is $\frac{\pi}{2}$ then we set $T=0$ in $\mathcal D$, and similarly for $B$ if the $y$-coordinate of $B$ is $-\frac{\pi}{2}$.
		Thus, each of $T$ and $B$ is an indecomposable object or $0$ in $\mathcal D$.
		
		These can be seen in Figure \ref{fig:hom and rectangle} as the top and bottom points of the rectangle whose respective left and right endpoints are $(x,y)$ and $(x',y')$.
		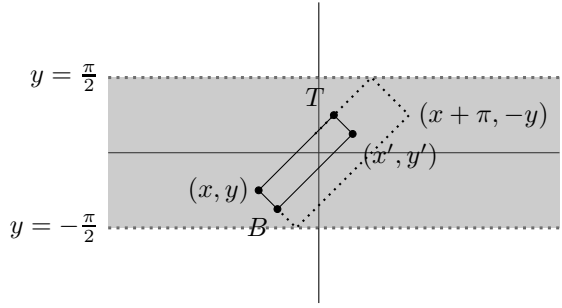
\begin{figure}[h]
		\begin{center}
			\begin{tikzpicture}
				\draw (-3,0) -- (3,0);
				\draw (-.2,-2) -- (-.2,2);
				\draw[line width=.4mm, draw opacity = .6, dotted] (-3,-1) -- (3,-1);
				\draw[line width=.4mm, draw opacity = .6, dotted] (-3,1) -- (3,1);
				\filldraw[fill opacity =.4, draw opacity =0, fill=gray] (-3,-1) --  (-3,1) -- (3,1) -- (3,-1) -- (-3,-1);
				\draw (-3,1) node[anchor=east] {$y=\frac{\pi}{2}$};
				\draw (-3,-1) node[anchor=east] {$y=-\frac{\pi}{2}$};
				\filldraw (-1,-0.5) circle[radius=.3ex];
				\draw (-1,-0.5) node[anchor=east] {$(x,y)$};
				\draw[thick, dotted] (-.25,.25) -- (0.5,1) -- (1,0.5) -- (-0.5,-1) -- (-.75,-.75);
				\draw (1,0.5) node[anchor=west] {$(x+\pi,-y)$};
				\filldraw (.25,.25) circle[radius=.3ex];
				\draw (.25,.25) node[anchor=north west] {$(x',y')$};
				\filldraw (0,.5) circle[radius=.3ex];
				\filldraw (-.75,-.75) circle[radius=.3ex];
				\draw (0,.5) node[anchor=south east] {$T$};
				\draw (-.75,-.75) node[anchor=north east] {$B$};
				\draw (0,.5) -- (.25,.25) -- (-.75,-.75) -- (-1,-.5) -- (0,.5);
			\end{tikzpicture}
			\caption{An example of a tilting rectangle corresponding to the distinguished triangle $(x,y)\to T\oplus B\to (x',y')\to (x+\pi,-y)$.}\label{fig:hom and rectangle}
			\end{center}
		\end{figure}
		If $T=B=0$ in $\mathcal D$ then $(x',y')=(x+\pi,-y)$, and we have the distinguished triangle
		\begin{displaymath}
			\xymatrix{
			(x,y) \ar[r] & 0 \ar[r] & (x+\pi,-y) \ar[r]^-{\cong} & (x+\pi,-y).
			}
		\end{displaymath}
		
		The rectangles we obtain from such distinguished triangles will play a key role in our construction of the continuous associahedron.
		\begin{definition}\label{def:tilting rectangle}
			Let $X=(x,y)$ be a point in the strip $\R\times (-\frac{\pi}{2},\frac{\pi}{2})$.
			Let $a$ and $b$ be positive real numbers such that
			\begin{align*}
				0 < & \, \, a \leq \frac{\pi}{2}-y \\
				0 < & \, \, b \leq y+\frac{\pi}{2}.
			\end{align*}
			Let $Y$, $Z$, and $W$ be the points
			\begin{align*}
				Y &= (x+a, y+a) \\
				Z &= (x+b, y-b) \\
				W &= (x+a+b, y+a-b).	
			\end{align*}
			We call the rectangle $XYWZ$ a \textdef{tilting rectangle}.
		\end{definition}
		We see an example of a tilting rectangle in Figure \ref{fig:hom and rectangle}.
		
		\begin{proposition}\label{prop:rectangles are triangles}
			There is a bijection between tilting rectangles in $\R\times [-\frac{\pi}{2},\frac{\pi}{2}]$ and distinguished triangles in $\mathcal D$ whose first and third terms are both nonzero and indecomposable.
		\end{proposition}
		\begin{proof}
			This follows from Igusa and Todorov's construction \cite[Section 2]{IT15}.
		\end{proof}
		This proposition asserts that the set $\R\times (-\frac{\pi}{2},\frac{\pi}{2})$ acts like the Auslander--Reiten quiver of $\mathcal D$ and $\R\times [-\frac{\pi}{2},\frac{\pi}{2}]$ acts like the augmented Auslander--Reiten quiver of $\mathcal D$.
		In \cite{IT15} the authors show that this, along with a choice of which triangles are distinguished, are sufficient to yield a triangulated structure on $\mathcal D$.
		This choice does not affect our constructions.

		\subsection{Continuous mesh relations}\label{sec:continuous mesh relations}
		In the finite and discrete setting, the mesh relations are induced by the almost split triangles in $\Db(k\mathbb A_n)$.
		Recall that almost split triangles are not trivial.
		However, in $\mathcal D$, the only irreducible morphisms are isomorphisms.
		Thus, we cannot have any almost split triangles.
		Nevertheless, a continuous version of the mesh relations hold instead.
		
		From now on, we use capital letters to denote indecomposable objects in $\mathcal D$.
		\begin{definition}
			Let $\mathbb V$ be an arbitrary real vector space and  $\Phi:\Ind(\mathcal D)\sqcup \{0\}\to\mathbb V$ a function satisfying $\Phi(0)=0$.
			We extend $\Phi$ to all objects of $\mathcal{D}$ by defining its value on a direct sum to be the sum of the values of $\Phi$ on each indecomposable summand.
			
			We say $\Phi$ \textdef{satisfies the continuous mesh relations} provided that for every distinguished triangle
			\begin{displaymath}
				\xymatrix{
					X \ar[r] & Y \ar[r] & Z \ar[r] & X[1],
				}
			\end{displaymath}
			where $X$ and $Z$ are indecomposable, the following equation holds:
			\begin{displaymath}
				\Phi(X) + \Phi(Z) = \Phi(Y).
			\end{displaymath}
		\end{definition}
		
		Two examples of such functions are $g_{\mathcal Z}$-vectors (Section \ref{sec:g-vectors}) and dimension vectors (Section \ref{sec:dimension vectors}).
		We now set the context for an important property of functions satisfying continuous mesh relations, which is stated in Proposition \ref{prop:continuous mesh}.
		Suppose there are nontrivial distinguished triangles
		\begin{displaymath}
			\xymatrix@R=3ex{
				(1) & X \ar[r] & E_1\oplus E_2 \ar[r] & Y \ar[r] & X[1] \\
				(2) & Y \ar[r] & G_1\oplus G_2 \ar[r] & Z \ar[r] & Y[1] \\
				(3) & X \ar[r] & F_1\oplus F_2 \ar[r] & Z \ar[r] & X[1].
			}
		\end{displaymath}
		Without loss of generality, we assume that $E_1$, $F_1$, and $G_1$ are top points of the tilting rectangles in $\R\times(-\frac{\pi}{2},\frac{\pi}{2})$ obtained from the given distinguished triangles in $\mathcal D$.
		
		The relationship between the geometry in $\R\times(-\frac{\pi}{2},\frac{\pi}{2})$ and the distinguished triangles in $\mathcal D$ yields the schematic of the points in $\R\times(-\frac{\pi}{2},\frac{\pi}{2})$ in Figure \ref{fig:schematic}.
		\begin{figure}
		\begin{center}
			\begin{tikzpicture}
				\draw (0,0) -- (1.5,1.5) -- (3.5,-0.5) -- (2,-2) -- (0,0);
				\draw (.75,.75) -- (2.75,-1.25);
				\draw (1,-1) -- (2.5,0.5);
				\filldraw (0, 0) circle[radius=.4ex];
				\filldraw (.75, .75) circle[radius=.4ex];
				\filldraw (1,-1) circle[radius=.4ex];
				\filldraw (1.5, 1.5) circle[radius=.4ex];
				\filldraw (1.75, -0.25) circle[radius=.4ex];
				\filldraw (2,-2) circle[radius=.4ex];
				\filldraw (2.5, 0.5) circle[radius=.4ex];
				\filldraw (2.75,-1.25) circle[radius=.4ex];
				\filldraw (3.5, -0.5) circle[radius=.4ex];
				\draw (0, 0) node[anchor=east] {$X$};
				\draw (.75, .75) node[anchor=east] {$E_1$};
				\draw (1,-1) node[anchor=east] {$E_2$};
				\draw (1.5, 1.5) node[anchor=east] {$F_1$};
				\draw (1.75, -0.25) node[anchor=east] {$Y$};
				\draw (2,-2) node[anchor=east] {$F_2$};
				\draw (2.5, 0.5) node[anchor=west] {$G_1$};
				\draw (2.75,-1.25) node[anchor=west] {$G_2$};
				\draw (3.5, -0.5) node[anchor=west] {$Z$};
			\end{tikzpicture}
			\caption{A schematic of distinguished triangles in $\mathcal D$ and the corresponding rectangles in $\R\times(-\frac{\pi}{2},\frac{\pi}{2})$.}\label{fig:schematic}
		\end{center}
		\end{figure}
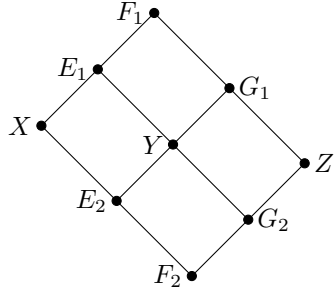
		It can be seen from the schematic that we have two more distinguished triangles:
		\begin{displaymath}
			\xymatrix@R=3ex{
				(4) & E_1 \ar[r] & F_1\oplus Y \ar[r] & G_1 \ar[r] & E_1[1] \\
				(5) & E_2 \ar[r] & Y\oplus F_2 \ar[r] & G_2 \ar[r] & E_2[1].
			}
		\end{displaymath}
		We see that distinguished triangles (1), (3) yield distinguished triangles (2), (4), and (5) both algebraically in $\mathcal D$ and geometrically in $\R\times(-\frac{\pi}{2},\frac{\pi}{2})$.
		
	    The following proposition highlights how the continuous mesh relations behave similarly to the discrete version.
		\begin{proposition}\label{prop:continuous mesh}
		    Assume $\Phi$ satisfies the continuous mesh relations.
			Let $X$, $E_1$, $E_2$, $F_1$, and $F_2$ be indecomposables in $\Ind(\mathcal D)$.
			Let $X\to E_1$ and $X\to E_2$ be morphisms, the slope from $X$ to $E_1$ be $1$, and the slope from $X$ to $E_2$ be $-1$.
			Also let $E_1\to F_1$ and $E_2\to F_2$ be morphisms, the slope from $X$ to $F_1$ be $1$, and the slope from $X$ to $F_2$ be $-1$.
			
			Then one obtains the distinguished triangles (1)--(5) above.
			Furthermore, the values $\Phi(X)$, $\Phi(E_1)$, $\Phi(E_2)$, $\Phi(F_1)$, and $\Phi(F_2)$ determine $\Phi(Y)$, $\Phi(G_1)$, $\Phi(G_2)$, and $\Phi(Z)$.
		\end{proposition}
		\begin{proof}
			The application of the octahedral axiom yields the distinguished triangles.
			We see that
			\begin{displaymath}
				\Phi(Y) = \Phi(E_1)+\Phi(E_2)-\Phi(X).
			\end{displaymath}
			Then we may replace $\Phi(Y)$ in the equations for $\Phi(G_1)$ and $\Phi(G_2)$.
			Finally we replace $\Phi(Y)$, $\Phi(G_1)$, and $\Phi(G_2)$ in the equation for $\Phi(Z)$.
		\end{proof}
		Note that, in the proposition above, if any of $E_1$, $E_2$, $F_1$, and $F_2$ are the $0$ object then we may choose points on the boundary of $\R\times[-\frac{\pi}{2},\frac{\pi}{2}]$ such that the assertions hold.
		
	\section{Continuous deformed mesh relations}\label{Sec:Continuous Deformed Mesh Relations}
	The goal of this section is to introduce the continuous version of the deformed mesh relations~\eqref{defmesh} from Section \ref{sec:Amplituhedron for Dynkin Quivers}. For that we first develop a technique that we call \textdef{quilting}, which does not depend on knitting.
		\subsection{Quilting}\label{subsec:Quilting}

		\subsubsection{Patches and quilting}
		Notice that in Section \ref{sec:Amplituhedron for Dynkin Quivers} the projective slice forms a zigzag shape.
		To introduce a continuous version of the projective slice, we begin with the following definition.
		
		\begin{definition}\label{def:zigzag}\label{def:left vertex}\label{def:right vertex}
            Let $L=\{\ell_1,\ldots,\ell_n\}$ be a set of distinct line segments in $\R\times[-\frac{\pi}{2},\frac{\pi}{2}]$ which satisfy the following conditions.
			\begin{itemize}
				\item The slope of $\ell_1$ is $\pm 1$, and its top point is on the line $y=\frac{\pi}{2}$.
				\item For $1\leq i < n$, the segment $\ell_{i+1}$ is below $\ell_{i}$ and they intersect at a right angle.
				\item The bottom point of $\ell_n$ lies on the line $y=-\frac{\pi}{2}$.
			\end{itemize}
			Then we set
			\begin{displaymath}
				\mathcal Z = \left(\bigcup_{i=1}^n\ell_i\right) \cap \left( \R\times\left(-\frac{\pi}{2},\frac{\pi}{2}\right)\right),
			\end{displaymath}
			and call it a \textdef{zigzag with $n$ line segments}. We consider $\mathcal Z$ as a full subcategory of $\mathcal D$. 
			
			For each $1\leq i<n$, the second condition implies $\ell_{i+1}$ and $\ell_{i}$ have exactly one point $P_i=(a_i,b_i)$ in common and $y<b_i$, for any other point $(x,y)$ on $\ell_{i+1}$. If the slope of $\ell_i$ is $1$ (respectively, $-1$), this unique point $P_i$ is called a \textdef{left vertex} (respectively, \textdef{right vertex}) of $\mathcal{Z}$.
		\end{definition}
		
		In Definition \ref{def:zigzag}, the top point of $\ell_1$ and the bottom point of $\ell_n$ do not belong to $\mathcal{Z}$, and all points in $\mathcal Z$ are indecomposable objects in the category $\mathcal D$.
		
		Next we define the two types of patches.
		We will save visual depictions until we define quilting.
		
		\begin{definition}\label{def:patch}
			Let $\mathcal Z$ be a zigzag with $n$ line segments.
			\begin{itemize}
			\item Let $\ell_i$ and $\ell_{i+1}$ be line segments of $\mathcal Z$ that share a left vertex.
				Let the top point of $\ell_i$ be $(y_1,y_2)=Y$ and the bottom point of $\ell_{i+1}$ be $(z_1,z_2)=Z$.
				Let the shared left vertex be $(x,y)=X$, and let
		    	\begin{displaymath}
					W= (y_1+z_1-x_1,y_2+z_2-x_2).
				\end{displaymath}
				Then $XYWZ$ is a rectangle in $\R\times[-\frac{\pi}{2},\frac{\pi}{2}]$ whose sides have slopes $\pm 1$.
				By a \textdef{rectangular patch}, we mean the part of $XYWZ$ in $\R\times(-\frac{\pi}{2},\frac{\pi}{2})$ together with the region it bounds.
				
			\item If $\ell_i$ is a line segment that does not share a left vertex with another line segment, then $i=1$ or $i=n$.
				Let $Y$ be the right vertex on $\ell_i$ and $X$ the point in $\R\times\{-\frac{\pi}{2},\frac{\pi}{2}\}$ on $\ell_i$.
				We reflect $\ell_i$ about the vertical line through $Y$. Then there is a point $Z$ in both $\R\times\{-\frac{\pi}{2},\frac{\pi}{2}\}$ and the reflection of $\ell_i$.
				This defines a triangle $XYZ$ in $\R\times[-\frac{\pi}{2},\frac{\pi}{2}]$.
				By a \textdef{triangular patch}, we mean  the part of $XYZ$ in $\R\times(-\frac{\pi}{2},\frac{\pi}{2})$ together with the region it bounds.
			\end{itemize}
			Henceforth, by a \textdef{patch} we mean a rectangular patch or a triangular patch.
		\end{definition}
		
		Note that patches are always to the right of the corresponding zigzag.
		
		\begin{definition}[Quilting]\label{def:quilt}\label{def:quilt of zigzag}
		For a zigzag $\mathcal{Z}$, we define \emph{quilting} $\quilt(\mathcal Z)$ of $\mathcal{Z}$ via a process on all patches induced by $\mathcal{Z}$.
		
		First, let $XYWZ$ be a rectangular patch. By $\quilt(\ell_{i+1})$ we denote the top right side of $XYWZ$, which is a translation of $\ell_{i+1}$.
	    Furthermore, by $\quilt(\ell_i)$ we denote the bottom right side of $XYWZ$, which is a translation of $\ell_i$.
			Note that left vertices of $\quilt(\ell_i)\cup\quilt(\ell_{i+1})$ are the right vertices of $\ell_i\cup\ell_{i+1}$.
			Also note that $\quilt(\ell_{i+1})$ is now above $\quilt(\ell_i)$.
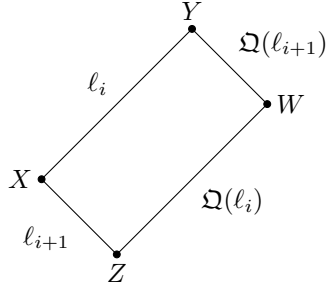
			\begin{displaymath}
				\begin{tikzpicture}
					\draw (0,1) -- (1,0) -- (3,2) -- (2,3) -- (0,1);
					\draw (.5,.5) node[anchor=north east] {$\ell_{i+1}$};
					\draw (1,2) node[anchor=south east] {$\ell_i$};
					\draw (2.5,2.5) node[anchor=south west] {$\quilt(\ell_{i+1})$};
					\draw (2,1) node[anchor=north west] {$\quilt(\ell_i)$};
					\draw (0,1) node[anchor=east] {$X$};
					\draw (1,0) node[anchor=north] {$Z$};
					\draw (2,3) node[anchor=south] {$Y$};
					\draw (3,2) node[anchor=west] {$W$};
					\filldraw (0,1) circle[radius=.3ex];
					\filldraw (1,0) circle[radius=.3ex];
					\filldraw (3,2) circle[radius=.3ex];
					\filldraw (2,3) circle[radius=.3ex];
				\end{tikzpicture}
			\end{displaymath}
			
			Now let $XYZ$ be a triangular patch.
			We define $\quilt(\ell_i)$ to be the reflection of $\ell_i$ across the vertical line through $Y$.
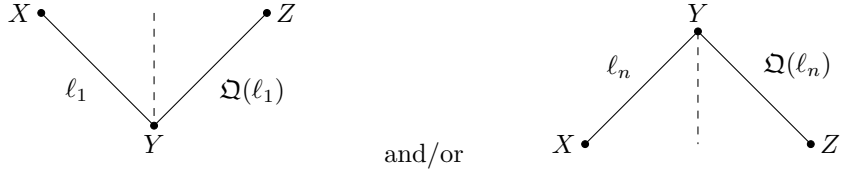
			\begin{displaymath}
				\begin{tikzpicture}
					\draw (0,1.5) -- (1.5,0) -- (3,1.5);
					\draw[dashed] (1.5,1.5) -- (1.5,0);
					\draw (.75,.75) node[anchor=north east] {$\ell_1$};
					\draw (2.25,.75) node[anchor=north west] {$\quilt (\ell_1)$};
					\draw (0,1.5) node[anchor=east] {$X$};
					\draw (1.5,0) node[anchor=north] {$Y$};
					\draw (3,1.5) node[anchor=west] {$Z$};
					\filldraw (0,1.5) circle[radius=.3ex];
					\filldraw (1.5,0) circle[radius=.3ex];
					\filldraw (3,1.5) circle[radius=.3ex];
				\end{tikzpicture}
				\qquad
				\quad
				\text{and/or}
				\quad
				\qquad
				\begin{tikzpicture}
					\draw (0,0) -- (1.5,1.5) -- (3,0);
					\draw[dashed] (1.5,1.5) -- (1.5,0);
					\draw (.75,.75) node[anchor=south east] {$\ell_n$};
					\draw (2.25,.75) node[anchor=south west] {$\quilt (\ell_n)$};
					\draw (0,0) node[anchor=east] {$X$};
					\draw (1.5,1.5) node[anchor=south] {$Y$};
					\draw (3,0) node[anchor=west] {$Z$};
					\filldraw (0,0) circle[radius=.3ex];
					\filldraw (1.5,1.5) circle[radius=.3ex];
					\filldraw (3,0) circle[radius=.3ex];
				\end{tikzpicture}
			\end{displaymath}
			We define $\quilt(\mathcal Z) := \left(\bigcup_i \quilt(\ell_i)\right) \cap \left(\mathbb R \times (-\frac{\pi}{2},\frac{\pi}{2})\right)$.
			Notice again that the left vertex of $\quilt(\ell_i)$ is the right vertex of $\ell_i$.
			In the top-to-bottom ordering of the line segments of $\quilt(\mathcal Z)$, we consider these reflections to be fixed when compared to $\mathcal Z$.
		\end{definition}
		
		\begin{remark}
			It follows directly from the definition that $\quilt(\mathcal Z)$ is unique.
			Moreover, in $\R\times[-\frac{\pi}{2},\frac{\pi}{2}]$, a rectangular patch is a tilting rectangle.
		\end{remark}
		
		\subsubsection{Requisite combinatorics}
		In this section, we collect the combinatorial tools needed for our proofs about quilting and continuous deformed mesh relations.
		We treat permutations as functions, and therefore if $\alpha$ and $\beta$ are permutations of a set $X$, by $\beta\circ \alpha$ we denote the permutation obtained by first performing $\alpha$ and then performing $\beta$. 
		If $\nset$ is an ordered set with $n$ elements, say $\{x_1,\ldots,x_n\}$, by $s_i$ we denote the simple permutation which swaps the position of $x_i$ and $x_{i+1}$. In particular, $s_i(\nset)=\{x_1,\ldots,x_{i-1},\,\,\, x_{i+1},x_i,\,\,\, x_{i+2}, \ldots, x_n\}$.
		
		For any $m \in \mathbb{R}$, let $\lceil m \rceil:= \min \{a \in \mathbb{Z} \mid m \leq a \}$ and
		$\lfloor m \rfloor:= \max \{b \in \mathbb{Z} \mid b \leq m \}$.
		
		\begin{definition}\label{def:permutations}
		We consider two particular permutations on the ordered set $\nset$:
			\begin{enumerate}
				\item Let $\pi_1:= s_{2\lfloor \frac{n}{2} \rfloor-1}\circ \cdots \circ s_3 \circ s_1$.
				\item Let $\pi_2:= s_{2\lceil \frac{n-1}{2} \rceil} \circ \cdots \circ s_4 \circ s_2$.
			\end{enumerate}
		\end{definition}
		
		One may think of $\pi_1$ as a permutation that moves elements in odd positions forward and elements in even positions backward, where the indices permit.
		
		\begin{example}\label{xmp:permutations}
			Consider the sets $\{A,B,C,D,E\}$ and $\{A,B,C,D,E,F\}$.
			We can visualize $\pi_2\circ \pi_1$ on the left and right, respectively. 
			
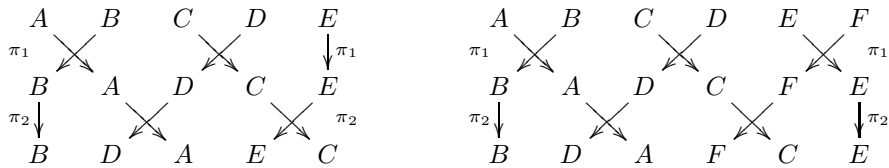
\begin{displaymath}
				\xymatrix@R=3ex@C=3ex{
					A \ar@{}[d]_-{\pi_1} \ar[dr] & B \ar[dl] & C \ar[dr] & D\ar[dl] & E\ar[d]^-{\pi_1} 
					& & & A \ar@{}[d]_-{\pi_1} \ar[dr] & B\ar[dl] & C \ar[dr]& D\ar[dl] & E\ar[dr] & F\ar[dl] \ar@{}[d]^-{\pi_1} \\
					B\ar[d]_-{\pi_2} & A\ar[dr] & D\ar[dl] & C\ar[dr] & E\ar[dl] \ar@{}[d]^-{\pi_2}
					& & & B \ar[d]_-{\pi_2} & A \ar[dr]& D \ar[dl] & C\ar[dr] & F\ar[dl] & E\ar[d]^-{\pi_2} \\
					B & D & A & E & C & & & B & D & A & F & C & E\\
				}
			\end{displaymath}
			
		Observe that the permutations behave slightly differently depending on whether or not we are working with an even number of elements or an odd number of elements. For a permutation $\alpha$, by $\Fix(\alpha)$ we denote the set of elements fixed under $\alpha$.
		\end{example}
		
		\begin{proposition}\label{prop:fixed elements}
		For $\nset=\{x_1,\ldots, x_n\}$, we have the following properties.
			\begin{itemize}
				\item Both $\pi_1$ and $\pi_2$ are involutions on $\nset$.
				\item If $n$ is odd, then $\Fix(\pi_1)=\{x_n\}$ and $\Fix(\pi_2)=\{x_1\}$.
				\item If $n$ is even, then $\Fix(\pi_1)=\emptyset$ and $\Fix(\pi_2)=\{x_1, x_n\}$.
				\item $\Fix(\pi_1) \cap \Fix(\pi_2)= \emptyset $.
			\end{itemize}
		\end{proposition}
		\begin{proof}
			The statements follow directly from Definition \ref{def:permutations}.
		\end{proof}
		
		\begin{definition}
			Let $\nset$ be an ordered set with $n$ elements.
			For an even positive integer $j$, we define
			\begin{itemize}
				\item the \textdef{odd $j$th composition} to be
						$\pi^j_o:= (\pi_2\circ \pi_1)^{\frac{j}{2}}$,
				\item and the \textdef{even $j$th composition} to be
						$\pi^j_e:= (\pi_1\circ\pi_2)^{\frac{j}{2}}$.
			\end{itemize}
			For an odd positive integer $j$, we define
			\begin{itemize}
				\item the \textdef{odd $j$th composition} to be
						$\pi^j_o:=\pi_1\circ(\pi_2\circ \pi_1)^{\frac{j-1}{2}}$,
				\item and the \textdef{even $j$th composition} to be
						$\pi^j_e:=\pi_2\circ(\pi_1\circ\pi_2)^{\frac{j-1}{2}}$.
			\end{itemize}
		\end{definition}
		The naming convention above was chosen so that the odd compositions start with $\pi_1$ and the even compositions start with $\pi_2$.
		In essence, $j$ counts the number of alternating compositions of $\pi_1$ and $\pi_2$ that we are performing.

		\begin{proposition}\label{prop:forwards}
			Let $\nset=\{x_1,\ldots, x_n\}$. For $1\leq i < n$ and $j=n-i$, we have
			\begin{itemize}
				\item if $i$ is odd, then $\pi_o^j (x_i)=x_n$ and 
				$\pi_o^{j+1}(x_i)=x_n$.
				\item if $i$ is even, then $\pi_e^j(x_i)=x_n$ and
				$\pi_e^{j+1}(x_i)=x_n$.
			\end{itemize}
		\end{proposition}
		\begin{proof}
			The proof is by induction on $j$.
			First suppose $j=1$; we use Proposition \ref{prop:fixed elements}.
			If $i$ is odd, then $n$ is even.
			So $\pi_1(x_i)=x_n$ and $\pi_2(x_n)=x_n$.
			If instead $i$ is even, then $n$ is odd.
			In this case, $\pi_2(x_i)=x_n$ and $\pi_1(x_n)=x_n$.
		
			Assume that the assertion holds for all positive integers less than or equal to $\ell$, and let $j=\ell+1$.
			Suppose $i$ is odd. Then $\pi_1(x_i)=x_{i+1}$.
			We know $n-(i+1)=\ell$ and the proposition holds for $\ell$.
		
			If we perform $\pi_e^\ell\circ \pi_1$ then we have performed $\pi_o^j$, which sends $x_i$ to $x_n$.
			If we perform $\pi_e^{\ell+1}\circ \pi_1$ we have performed $\pi_o^{j+1}$, which sends $x_i$ to $x_n$.
			If $i$ is even, we can perform a similar argument.
		\end{proof}
		
		\begin{proposition}\label{prop:backwards}
	    Let $\nset=\{x_1,\ldots, x_n\}$. For $1< i \leq n$ and $j=i-1$, we have
			\begin{itemize}
				\item if $i$ is odd, then $\pi_e^j(x_i)=x_1$ and $\pi_e^{j+1}(x_i)=x_1$.
				\item if $i$ is even, then $\pi_o^j(x_i)=x_1$ and $\pi_o^{j+1}(x_i)=x_1$.
			\end{itemize}
		\end{proposition}
		\begin{proof}
			The statement and the proof are symmetric to Proposition \ref{prop:forwards} and its proof.	
		\end{proof}
		
		\begin{lemma}\label{lem:reverse}\label{lem:fixed once}\label{lem:combinatorial directions}
			Let $\nset$ be an ordered set with $n$ elements.
			The following hold.
			\begin{itemize}
				\item $\pi_o^n$ and $\pi_e^n$ both reverse the order of the element in $\nset$.
				\item
				If  $x\in\nset$, there exist unique $0\leq i, j< n$ such that $\pi_o^i(x)=\pi_o^{i+1}(x)$ and $\pi_e^j(x)=\pi_e^{j+1}(x)$.
			\end{itemize}
		\end{lemma}
		\begin{proof}
		We prove the statement by showing that the $i$th element is sent to the $((n+1)-i)$th element.
		First suppose $i$ is odd.
		Then, by Proposition \ref{prop:forwards}, both $\pi_o^{n-i}$ and $\pi_{o}^{n+1-i}$ send $i$ to the $n$th element.
			
			If $n$ is odd $\pi_e^{i-1}$ sends the $n$th element to the $((n+1)-i)$th element.
			In this case, $((n+1)-i)$ is odd; so $\pi_e^{i-1}\circ\pi_o^{n+1-i}=\pi_o^n$.
				
			If $n$ is even then $\pi_o^{i-1}$ sends the $n$th element to the $((n+1)-i)$th element.
			In this case, $((n+1)-i)$ is even; so $\pi_o^{i-1}\circ \pi_o^{n+1-i}=\pi_o^n$.
			The case when $i$ is odd and we start with the $((n+1)-i)$th even composition is similar.
			
			Now suppose $i$ is odd and we first perform $\pi_e^i$.
			By Proposition \ref{prop:backwards}, $\pi_e^i$ and $\pi_e^{i-1}$ send the $i$th element to the first position.
			Thus, $\pi_e^i=\pi_2\circ \pi_e^{i-1}$.
			We know the odd $(n-i)$th composition will send the first element to the $((n+1)-i)$th position.
			We also see that $\pi_o^{n-1}\circ\pi_e^i = \pi_e^n$.
			The case when $i$ is even is similar.
			This concludes the proof.
		\end{proof}
		
		\subsubsection{Quilting}
		For a zigzag $\mathcal Z$ with $n$ line segments, we show in Theorem~\ref{thm:n quilt of a zigzag is the shift} below that performing quilting $n$ times on $\mathcal Z$ yields $\mathcal Z[1]$.
		
		\begin{proposition}\label{prop:lines after quilting}
		Let $\mathcal Z$ be a zigzag with the set of line segments $L:=\{\ell_1, \ldots, \ell_n\}$.
			Then $\quilt(\mathcal Z)$ is a zigzag, and the order on $\{\quilt(\ell_i)\mid \ell_i\in L\}$ is $\pi_1(L)$ if the slope of $\ell_1$ is $+1$, and $\pi_2(L)$ otherwise.
		\end{proposition}
		\begin{proof}
			By Definition \ref{def:quilt of zigzag}, $\quilt(\mathcal Z)$ is indeed a zigzag: all line segments have alternating slopes $\pm 1$ and the top and bottom points respectively belong to the lines $y=\frac{\pi}{2}$ and $y=-\frac{\pi}{2}$. To show the assertion, we treat two cases based on the slope of $\ell_1$.
			
			We consider the case where the slope of $\ell_1$ is $+1$
			(the case where the slope of $\ell_1$ is $-1$ is similar).
		    In this case, the slope of all odd line segments are $+1$.
		    Hence, the slope of all even line segments are $-1$.
			Thus, for each pair $\ell_i$ and $\ell_{i+1}$ that share a left vertex, $i$ must be odd.
			So, $\quilt(\ell_{i+1})$ is above $\quilt(\ell_i)$.
			If $n>1$ is odd, then, by Definition \ref{def:quilt of zigzag}, $\quilt(\ell_n)$ is a reflection.
			If $n$ is even there are no reflections.
			Thus, the top-to-bottom order of the line segments of $\quilt(\mathcal Z)$ is given by $\pi_1(L)$.
		\end{proof}
		
		\begin{proposition}\label{prop:inverse quilt}
			For a zigzag $\mathcal Z$, there exists a unique zigzag $\mathcal Z'$ such that $\quilt(\mathcal Z')=\mathcal Z$.
		\end{proposition}
		\begin{proof}
			One may check that the symmetric construction to Definition \ref{def:quilt of zigzag} also yields a zigzag $\mathcal Z'$.
			It follows directly that $\quilt(\mathcal Z')=\mathcal Z$ and is unique.
		\end{proof}

		Propositions \ref{prop:lines after quilting} and \ref{prop:inverse quilt} justify the following definition and notation.
		
		\begin{notation}\label{note:inverse quilt}
	    Let $\mathcal Z$ be a zigzag.
		By $\quilt^{-1}(\mathcal Z)$ we denote the zigzag $\mathcal Z'$ such that $\quilt(\mathcal Z')=\mathcal Z$. Furthermore, for $i\in\Z$, we set
			\begin{itemize}
				\item  $\quilt^i(\mathcal Z)=\overbrace{ \quilt ( \quilt( \cdots (\quilt}^{i \text{ times}} (\mathcal Z)) \cdots ) )$ if $i>0$,
				\item $\quilt^i(\mathcal Z)=\mathcal Z$ if $i=0$, and	
					
				\item $\quilt^i(\mathcal Z)=\underbrace{ \quilt^{-1} ( \quilt^{-1}( \cdots (\quilt^{-1}}_{i \text{ times}} (\mathcal Z)) \cdots ) )$ if $i<0$. 
			\end{itemize}
		\end{notation}
		
		\begin{lemma}\label{lem:n-quilt of zigzag}
		Let $\mathcal Z$ be a zigzag and $L$ the set of its $n$ line segments. 
		The line segments of the zigzag $\quilt^n (\mathcal Z)$ are the translations of the reflections of those in $\mathcal Z$, in the reverse order from top-to-bottom.	
		\end{lemma}
		\begin{proof}
			Notice that for the line segments $\ell_1$ and $\ell_n$ of a zigzag, reflection implies $\quilt(\ell_1)$ and $\quilt(\ell_n)$ are respectively the first and last line segments of $\quilt(\mathcal Z)$.
			If $\ell_1$ in $\mathcal Z$ has slope $+1$ then $\quilt(\ell_2)$, which is the top of $\quilt(\mathcal Z)$, has slope $-1$.
			If $\ell_1$ in $\mathcal Z$ has slope $-1$ then $\quilt(\ell_1)$ has slope $+1$.
			Thus, the order of the line segments in $\quilt^n(\mathcal Z)$ is given either by $\pi_o^n(L)$ or $\pi_e^n(L)$.
			In either case, by Lemma \ref{lem:reverse}  the line segments of $\quilt(\mathcal Z)$ are as described.	
		\end{proof}
		
		We now prove the main result of the section.
		In Figure \ref{fig:quilt}, we see an example of Theorem \ref{thm:n quilt of a zigzag is the shift}.
		\begin{figure}
\begin{center}
	\begin{tikzpicture}[scale=.8]
	\draw[white!75!black, line width=.4mm] (-2,0) -- (12,0);
	\draw[white!75!black, line width=.4mm] (0,-6) -- (0, 6);
	\draw[dotted, line width=.4mm] (-2,5) -- (12,5);
	\draw[dotted, line width=.4mm] (-2,-5) -- (12,-5);
	
	\draw[line width=.4mm] (0,5) -- (1,4) -- (-.25,2.75) -- (1.25,1.25) -- (-.5,-.5) -- (1.5,-2.5) -- (-1,-5);
	
	\draw[line width=.4mm, red] (2,5) -- (1,4) -- (2.5,2.5) -- (1.25,1.25) -- (3.25, -.75) -- (1.5,-2.5) -- (4,-5);
	
	\draw[line width=.4mm, blue] (2,5) -- (3.5,3.5) -- (2.5,2.5) -- (4.5,0.5) -- (3.25,-.75) -- (5.75, -3.25) -- (4,-5);
	
	\draw[line width=.4mm, red] (5,5) -- (3.5,3.5) -- (5.5,1.5) -- (4.5,0.5) -- (7,-2) -- (5.75,-3.25) -- (7.5,-5);
	
	\draw[line width=.4mm, blue] (5,5) -- (7,3) -- (5.5,1.5) -- (8,-1) -- (7,-2) -- (8.75,-3.75) -- (7.5,-5);
	
	\draw[line width=.4mm, red] (9,5) -- (7,3) -- (9.5,.5) -- (8,-1) -- (9.75,-2.75) -- (8.75,-3.75) -- (10,-5);
	
	\draw[line width=.4mm, blue] (9,5) -- (11.5,2.5) -- (9.5,.5) -- (11.25,-1.25) -- (9.75,-2.75) -- (11,-4) -- (10,-5);
	
	\draw (.5,4.5) node[anchor=east] {1};
	\draw (1.5,4.5) node[anchor=east] {\textcolor{red}{r1}};
	\draw (3,3) node[anchor=east] {\textcolor{blue}{r1}};
	\draw (5,1) node[anchor=east] {\textcolor{red}{r1}};
	\draw (7.5,-1.5) node[anchor=east] {\textcolor{blue}{r1}};
	\draw (9.25,-3.25) node[anchor=east] {\textcolor{red}{r1}};
	\draw (10.5,-4.5) node[anchor=east] {\textcolor{blue}{r1}};
	
	\draw (.325,3.375) node[anchor=east] {2};
	\draw (1.825, 1.875) node[anchor=east] {\textcolor{red}{2}};
	\draw (3.825, -0.125) node[anchor=east] {\textcolor{blue}{2}};
	\draw (6.325, -2.625) node[anchor=east] {\textcolor{red}{2}};
	\draw (8.075, -4.375) node[anchor=east] {\textcolor{blue}{2}};
	\draw (9.325, -4.375) node[anchor=east] {\textcolor{red}{r2}};
	\draw (10.325, -3.375) node[anchor=east] {\textcolor{blue}{r2}};
	
	\draw (0.5,2) node[anchor=east] {3};
	\draw (1.75, 3.25) node[anchor=east] {\textcolor{red}{3}};
	\draw (2.75, 4.25) node[anchor=east] {\textcolor{blue}{3}};
	\draw (4.25, 4.25) node[anchor=east] {\textcolor{red}{r3}};
	\draw (6.25, 2.25) node[anchor=east] {\textcolor{blue}{r3}};
	\draw (8.75, -0.25) node[anchor=east] {\textcolor{red}{r3}};
	\draw (10.5, -2) node[anchor=east] {\textcolor{blue}{r3}};
	
	\draw (.375,.375) node[anchor=east] {4};
	\draw (2.375, -1.625) node[anchor=east] {\textcolor{red}{4}};
	\draw (4.875, -4.125) node[anchor=east] {\textcolor{blue}{4}};
	\draw (6.625, -4.125) node[anchor=east] {\textcolor{red}{r4}};
	\draw (7.875, -2.875) node[anchor=east] {\textcolor{blue}{r4}};
	\draw (8.875, -1.875) node[anchor=east] {\textcolor{red}{r4}};
	\draw (10.375, -0.375) node[anchor=east] {\textcolor{blue}{r4}};
	
	\draw (.5,-1.5) node[anchor=east] {5};
	\draw (2.25, 0.25) node[anchor=east] {\textcolor{red}{5}};
	\draw (3.5, 1.5) node[anchor=east] {\textcolor{blue}{5}};
	\draw (4.5, 2.5) node[anchor=east] {\textcolor{red}{5}};
	\draw (6, 4) node[anchor=east] {\textcolor{blue}{5}};
	\draw (8,4) node[anchor=east] {\textcolor{red}{r5}};
	\draw (10.5, 1.5) node[anchor=east] {\textcolor{blue}{r5}};
	
	\draw (.25 ,-3.75) node[anchor=east] {6};
	\draw (2.75 ,-3.75) node[anchor=east] {\textcolor{red}{r6}};
	\draw (4.5, -2) node[anchor=east] {\textcolor{blue}{r6}};
	\draw (5.75, -0.75) node[anchor=east] {\textcolor{red}{r6}};
	\draw (6.75, 0.25) node[anchor=east] {\textcolor{blue}{r6}};
	\draw (8.25, 1.75) node[anchor=east] {\textcolor{red}{r6}};
	\draw (10.25, 3.75) node[anchor=east] {\textcolor{blue}{r6}};
	
	\draw (.25,-5) node[anchor=north] {$\mathcal Z$};
	\draw (2.75,-5) node[anchor=north] {\textcolor{red}{$\mathfrak Q(\mathcal Z)$}};
	\draw (4.875,-5) node[anchor=north] {\textcolor{blue}{$\mathfrak Q^2(\mathcal Z)$}};
	\draw (6.625,-5) node[anchor=north] {\textcolor{red}{$\mathfrak Q^3(\mathcal Z)$}};
	\draw (8.075,-5) node[anchor=north] {\textcolor{blue}{$\mathfrak Q^4(\mathcal Z)$}};
	\draw (9.325,-5) node[anchor=north] {\textcolor{red}{$\mathfrak Q^5(\mathcal Z)$}};
	\draw (10.75,-5) node[anchor=north] {\textcolor{blue}{$\mathfrak Q^6(\mathcal Z)$}};
	\end{tikzpicture}
	\caption{
	On a zigzag $\mathcal Z$ with $6$ line segments, performing quilting $6$ times results in $\mathcal{Z}[1]$.
	Here `r' indicates the reflection of the line segment as in Definition \ref{def:quilt of zigzag}.}
	\label{fig:quilt}
\end{center}
\end{figure}
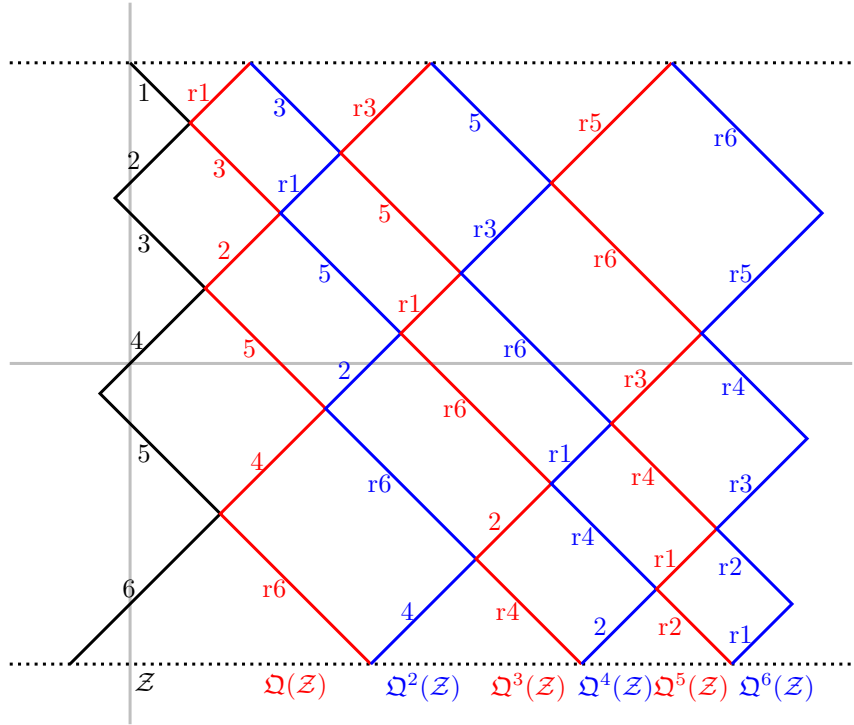
		
		\begin{theorem}\label{thm:n quilt of a zigzag is the shift}
		Let $\mathcal{Z}$ be a zigzag in $\mathcal D$ with $n$ line segments. Then, $\quilt^n (\mathcal Z) = \mathcal Z[1]$ and $\quilt^{-n} (\mathcal Z)=\mathcal Z[-1]$.
		\end{theorem}
		\begin{proof}
				If $\mathcal Z$ has one line segment the theorem is immediate.
				Assume $\mathcal Z$ has at least two line segments.
				By Lemma \ref{lem:n-quilt of zigzag}, $\quilt^n(\mathcal Z)$ and $\mathcal Z[1]$ have the same shape; i.e., they are translations of each other.
				Similarly, $\quilt^{-n}(\mathcal Z)$ is a translation of $\mathcal Z[-1]$.
				
				We show the top points of $\quilt^n(\mathcal Z)$ and $\mathcal Z[1]$ are the same, which implies that the translation is the identity and $\quilt^n(\mathcal Z)$ and $\mathcal Z[1]$ are the same.
				The symmetric argument shows $\quilt^{-n}(\mathcal Z)=\mathcal Z[-1]$.
				
				We assume the slope of $\ell_1$ is $-1$, as the proof when $\ell_1$ has slope $+1$ is similar.
				
				Let $m$ be the largest odd number inclusively between $1$ and $n$.
				We first show that the top point of $\quilt^n(\mathcal Z)$ is the same as the top point of $\quilt^m(\ell_m)$.
				Then, by combining Propositions \ref{prop:forwards} and \ref{prop:backwards} with Lemma \ref{lem:fixed once}, we conclude that the top line segment of each zigzag obtained from $\mathcal Z$ is among
				\begin{displaymath}
					\quilt^0(\ell_1), \quilt(\ell_1), \quilt^2(\ell_3), \quilt^3(\ell_3),\ldots, \quilt^{i-1}(\ell_i),\quilt^i (\ell_i),\ldots 
				\end{displaymath}
				where $i$ is odd.
				
				By Lemma~\ref{lem:n-quilt of zigzag}, note that the top line segment of $\quilt^n(\mathcal Z)$ is $\quilt^n(\ell_n)$.
				If $n$ is odd this line segment has slope $+1$ since it was reflected exactly once.
				If $n$ is even this line segment has slope $-1$ and so shares its top point with $\quilt^{n-1}(\ell_{n-1})$.
				Thus, the top point of $\quilt^n(\mathcal Z)$ is the top point of $\quilt^m(\ell_m)$.
				We have also shown that the top point of $\quilt^i(\mathcal Z)$ is the top point of $\quilt^{i+1}(\mathcal Z)$ when $i<n$ is odd.
				
				For each line segment $\ell_i$, let $x_{b_i}$ and $x_{t_i}$ be respectively the $x$-coordinates of the bottom point and top point of $\ell_i$. Let $h_i=|x_{t_i}-x_{b_i}|$.
				Now we show that the distance between the top point of $\mathcal Z$ and the top point of $\quilt^n(\mathcal Z)$ is 
				\begin{displaymath}
					2 \left(\sum_{\text{odd } 1\leq i \leq n} h_i\right).
				\end{displaymath}
				We see the distance between the top point of $\mathcal Z$ and the top point of $\quilt^1(\mathcal Z)$ is $2h_1$. The distance between the top point of $\quilt^2 (\mathcal Z)$ and $\quilt^3(\mathcal Z)$ is $2h_3$, and so on.
				
				Let $x_b$ (respectively, $x_t$) denote the $x$-coordinate of the bottom point (respectively, the top point) of $\mathcal Z$. Observe that 
				\begin{displaymath}
					x_b - x_t = \left(\sum_{\text{odd } 1\leq i \leq n} h_i\right) - \left( \sum_{\text{even } 1\leq i \leq n} h_i \right).
				\end{displaymath}
				Denote by $x_q$ the $x$-coordinate of the top point of $\quilt^n(\mathcal Z)$.
				From the previous equation it follows that
				\begin{align*}
					x_q - x_b &=
					2\left(\sum_{\text{odd } 1\leq i \leq n} h_i\right) - \left(\sum_{\text{odd } 1\leq i \leq n} h_i\right) + \left( \sum_{\text{even } 1\leq i \leq n} h_i \right) \\
					&= \sum_{i=1}^n h_i=\pi.
				\end{align*}
				This implies that the top point of $\quilt^n(\mathcal Z)$ is $(x_b+\pi,\frac{\pi}{2})$, which is the top point of $\mathcal Z[1]$.
		\end{proof}
		
		\subsection{Deformed mesh relations}
		In this section, we define the continuous analog of the deformed mesh relations \eqref{defmesh} in Section \ref{sec:Amplituhedron for Dynkin Quivers}.
		
		\subsubsection{Permissible functions}\label{sec:permissible functions}
		We define the continuous analog of the $\intc$ values in Section \ref{sec:Amplituhedron for Dynkin Quivers}.
		In practice we want to be able to integrate such a $\intc$ function over any tilting rectangle (Definition \ref{def:tilting rectangle}) in the closed strip $\R\times[-\frac{\pi}{2},\frac{\pi}{2}]$.
		This will allow us to define what it means for a function $\Phi:\Ind (\mathcal D)\sqcup \{0\} \to \mathbb V$ to satisfy the continuous deformed mesh relations (Definition \ref{def:contintuous deformed mesh relations}).
		
		\begin{definition}\label{def:permissible function}
			Let $\intc:\Ind(\mathcal D)\sqcup \{0\}\to \R$ be a function such that $\intc (0)=0$.
			We say $\intc$ is \textdef{permissible} if for every tilting rectangle $XYWZ$ in $\R\times[-\frac{\pi}{2},\frac{\pi}{2}]$, the surface integral 
			\begin{displaymath}
				\int_{XYWZ} \intc
			\end{displaymath}
			over $XYWZ$ yields a real number.
		\end{definition}
		
		We now define what it means for a function $\Phi$ to satisfy the continuous deformed mesh relations. The reader is invited to compare the following definition with constructions preceding Theorem~\ref{th-one}.
		When we say a real vector space $\mathbb V$ has coordinates indexed by a set $\Omega$, we mean $\mathbb V = \prod_\Omega \R$.
		\begin{definition}\label{def:contintuous deformed mesh relations}
			Let $\mathcal Z$ be a zigzag in $\mathcal D$,  $\intc:\Ind(\mathcal D)\sqcup\{0\}\to \R$ a permissible function, and $\mathbb V$ a real vector space of arbitrary dimension whose coordinates are indexed by a set $\Omega$.
			Suppose $\Phi:\Ind(\mathcal D)\sqcup\{0\} \to \mathbb V$ is a function such that for every tilting rectangle $XYWZ$ in $\R\times[-\frac{\pi}{2},\frac{\pi}{2}]$ and $\omega\in\Omega$ the following equation is satisfied:
			\begin{displaymath}
				\Phi(X)(\omega)+\Phi(W)(\omega)=\Phi(Y)(\omega) + \Phi(Z)(\omega) + \int_{XYWZ} \intc.
			\end{displaymath}
			Then we say $\Phi$ \textdef{satisfies the continuous deformed mesh relations over $\intc$}.
		\end{definition}
		
		In Section~\ref{sec:associahedron} we work with $\mathbb V=\R$.
		However, to adopt a framework that does not require modification in other contexts, we also allow $\mathbb V$ to be an arbitrary product of copies of $\R$.
		
		\begin{remark}\label{rmk:deformed over 0 is not deformed}
		    If $\intc$ in Definition \ref{def:contintuous deformed mesh relations} is the constant function at $0$, then $\Phi$ satisfies the continuous mesh relations as in Section \ref{sec:continuous mesh relations}.
		\end{remark}
		
		\begin{proposition}\label{prop:tilting rectangle for quilting}
			Let $\mathcal Z$ be a zigzag in $\mathcal D$ and $W$ an indecomposable in a patch between $\mathcal Z$ and $\quilt (\mathcal Z)$.
			If $W$ is not in $\mathcal Z$, there exist $X$, $Y$, and $Z$ in $Ob(\mathcal Z)\sqcup \{0\}$ such that $XYWZ$ is a tilting rectangle.
			If $W$ is instead an indecomposable in a patch between $\quilt^{-1}(\mathcal Z)$ and $\mathcal Z$, but not in $\mathcal Z$, then there exist $X$, $Y$, and $Z$ in $Ob(\mathcal Z)\sqcup \{0\}$ such that $WYXZ$ is a tilting rectangle.
		\end{proposition}
		\begin{proof}
			First assume $W$ is contained in a patch between $\mathcal Z$ and $\quilt (\mathcal Z)$.
			If $W$ is contained in a rectangular patch then the indecomposables $Y$ and $Z$ are obtained by intersecting $\mathcal Z$ with the lines of slope $\pm 1$ that intersect at $W$.
			The indecomposable $X$ is the left vertex shared by the line segments used to create the rectangle.
			
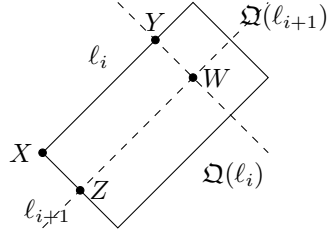
\begin{displaymath}
				\begin{tikzpicture}
					\draw (0,1) -- (1,0) -- (3,2) -- (2,3) -- (0,1);
					\filldraw (0,1) circle[radius=.5mm];
					\draw (0,1) node[anchor=east] {$X$};
					\draw (.5,.5) node[anchor=north east] {$\ell_{i+1}$};
					\draw (1,2) node[anchor=south east] {$\ell_i$};
					\draw (2.5,2.5) node[anchor=south west] {$\quilt (\ell_{i+1})$};
					\draw (2,1) node[anchor=north west] {$\quilt (\ell_i)$};
					\draw[dashed] (0,0) -- (3,3);
					\draw[dashed] (1,3) -- (3,1);
					\filldraw (2,2) circle[radius=.5mm];
					\draw (2,2) node[anchor=west] {$W$};
					\filldraw (1.5,2.5) circle[radius=.5mm];
					\draw (1.5,2.5) node[anchor=south] {$Y$};
					\filldraw (.5,.5) circle[radius=.5mm];
					\draw (.5,.5) node[anchor=west] {$Z$};
				\end{tikzpicture}
			\end{displaymath}
			Note that if $W$ is the intersection of $\quilt(\ell_{i+1})$ and $\quilt(\ell_i)$ then $Y$ is the top point of the rectangle and $Z$ is the bottom point.
			In this case $Y$ or $Z$ may be $0$.
			
			If $W$ is contained in a triangular patch then $Y$ and $Z$ are obtained by intersecting $\mathcal Z$ and the base of the triangle with the lines of slope $\pm 1$ that intersect at $W$.
			At least one of these will be $0$.
			
			Let $h$ be the line containing $W$ that intersects the base of the triangle to the left of $W$.
			Then $X$ is obtained by intersecting $\mathcal Z$ with the line perpendicular to $h$ that intersects at the same point on the base of the triangle.
			
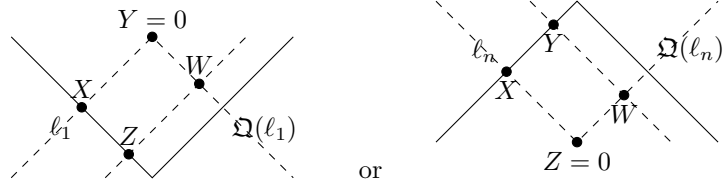
\begin{displaymath}
				\begin{tikzpicture}[scale=1.25]
					\draw (0,1.5) -- (1.5,0) -- (3,1.5);
					\draw (.75,.75) node[anchor=north east] {$\ell_1$};
					\draw (2.25,.75) node[anchor=north west] {$\quilt (\ell_1)$};
					\filldraw (2,1) circle[radius=.5mm];
					\draw (2,1) node[anchor=south] {$W$};
					\filldraw (1.5,1.5) circle[radius=.5mm];
					\draw (1.5,1.5) node[anchor=south] {$Y=0$};
					\filldraw (1.25,.25) circle[radius=.5mm];
					\draw (1.25,.25) node[anchor=south] {$Z$};
					\filldraw (.75,.75) circle[radius=.5mm];
					\draw (.75,.75) node[anchor=south] {$X$};
					\draw[dashed] (1,0) -- (2.5,1.5);
					\draw[dashed] (0,0) -- (1.5,1.5) -- (3,0);
				\end{tikzpicture}
				\qquad
				\text{or}
				\qquad
				\begin{tikzpicture}[scale=1.25]
					\draw (0,0) -- (1.5,1.5) -- (3,0);
					\draw (.75,.75) node[anchor=south east] {$\ell_n$};
					\draw (2.25,.75) node[anchor=south west] {$\quilt (\ell_n)$};
					\filldraw (2,.5) circle[radius=.5mm];
					\draw (2,.5) node[anchor=north] {$W$};
					\draw[dashed] (1,1.5) -- (2.5,0);
					\draw[dashed] (0,1.5) -- (1.5,0) -- (3,1.5);
					\filldraw (1.5,0) circle[radius=.5mm];
					\draw (1.5,0) node[anchor=north] {$Z=0$};
					\filldraw (1.25,1.25) circle[radius=.5mm];
					\draw (1.25,1.25) node[anchor=north] {$Y$};
					\filldraw (.75,.75) circle[radius=.5mm];
					\draw (.75,.75) node[anchor=north] {$X$};
				\end{tikzpicture}
			\end{displaymath}
			If $W$ is in $\quilt (\ell_i)$ and $\mathcal Z$ has one line segment then both $Y$ and $Z$ are $0$.
			
			If instead $W$ is contained in a patch between $\quilt^{-1}(\mathcal Z)$ and $\mathcal Z$, then the argument is the same using the symmetric geometry.
		\end{proof}
		
		\begin{definition}\label{def:can be quilted from}
		    Let $\mathcal Z$ be a zigzag in $\mathcal D$ and let $W$ be an indecomposable in $\mathcal D$.
		    We say \textdef{$W$ can be quilted from $\mathcal Z$} if there exists a nonnegative $n\in\Z$ such that $W$ is in a patch of $\quilt^n(\mathcal Z)$.
		    We say \textdef{$W$ can be inverse quilted from $\mathcal Z$} if there exists a negative $n\in\Z$ such that $W$ is in a patch of $\quilt^n(\mathcal Z)$.
		\end{definition}
		
		The following lemma shows that for any zigzag $\mathcal Z$ in $\mathcal D$ and any indecomposable $W$ in $\mathcal D$, $W$ can be quilted or inverse quilted from $\mathcal Z$.
		
		\begin{lemma}\label{lem:totality of quilting}
			Let $\mathcal Z$ be a zigzag in $\mathcal D$.
			For every indecomposable $W$ in $\mathcal D$, there exists $n\in\Z$ such that $W$ is in a patch of $\quilt^n (\mathcal Z)$.
		\end{lemma}
		\begin{proof}
			Let $y_W$ be the $y$-coordinate of $W$.
			Then there exist $X$ and $X'$ in $\mathcal Z$ where the $y$-value of $X$ is $y_W$ and the $y$-value of $X'$ is $-y_W$.
			If $y_W=0$ then $X=X'$.
			
			Let $x_W$, $x_0$, and $x'_0$ be the $x$-values of respectively $W$, $X$, and $X'$.
			Since $|x_0-x'_0|<\pi$, we have $x'_0 < x_0+\pi$ and $x_0 < x'_0+\pi$.
			So, for every $m\in\Z$ define
			\begin{displaymath}
				x_m	= 	\begin{cases}
							m\pi + x_0 & m\text{ is even} \\
							m\pi + x'_0 & m\text{ is odd}.
 						\end{cases}
			\end{displaymath}
			Then there exists $m\in\Z$ such that $x_m \leq x_W < x_{m+1}$.
			
			Let $j$ be the number of line segments in $\mathcal Z$.
			Then $\quilt^j(\mathcal Z[m])=\mathcal Z[m+1]$ by Theorem \ref{thm:n quilt of a zigzag is the shift}.
			For each $0\leq i < j$ let $X_i$ be the object on $\quilt^i(\mathcal Z[m])$ with $y$-coordinate $y_W$.
			Let $x_{m,i}$ be the $x$-coordinate of $X_i$ and set $x_{m,j}=x_{m+1}$.
			Then there exists $0\leq i_W < j$ such that $x_{m,i_W} \leq x_W \leq x_{m,i_W+1}$.
			
			It follows that $W$ is in a patch of $\quilt^{i_W}(\mathcal Z[m])$, so $W$ is in a patch of $\quilt^n(\mathcal Z)$, where $n= jm+i_W$.
		\end{proof}
		
			Let  $\intc:\Ind(\mathcal D)\sqcup \{0\}\to \R$ be a permissible function,  $\mathbb V$ a real vector space whose coordinates are indexed by a set $\Omega$, $\mathcal Z$ a zigzag in $\mathcal D$, and $\Phi:Ob(\mathcal Z)\to \mathbb V$ a function.
			For each indecomposable $W\notin \mathcal Z$ in a patch between $\mathcal Z$ and $\quilt (\mathcal Z)$ (respectively between $\quilt^{-1}(\mathcal Z)$ and $\mathcal Z$), there is a unique tilting rectangle $\Diamond=XYWZ$ (respectively $\Diamond=WYXZ$) as in Proposition \ref{prop:tilting rectangle for quilting}.
			
			\begin{definition}\label{def:quilting value}
			With the same notation and setting as above, the \textdef{quilting value of $W$ over $\intc$} is the vector $\Phi(W)$ in $\mathbb V$, where the value for each coordinate $\omega\in\Omega$ is given by
			\begin{displaymath}
				\Phi(W)(\omega) := \Phi (Y)(\omega) + \Phi (Z)(\omega) - \Phi (X)(\omega) + \int_{\Diamond} \intc.
			\end{displaymath}
		\end{definition}
		
		\begin{proposition}\label{prop:all the values}
			Let $\intc:\Ind(\mathcal D)\sqcup \{0\}\to \R$ be a permissible function and $\mathbb V$ a real vector space with coordinates indexed by a set $\Omega$.
			For a zigzag $\mathcal Z$ and function $\Phi:\Ind(\mathcal Z)\to \mathbb V$, there exists a unique extension of $\Phi$ to $\Ind(\mathcal D)\sqcup\{0\}$ that satisfies the continuous deformed mesh relations over $\intc$.
		\end{proposition}
		\begin{proof}
			Using Definition \ref{def:quilting value} we may extend $\Phi$ to all the patches of $\mathcal Z$ and, on these patches, the extension satisfies the continuous deformed mesh relations over $\intc$.
			Similarly, we may extend $\Phi$ to all the patches of $\quilt^{-1}(\mathcal Z)$.
			Notice these extensions are unique.
		
			By Lemma \ref{lem:totality of quilting}, each indecomposable $W$ in $\mathcal D$ is in a patch of $\quilt^n(\mathcal Z)$, for some $n \in \Z$.
			We recursively use this argument to obtain the desired extension of $\Phi$.
			In particular, this uniquely defines $\Phi(W)$.
			Therefore, we may extend $\Phi$ to $\Ind(\mathcal D)$ as stated in the proposition.
		\end{proof}

	\section{Connections to representation theory}\label{sec:rep theory}
	Our goal in this section is to highlight some fundamental connections between our construction and the representation theory of quivers. 
	More specifically, in Sections~\ref{sec:g-vectors} and~\ref{sec:dimension vectors} we respectively introduce the notions of $g$-vectors and dimension vectors in our setting. 
	Furthermore, from the results of Sections~\ref{sec:t-structures} and~\ref{sec:additional properties of heart}, we observe that for any zigzag $\mathcal Z$, there is a $t$-structure in $\mathcal D$ whose heart is analogous to the category of finitely generated representations of an $A_n$ quiver. In particular, by the end of this section the reader observes that the subspace of $\R\times(-\frac{\pi}{2},\frac{\pi}{2})$ corresponding to the indecomposables in $\mathcal D^\heartsuit$  shares many properties with the Auslander--Reiten quiver of a type $A_n$ quiver (for further details, see Section \ref{sec:additional properties of heart}). 
		\subsection{$t$-structures}\label{sec:t-structures}
		For a zigzag $\mathcal Z$ in $\mathcal D$, define the full subcategories
		\begin{align*}
			\mathcal D^{\leq 0} &= \add\{X \mid X\text{ can be quilted from }\mathcal Z\} \\
			\mathcal D^{\geq 0} &= \add\left(\{Y \mid Y\text{ can be inverse quilted from }\mathcal Z[1]\}\setminus \mathcal Z[1]\right).
		\end{align*}
		We have a $t$-structure because the following hold.
		\begin{itemize}
			\item $\mathcal D^{\leq 0}$ is closed under $[1]$ and $\mathcal D^{\geq 0}$ is closed under $[-1]$.
			\item For any indecomposables $X$ in $\mathcal D^{\leq 0}$ and $Y$ in $\mathcal D^{\geq 0}$, we have $\Hom_{\mathcal D}(X,Y[-1])=0$.
				Since $\mathcal D$ is Krull--Schmidt, this extends to all $X$ in $\mathcal D^{\leq 0}$ and $Y$ in $\mathcal D^{\geq 0}$.
			\item Every indecomposable in $\mathcal D$ belongs to at least one of $\mathcal D^{\leq 0}$ or $\mathcal D^{\geq 0}$.
			This immediately yields that every object $E$ in $\mathcal D$ belongs to a distinguished triangle
				\begin{displaymath}
					\xymatrix{
						X \ar[r] & E \ar[r] & Y[-1] \ar[r] & X[1],
					}
				\end{displaymath}
				where $X$ is in $\mathcal D^{\leq 0}$ and $Y$ is in $\mathcal D^{\geq 0}$. 
		\end{itemize}
		
	Notice that the heart $\mathcal D^{\heartsuit}:=\mathcal{D}^{\ge 0} \cap \mathcal{D}^{\le 0}$ does not contain $\mathcal Z[1]$, but it does contain all other indecomposables that can be both quilted from $\mathcal Z$ (including $\mathcal Z$) and inverse quilted from $\mathcal Z[1]$.
	We observe that $\mathcal D^\heartsuit$ is similar to the categories of representations of continuous quivers in \cite{IT15,IRT23,R19}, where the projective representations are those on $\mathcal Z$ (Proposition \ref{prop:zigzag is projective}).
	See Section \ref{sec:additional properties of heart} for further discussion.
	Representations of continuous quivers connect our interpretation to the construction given in \cite{B-MDMTY24}.
	
	We remark that once we have chosen $\mathcal Z$, we have no further choices regarding our $t$-structure.
	This implies that the heart is in some sense canonical.
		
	\subsection{$\gZ$-vectors}\label{sec:g-vectors}
	Let $\mathcal Z$ be a zigzag in $\mathcal D$.
	We now show that the heart $\mathcal D^\heartsuit$ of the $t$-structure obtained from $\mathcal Z$ in Section \ref{sec:t-structures} has $\gZ$-vectors, which behave like $g$-vectors in the classical sense.
		
	Before we state the next definition, recall that each object $Z$ in $\mathcal Z$ is an indecomposable in $\mathcal D$.
		
	\begin{definition}[$g_{\mathcal Z}$-vectors]\label{def:g-vector}
	Let $\mathbb V$ be the real vector space whose coordinates are indexed by $Ob(\mathcal Z)$.  For each object $Z$ in $\mathcal Z$, set $g_{\mathcal Z}(Z)$ to be the vector in $\mathbb V$ whose coordinates are $1$ in the $Z$-coordinate and $0$ elsewhere.
	Let $\intc$ be the $0$ function.
	We uniquely extend $g_{\mathcal Z}$ to all of $\mathcal D$ as in Proposition \ref{prop:all the values}.
	For an indecomposable $X$ in $\mathcal D^\heartsuit \sqcup \mathcal Z[1]$, the \textdef{$\gZ$-vector of $X$} is defined to be $g_{\mathcal Z}(X)$.
	\end{definition}
		
		We now give an explicit description of the $g_{\mathcal Z}$-vectors in $\mathcal{D}^\heartsuit \sqcup \mathcal{Z}[1]$; namely, in Proposition \ref{prop:g-vectors}, we show that each such $g_{\mathcal Z}$-vector $g_{\mathcal Z}(X)$ is a finite sum of $g_{\mathcal Z}$-vectors of objects in $\mathcal Z$.
		
    For each indecomposable in $\mathcal{Z}$, the associated $g_{\mathcal{Z}}$-vector is already defined. Therefore, to explicitly describe the $\gZ$-vector of every indecomposable $X$ in $\mathcal D^\heartsuit$, we only need to treat those indecomposable $X$ in $\mathcal D^\heartsuit$ such that $X$ is not in $\mathcal Z$. 
	For every such indecomposable $X$ consider the rays $\mathfrak u$ and $\mathfrak d$ in $\R\times[-\frac{\pi}{2},\frac{\pi}{2}]$, respectively with slopes $-1$ and $+1$, which emanate from $X$ and propagate in the negative $x$-direction.  (For a graphical depiction, see Figure \ref{fig:g-vector}.)
	Each of these rays $\mathfrak u$ and $\mathfrak d$ may ``bounce'' off one of the horizontal lines $y=\frac{\pi}{2}$ or $y=-\frac{\pi}{2}$ at most once before they intersect $\mathcal Z$.
	Further, we associate a unique point $Z_{\mathfrak{u}}$ in $\mathcal{Z}$ to the ray $\mathfrak{u}$ as follows. (The point $Z_{\mathfrak d}$ can be described analogously.)
		
		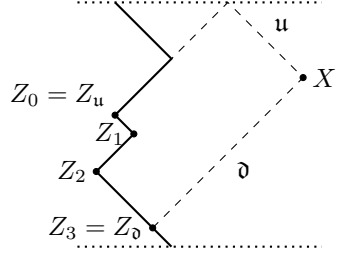
\begin{figure}
		\begin{center}
			\begin{tikzpicture}
				\foreach \x in {0, 3.25}
					\draw[dotted, thick] (-.25,\x) -- (3,\x);
					
				\draw[thick] (1,0) -- (0,1) -- (.5,1.5) -- (.25, 1.75) -- (1,2.5) -- (.25, 3.25);
				
				\draw[dashed] (0.75,0.25) -- (2.75,2.25) -- (1.75,3.25) -- (1,2.5);
				
				\filldraw (2.75,2.25) circle[radius=.4mm];
				\draw (2.75,2.25) node[anchor=west] {$X$};
				\draw (2.25,2.75) node[anchor=south west] {$\mathfrak u$};
				\draw (1.75, 1.25) node[anchor=north west] {$\mathfrak d$};
				
				\filldraw (.25,1.75) circle[radius=.4mm];
				\draw (.25,1.75) node[anchor=south east] {$Z_0=Z_{\mathfrak u}$};
				\filldraw (.75,.25) circle[radius=.4mm];
				\draw (.75,.25) node[anchor=east] {$Z_3=Z_{\mathfrak d}$};
				\filldraw (.5,1.5) circle[radius=.4mm];
				\draw (.5,1.5) node[anchor=east] {$Z_1$};
				\filldraw (0,1) circle[radius=.4mm];
				\draw (0,1) node [anchor=east] {$Z_2$};
			\end{tikzpicture}
			\caption{An example $\mathcal Z$ and $X$ with $\mathfrak u$, $\mathfrak d$, $Z_{\mathfrak u}$, and $Z_{\mathfrak d}$. Each of the $Z_i$ are also shown.
			The region bounded by $\mathfrak u$, $\mathfrak d$, and the zigzag will be of use in Section \ref{sec:solutions}.}\label{fig:g-vector}
		\end{center}
		\end{figure}

		\begin{enumerate}
			\item Suppose $\mathfrak u$ bounces off  $y=\frac{\pi}{2}$ \emph{and} intersects $\mathcal Z$ at a right vertex $Z$ (Definition \ref{def:left vertex}).
				Then we define $Z_{\mathfrak u}$ to be the adjacent left vertex below $Z$ or the intersection between $\mathfrak u$ and $\mathfrak d$ to the left of $X$, whichever is closer to $X$.
			\item If $\mathfrak u$ does not bounce off  $y=\frac{\pi}{2}$ \emph{or} does not intersect $\mathcal Z$ at a right vertex of $\mathcal Z$, we define $Z_{\mathfrak u}$ to be the rightmost intersection between $\mathcal Z$ and $\mathfrak u$.
		\end{enumerate}
		By Euclidean geometry and the triangulated structure of $\mathcal D$ (Section \ref{sec:triangulated structure of D}), $Z_{\mathfrak u} =Z_{\mathfrak d}$ if and only if $X\in \mathcal Z[1]$.
		
		Suppose $X\notin\mathcal Z[1]$.
		Consider the (possibly empty) set of all left and right vertices of $\mathcal Z$ whose $y$-coordinates are strictly between those of $Z_{\mathfrak u}$ and $Z_{\mathfrak d}$.
		Enumerate these left and right vertices as $Z_i$, starting with $Z_1$, where $i<i'$ if the $y$-coordinate of $Z_i$ is greater than the $y$-coordinate of $Z_{i'}$.
		Let $j$ be the number of such $Z_i$'s, $Z_0=Z_{\mathfrak u}$, and $Z_{j+1}=Z_{\mathfrak d}$.
		For $i\notin[0,j+1]$, set $Z_i=0$.
		
		\begin{proposition}\label{prop:g-vectors}
			Let $X$ be an indecomposable in $\mathcal D^\heartsuit\sqcup \mathcal Z[1]$, but not in $\mathcal Z$.
			Then the $\gZ$-vector of $X$ is given by
			\begin{displaymath}
				g_{\mathcal Z}(X) = \begin{cases}
 					\sum\limits_{\mathrm{odd}\, i} g_{\mathcal Z}(Z_i) - \sum\limits_{\mathrm{even}\, i} g_{\mathcal Z}(Z_i) & Z_0 \text{ is left of } Z_1 \\
 					\sum\limits_{\mathrm{even }\, i} g_{\mathcal Z}(Z_i) - \sum\limits_{\mathrm{odd }\, i} g_{\mathcal Z}(Z_i) & Z_0 \text{ is right of } Z_1 \\
 					-g_{\mathcal Z}(X[-1]) & X\in\mathcal Z[1].
 				\end{cases}
			\end{displaymath}	
		\end{proposition}
	\begin{proof}
		First, we treat the case $X\in\mathcal Z[1]$. In this case, note that $Z_{\mathfrak u}=Z_{\mathfrak d}=X[-1]$.
		Consider the tilting rectangle whose left corner is $X[-1]$ and right corner is $X$, where the top and bottom corners lie on $y=\frac{\pi}{2}$ and $y=-\frac{\pi}{2}$, respectively.
		Thus, $g_{\mathcal Z}(X[-1]) + g_{\mathcal Z}(X)= 0$ implies the desired result.
			
		Now suppose $X\notin \mathcal Z[1]$.
		Let $n_\ell$ be the number of line segments of $\mathcal Z$.
		We proceed by strong induction on $n<n_\ell$, starting with $n=0$.
		If $X$ is in a patch of $\mathcal Z=\quilt^0(\mathcal Z)$, the result follows from Definition \ref{def:quilting value}.
			
		Suppose the result holds for all $X$ in any patch of $\quilt^m_{\mathcal Z}$, for all $0\leq m \leq n< n_\ell-1$.
		Let $X$ be in a patch of $\quilt^{n+1}(\mathcal Z)$, and define $Z_{\mathfrak u}$, $Z_{\mathfrak d}$, and the $Z_i$ as above. 
			Further, we can assume $X$ is not in $\quilt^{n+1}(\mathcal Z)$, as otherwise $X$ is in a patch of $\quilt^n(\mathcal Z)$ and we are done.
			
			Since $X\notin\mathcal Z[1]$, we know $Z_{\mathfrak u}\neq Z_{\mathfrak d}$ and the $y$-coordinate of $Z_{\mathfrak u}$ is greater than the $y$-coordinate of $Z_{\mathfrak d}$.
			Furthermore, since $X$ is in a patch of $\quilt^{n+1}(\mathcal Z)$, there exist $X'$, $Y'$, and $Z'$ in $\quilt^{n+1}(\mathcal Z)\sqcup \{0\}$ such that $X'Y'XZ'$ is a tilting rectangle.
			By induction we know the proposition holds for $X'$, $Y'$, and $Z'$.
			Note that the rays $\mathfrak d$ coming from $Y'$ and $X'$ will determine the same $Z_{\mathfrak d}$ for both objects.
			Similarly, $Z_{\mathfrak u}$ is the same for $Z'$ and $X'$.
			Moreover, $Z_{\mathfrak u}$ is the same for $X$ and $Y'$; $Z_{\mathfrak d}$ is the same for $X$ and $Z'$.
			The $y$-coordinates of these four objects in $\mathcal Z$ will all be distinct since $X$ is not in $\quilt^{n+1}(\mathcal Z)$.
			Schematically, there are four cases of the $y$-coordinates based on which rays (if any) bounce off the boundary:
			
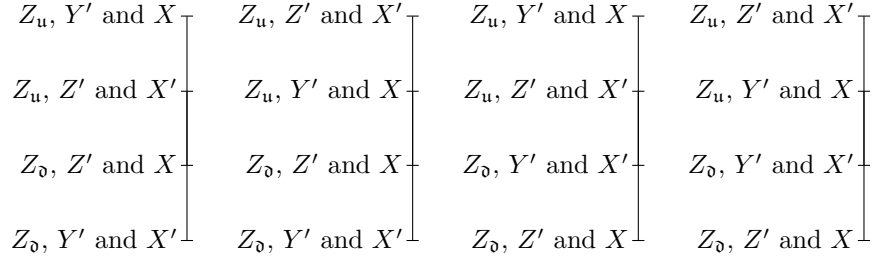
\begin{displaymath}
			\begin{tikzpicture}
				\foreach \x in {0, 3, 6, 9}
				{
					\draw[|-|] (\x,0) -- (\x,3);
					\draw[|-|] (\x,1) -- (\x,2);
				}
				\foreach \x in {0, 3}
				{
					\draw (\x,0) node[anchor=east] {$Z_{\mathfrak d}$, $Y'$ and $X'$};
					\draw (\x,1) node[anchor=east] {$Z_{\mathfrak d}$, $Z'$ and $X$};
				}
				\foreach \x in {6, 9}
				{
					\draw (\x,1) node[anchor=east] {$Z_{\mathfrak d}$, $Y'$ and $X'$};
					\draw (\x,0) node[anchor=east] {$Z_{\mathfrak d}$, $Z'$ and $X$};
				}
				\foreach \x in {0, 6}
				{
					\draw (\x,2) node[anchor=east] {$Z_{\mathfrak u}$, $Z'$ and $X'$};
					\draw (\x,3) node[anchor=east] {$Z_{\mathfrak u}$, $Y'$ and $X$};
				}
				\foreach \x in {3, 9}
				{
					\draw (\x,3) node[anchor=east] {$Z_{\mathfrak u}$, $Z'$ and $X'$};
					\draw (\x,2) node[anchor=east] {$Z_{\mathfrak u}$, $Y'$ and $X$};
				}
			\end{tikzpicture}	
			\end{displaymath}
			Therefore, since
			\begin{displaymath}
				g_{\mathcal Z}(X) = g_{\mathcal Z}(Y')+g_{\mathcal Z}(Z')-g_{\mathcal Z}(X'),
			\end{displaymath}
			we see the proposition follows.
		\end{proof}

		\subsection{Dimension vectors}\label{sec:dimension vectors}
		Again let $\mathcal Z$ be a zigzag and $\mathcal D^\heartsuit$ be the heart of the $t$-structure obtained from $\mathcal Z$, as in Section \ref{sec:t-structures}.
		Now we show that for each indecomposable $X$ in $\mathcal D^\heartsuit$, there is a notion of dimension vector analogous to the discrete case.
		
		First we introduce a partial order on $\mathcal{Z}$: If $Z$ and $Z'$ belong to the same line segment of $\mathcal{Z}$, we put $Z \leq Z'$ provided that the $x$-coordinate of $Z$ is no larger than the $x$-coordinate of $Z'$.
		If $Z$ and $Z'$ are not on the same line segment in $\mathcal Z$, then $Z$ and $Z'$ are not comparable.
		Then, for each $Z$ in $\mathcal Z$, define $\udim_{\mathcal Z}(Z)$ in $\prod_{Ob(\mathcal Z)} \R$ to be $1$ on each $Z'$-coordinate with $Z'\leq Z$, and $0$ elsewhere.
		Again by Proposition \ref{prop:all the values}, we extend $\udim_{\mathcal Z}$ uniquely to a function $\udim_{\mathcal Z}:\mathcal D\to\prod_{Ob(\mathcal Z)} \R$.
		Now, for each indecomposable object $X$ in $\mathcal D^\heartsuit$, define the \textdef{dimension vector of $X$ with respect to $\mathcal Z$} to be $\udim_{\mathcal Z}(X)$.
		
		\begin{proposition}\label{def:dimension vectors are intervals}
			Let $X$ be an indecomposable object in $\mathcal D^\heartsuit$.
			Then each coordinate of $\udim_{\mathcal Z}(X)$ is $0$ or $1$.
			Moreover, $\{Z\in \Ind(\mathcal Z) \mid Z\text{-coordinate of $\udim_{\mathcal Z}(X)$ is }1 \}$ forms a connected set, where $\Ind(\mathcal Z)\subsetneq \R^2$ has the subspace topology.
		\end{proposition}
		\begin{proof}
			If $X\in\mathcal Z$, the proposition follows by the definition of $\udim_{\mathcal Z}$.
			We use the same notation as in Section \ref{sec:g-vectors}. In particular, let $Z_{\mathfrak u}$, $Z_{\mathfrak d}$, and the $Z_i$ be as before.
			For the case where $X\notin\mathcal{Z}$, we know that $Z_0$ and $Z_1$ are distinct and we can assume that $Z_0$ is to the left of $Z_1$ (the other case is similar). 
			We further remark that the proof of Proposition \ref{prop:g-vectors} does not rely on the values of $g_{\mathcal Z}$. Thus,
			\begin{displaymath}
				\udim_{\mathcal Z}(X) = \sum\limits_{\text{odd } i} \udim_{\mathcal Z}(Z_i) - \sum\limits_{\text{even } i} \udim_{\mathcal Z}(Z_i) = \sum_{i\in \Z} (-1)^{i+1} \udim_{\mathcal Z}(Z_i).
			\end{displaymath}
			
			We now prove the result by induction on the number of $Z_i$'s.
			We start with $Z_0=Z_{\mathfrak u}$ and $Z_1=Z_{\mathfrak d}$.
			In this case, the result is straightforward to check.
			
			For the induction step on $j+1\geq 2$, let $\mathcal S$ be the set of $Z\in\mathcal Z$ such that the $Z$-coordinate of $\sum_{i=1}^j (-1) \udim_{\mathcal Z}(Z_i)$ is $1$.
			Assume $\mathcal S$ is connected and that the $Z$-coordinate of $\sum_{i=1}^j (-1) \udim_{\mathcal Z}(Z_i)$ is $0$ for all $Z\notin \mathcal S$.
			There are even and odd cases for the induction step; we first consider the even case.
			If $j+1$ is even, then $Z_j$ is a right vertex and $\mathcal S\setminus \{Z \mid Z\leq Z_{j+1}\}$ is also a connected set.
			If $j+1$ is odd, then $Z_j\notin \mathcal S$ is a left vertex and $\mathcal S\cup \{Z\mid Z\leq Z_{j+1}\}$ is a connected set.
			In both cases, $\sum_{i=1}^j (-1)^{i+1}\udim_{\mathcal Z}(Z_i)$ is $1$ precisely on $\mathcal S$ and $0$ elsewhere.
		\end{proof}

	\subsection{Continuous representations and $\mathcal D^\heartsuit$}\label{sec:additional properties of heart}
	    In this brief section we discuss how $\mathcal D^\heartsuit$ is similar to $\rep(A_n)$ for a type $A_n$ quiver, and thus may be thought of as a category of continuous representations.
	    We have already shown that the $\gZ$-vectors (Section \ref{sec:g-vectors}) and dimension vectors (Section \ref{sec:dimension vectors}) in $\mathcal D^\heartsuit$ behave similarly to $g$-vectors and dimension vectors, respectively, in $\rep(A_n)$.
	
	    For a type $A_n$ quiver (Section \ref{sec:Amplituhedron for Dynkin Quivers}), the projective indecomposables in $\rep(A_n)$ form a zigzag shape in the (augmented) Auslander--Reiten quiver.
	    The following proposition shows that the indecomposables in the zigzag $\mathcal Z$ are exactly the projective indecomposable objects in $\mathcal D^\heartsuit$.
	    
		\begin{proposition}\label{prop:zigzag is projective}
		    An indecomposable $Z$ is projective in $\mathcal D^\heartsuit$ if and only if $Z\in \Ind(\mathcal Z)$.
		\end{proposition}
		\begin{proof}
		    Let $Z$ be an indecomposable in $\mathcal Z$ and $X\to Y \to Z$ be a short exact sequence in $\mathcal D^\heartsuit$.
		    Note that if an indecomposable $W$ in $\mathcal D^\heartsuit$ is not on $\mathcal Z$ then $\Hom_{\mathcal D^\heartsuit}(W,Z)=0$.
		    This implies $X$ and $Y$ consist of indecomposable summands on $\mathcal Z$ as well.
		    Further note that for any pair of indecomposables $W, W'$ in $\mathcal Z$, if $W$ and $W'$ are not on the same line segment then $\Hom_{\mathcal{D}^\heartsuit}(W,W')=0$ (see Hom supports in Figure \ref{fig:morphs}).
		    By the triangulated structure proven in \cite[Proposition 2.5.1]{IT15}, and thus the abelian structure in $\mathcal D^\heartsuit$, any such short exact sequence must be split.
		    Thus, $Z$ is projective in $\mathcal D^\heartsuit$.
		    
		    Let $X$ be an indecomposable in $\mathcal D^\heartsuit$, but not in $\mathcal Z$.
		    Then we may find a small enough nondegenerate tilting rectangle $\Diamond$ in $\mathcal D^\heartsuit$ to the right of $\mathcal Z$, such that the right corner of $\Diamond$ is $X$.
		    This distinguished triangle, and thus the short exact sequence in $\mathcal D^\heartsuit$, does not split.
		\end{proof}
		
		In the following theorem, we capture some of the main properties of $\mathcal D^\heartsuit$. Since the results follow from Propositions \ref{prop:g-vectors} and \ref{prop:zigzag is projective}, we omit the proof. 
		
		\begin{theorem}\label{ref:nice heart}
		The abelian category 
		$\mathcal D^\heartsuit$ has enough projectives, it is Krull--Schmidt, and every indecomposable object is finitely generated.
		From the construction of $\mathcal D$, it follows that $\Ext^i_{\mathcal D^\heartsuit}(X,Y)=0$ for $i>1$, thus  $\mathcal D^\heartsuit$ is of global dimension $1$.
		Furthermore, the isomorphism classes of indecomposable objects of $\mathcal D$ are given by shifts of those in $\mathcal D^\heartsuit$.
		\end{theorem}
		
		From the preceding theorem, observe that one can think of $\mathcal D^\heartsuit$ as the category of certain finitely generated representations of a continuous quiver whose orientation is inherited by the partial order on $\mathcal Z$. The reader is referred to \cite{IRT23} for a detailed introduction to continuous quivers of type $A$.
	    Our results are inspired by, but not reliant upon, that work.
				
	\section{\textbf{T}-clusters}\label{sec:T-clusters}
		In this section, we present a continuous generalization of the clusters and compatibility in \cite[Section 2]{B-MDMTY24}.
		This will be used in our construction of a continuous analogue of the ABHY associahedron for type $A$ quivers, as we discuss in Section \ref{sec:associahedron}.
		
		For the remainder of the paper, let $\mathcal{Z}$ be a fixed
		zigzag in $\mathcal{D}$ (Definition~\ref{def:zigzag}).  Recall that $\mathcal D^\heartsuit$ denotes the heart of the $t$-structure in $\mathcal{D}$, as described in Section \ref{sec:t-structures}. We consider the following subcategory of $\mathcal D$. 
		\begin{notation}\label{note:closed heart}
			Let $\ClosedHeart := \add(\Ind(\mathcal D^\heartsuit)\sqcup \Ind(\mathcal Z[1]))$. Namely, $\ClosedHeart$ is the full subcategory of $\mathcal D$ whose objects are finite direct sums of indecomposable objects in $\mathcal D^\heartsuit$ and $\mathcal Z[1]$.
		\end{notation}
		\subsection{Compatibility}
		To generalize the compatibility in \cite{B-MDMTY24} to the continuous version, we now define compatibility in $\ClosedHeart$, making use of the continuous deformed mesh relations (Definition \ref{def:contintuous deformed mesh relations}).
		\begin{definition}\label{def:compatibility}
			Let $X$ and $Y$ be indecomposable objects in $\ClosedHeart$.
			We say $X$ and $Y$ are \textdef{incompatible} if there exists a distinguished triangle in $\mathcal D$ of one of the following forms:
			\begin{displaymath}
				\xymatrix{
					X \ar[r] & E \ar[r] & Y \ar[r] & X[1] & \text{or} & Y \ar[r] & E \ar[r] & X\ar[r] & Y[1].
				}
			\end{displaymath}
			Otherwise, we say $X$ and $Y$ are \textdef{compatible}.
		\end{definition}
		
		\begin{remark}\label{rmk:trivial case}\label{rmk:geometric compatibility}
		    Note that Definition \ref{def:compatibility} implies that each indecomposable $X\in\ClosedHeart$ is compatible with itself.
		    Recall that in Section \ref{sec:Amplituhedron for Dynkin Quivers}, for two vertices of $\widetilde{\Gamma}_{\mathcal{C}}$, we considered the notion of compatibility based on the associated cluster variables in the cluster algebra of type $A_n$.
		    From the tilting rectangles (Definition \ref{def:tilting rectangle}), one observes that the compatibility condition in Definition~\ref{def:compatibility} is analogous to that in Section \ref{sec:Amplituhedron for Dynkin Quivers}, described in Remark \ref{rmk:finite compatibility}.
		    In particular, two indecomposable objects $X$ and $Y$ in $\ClosedHeart$ are incompatible if and only if there exists a tilting rectangle in $\R\times[-\frac{\pi}{2},\frac{\pi}{2}]$ whose left and right corners are $X$ and $Y$ (or $Y$ and $X$).
			
			We also note that if $X$ is an indecomposable object in $\mathcal Z$, then $X$ and $X[1]$ are incompatible. This is because there is a distinguished triangle $X\to 0\to X[1]\to X[1]$ in $\mathcal D$, which corresponds to a tilting rectangle in $\R\times[-\frac{\pi}{2},\frac{\pi}{2}]$.
			By Proposition \ref{prop:zigzag is projective}, this recovers a well known property of cluster structures in the categorical settings: Every indecomposable projective $P$ is incompatible with its shift $P[1]$. 
		
		    Finally, we remark that our compatibility condition differs from the condition used by Igusa and Todorov in \cite{IT15}, which is based on $\Ext$ spaces.
		    It also differs from the condition for $\mathbf E$-clusters used by Igusa, Todorov, and the fourth author in \cite{IRT22}, which is based on the Euler product.
        \end{remark}

		\subsection{\textbf{T}-clusters}
		In this subsection, we define $\TT$-clusters.
		Since our compatibility condition is based on tilting rectangles, we use the prefix $\TT$-, which also distinguishes our compatibility condition from those in \cite{IT15} and \cite{IRT22}.
		While the $\TT$-clusters do not form a cluster structure in the sense of \cite{BIRS09}, they have many properties of clusters.
		In particular, in Section \ref{sec:algebraic mutation} we show that mutation of $\TT$-clusters is relatively well behaved.
		
		\begin{definition}\label{def:T-cluster}
		    A \textdef{$\TT$-cluster} $\mathcal{T}$ is a maximal collection of pairwise compatible indecomposable objects in $\ClosedHeart$.  That is, 
			a collection of indecomposable objects $\mathcal T$ in $\ClosedHeart$ is a $\TT$-cluster if every pair $X$ and $Y$ in $\mathcal T$ are compatible, and 
	        if for each $Z \notin \mathcal{T}$ there exists $X \in \mathcal{T}$ such that $X$ and $Z$ are incompatible.
		\end{definition}
		
		\begin{example}\label{xmp:universal examples}
			We provide a list of basic examples of $\TT$-clusters.
			\begin{enumerate}
				\item The set of indecomposable objects in $\mathcal Z\subsetneq \ClosedHeart$ is a $\TT$-cluster. Similarly, the set of indecomposable objects in $\mathcal Z[1]\subsetneq \ClosedHeart$ is a $\TT$-cluster.
				\item\label{xmp:universal examples:vertical curve} 
				Let $\ell$ be a smooth curve in $\mathcal D$ such that the slope of $\ell$ at each point is less than $-1$, greater than $+1$, or equal to $\infty$ and, for all $a\in(-\frac{\pi}{2},\frac{\pi}{2})$, there is an $A\in\ell$ such that the $y$-coordinate of $A$ is $a$.
				If $\ell\subsetneq \Ind(\ClosedHeart)$,
				then $\mathcal T=\{A \mid A\in \ell\}$ is a $\TT$-cluster.
				\item Let $\mathcal Z'$ be a zigzag in $\mathcal D$ distinct from $\mathcal Z$ and $\mathcal Z[1]$. If $\Ind(\mathcal Z')\subsetneq\Ind(\ClosedHeart)$, then $\mathcal Z'$ is a $\TT$-cluster.
			\end{enumerate}	
		\end{example}
		
		\begin{remark}\label{rmk:disconnected}
		    Note that all of the $\TT$-clusters in Example \ref{xmp:universal examples} are connected as subsets of $\R\times(-\frac{\pi}{2},\frac{\pi}{2})$.
		    We warn the reader that this is not always the case.
		    In fact, $\TT$-clusters may be \textdef{totally disconnected} in $\R\times(-\frac{\pi}{2},\frac{\pi}{2})$.
		    We further discuss this phenomenon through an example in Section \ref{sec:totally disconnected}.
		\end{remark}
		
		\subsubsection{A totally disconnected example}\label{sec:totally disconnected}
		Consider the zigzag $\mathcal Z_+$ consisting of one line segment of slope $+1$.
		Without loss of generality, assume the indecomposable corresponding to the point $(0,0)$ is in $\mathcal Z_+$.
		We use a recursive process to modify $\mathcal Z_+$ and obtain a totally disconnected $\TT$-cluster.
		In fact, we produce a sequence of $\TT$-clusters $\mathcal T_0, \mathcal T_1, \mathcal T_2,\ldots$ where the limit of the process, denoted by $\mathcal T_\infty$, is a totally disconnected set which is maximally compatible.  We note here that although the construction of the sequence of $\TT$-clusters depends on $\mathcal{Z}_+$, we avoid unwieldy notation by suppressing mention of $\mathcal{Z}_+$ in the notation of $\mathcal{T}_i$ and $\mathcal{T}_\infty$.
		
		As already stated in Example \ref{xmp:universal examples}, $\mathcal Z_+$ is a $\TT$-cluster.
		Set $\mathcal T_0=\mathcal Z_+$. Now we aim to explicitly describe $\mathcal{T}_1$, for which the reader may find it helpful to refer to Figure \ref{fig:T1} as we construct this new $\TT$-cluster.
        Choose an integer-indexed subset $\mathcal X=\{(x_i,y_i) \mid i \in \mathbb{Z} \}$ of $\mathcal T_0=\mathcal Z_+$ such that
		\begin{itemize}
			\item $y_i < y_{i+1}$, for all $i\in \mathbb{Z}$; 
			\item $\mathcal X$ has two accumulation points, exactly at $(-\frac{\pi}{2},-\frac{\pi}{2})$ and $(\frac{\pi}{2},\frac{\pi}{2})$.
		\end{itemize}
		For each $i \in \mathbb{Z}$, let $X_i$ be the indecomposable object in $\mathcal{Z}_+$ corresponding to $(x_i,y_i)$, and note that each of these points has the form $X_i = (t_i\pi-\frac{\pi}{2},t_i\pi-\frac{\pi}{2})$ for some $t_i\in(0,1)$.

		For each $Y\in\mathcal Z_+\setminus \mathcal X$, there exist $i\in\mathbb Z$ and $s\in(0,1)$ such that $Y$ corresponds to the point
		\begin{displaymath}
			(1-s)(x_i,y_i) + s(x_{i+1},y_{i+1}).
		\end{displaymath} 
		For each $X_i$, let
		\begin{align*}
			E_i &:= \Bigl( x_{i+1}+(x_i+\frac{\pi}{2}), y_{i+1}-(y_i+\frac{\pi}{2})\Bigr), \\
			F_i &:= \Bigl(x_{i+1}+(x_i+\frac{\pi}{2})+(x_{i+1}-x_i), \,-\frac{\pi}{2}\Bigr).	
		\end{align*}
		Then, by $f_1(Y)$ we denote the indecomposable in $\mathcal{C}_{\mathcal{Z}_+}$ which corresponds to the point $(1-s)E_i + sF_i$.
		Now, we set
		\begin{displaymath}
			\mathcal T_1 := \{X_i\} \cup \{f_1(Y)\mid \text{for all } i\in\mathbb Z, Y\neq X_i\}.
		\end{displaymath}
		
		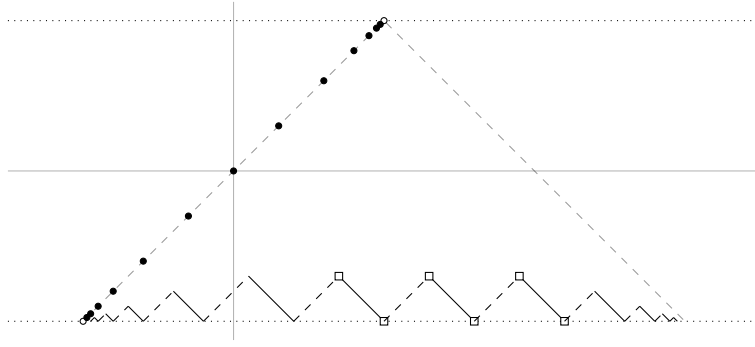
\begin{figure}
		\begin{center}
			\begin{tikzpicture}
				\draw[dotted] (-5,2) -- (5,2);
				\draw[dotted] (-5,-2) -- (5,-2);
				\draw[dashed, draw opacity=.4] (-4,-2) -- (0,2) -- (4,-2);
				\draw[draw opacity = .3] (-5,0) -- (5,0);
				\draw[draw opacity = .3] (-2,2.25) -- (-2,-2.25);
				
				\foreach \x in {-2,2}
					\filldraw[fill=white](\x-2,\x) circle[radius=.4mm];
				\foreach \x in {-1.95, -1.9, -1.8, -1.6, -1.2, -0.6, 0, 0.6, 1.2, 1.6, 1.8, 1.9, 1.95}
				{
					\filldraw (\x-2, \x) circle[radius=.4mm];
				}
				
				\coordinate (A) at (-3.85,-1.95);
				\coordinate (B) at (-3.8,-2);
				\coordinate (C) at (-3.9,-2);
				\draw[dashed] (C) -- (A);
				\draw (A) -- (B);
				\coordinate (A) at (-3.7,-1.9);
				\coordinate (B) at (-3.6,-2);
				\coordinate (C) at (-3.8,-2);
				\draw[dashed] (C) -- (A);
				\draw (A) -- (B);
				\coordinate (A) at (-3.4,-1.8);
				\coordinate (B) at (-3.2,-2);
				\coordinate (C) at (-3.6,-2);
				\draw[dashed] (C) -- (A);
				\draw (A) -- (B);
				\coordinate (A) at (-2.8,-1.6);
				\coordinate (B) at (-2.4,-2);
				\coordinate (C) at (-3.2,-2);
				\draw[dashed] (C) -- (A);
				\draw (A) -- (B);
				\coordinate (A) at (-1.8,-1.4);
				\coordinate (B) at (-1.2,-2);
				\coordinate (C) at (-2.4,-2);
				\draw[dashed] (C) -- (A);
				\draw (A) -- (B);
				\coordinate (A) at (-0.6,-1.4);
				\coordinate (B) at (0,-2);
				\coordinate (C) at (-1.2,-2);
				\draw[dashed] (C) -- (A);
				\draw (A) -- (B);
				\filldraw[fill=white] (-0.65,-1.45) -- (-0.65,-1.35) -- (-0.55,-1.35) -- (-0.55,-1.45) -- cycle;
				\filldraw[fill=white] (0.05,-1.95) -- (0.05,-2.05) -- (-0.05,-2.05) -- (-0.05,-1.95) -- cycle;
				\coordinate (A) at (0.6,-1.4);
				\coordinate (B) at (1.2,-2);
				\coordinate (C) at (0,-2);
				\draw[dashed] (C) -- (A);
				\draw (A) -- (B);
				\filldraw[fill=white] (0.65,-1.45) -- (0.65,-1.35) -- (0.55,-1.35) -- (0.55,-1.45) -- cycle;
				\filldraw[fill=white] (1.25,-1.95) -- (1.25,-2.05) -- (1.15,-2.05) -- (1.15,-1.95) -- cycle;
				\coordinate (A) at (1.8,-1.4);
				\coordinate (B) at (2.4,-2);
				\coordinate (C) at (1.2,-2);
				\draw[dashed] (C) -- (A);
				\draw (A) -- (B);
				\filldraw[fill=white] (1.85,-1.45) -- (1.85,-1.35) -- (1.75,-1.35) -- (1.75,-1.45) -- cycle;
				\filldraw[fill=white] (2.45,-1.95) -- (2.45,-2.05) -- (2.35,-2.05) -- (2.35,-1.95) -- cycle;
				\coordinate (A) at (2.8,-1.6);
				\coordinate (B) at (3.2,-2);
				\coordinate (C) at (2.4,-2);
				\draw[dashed] (C) -- (A);
				\draw (A) -- (B);
				\coordinate (A) at (3.4,-1.8);
				\coordinate (B) at (3.6,-2);
				\coordinate (C) at (3.2,-2);
				\draw[dashed] (C) -- (A);
				\draw (A) -- (B);
				\coordinate (A) at (3.7,-1.9);
				\coordinate (B) at (3.8,-2);
				\coordinate (C) at (3.6,-2);
				\draw[dashed] (C) -- (A);
				\draw (A) -- (B);
				\coordinate (A) at (3.85,-1.95);
				\coordinate (B) at (3.9,-2);
				\coordinate (C) at (3.8,-2);
				\draw[dashed] (C) -- (A);
				\draw (A) -- (B);
			\end{tikzpicture}
			\caption{The indecomposable objects of $\mathcal T_1$ drawn as filled points in $\Ind(\mathcal D)\subsetneq \R^2$.
			The open circles are the described limit points, but are not objects in $\mathcal D$.
			Additionally, $E_i$ and $F_i$ are indicated by squares, for $i\in\{-1,0,1\}$.
			}\label{fig:T1}
		\end{center}
		\end{figure}
		
		As it can be seen via the dashed lines in Figure \ref{fig:T1}, from the construction it follows that each of the open line segments from $E_i$ to $F_i$ form the right side of an isosceles triangle. 
		This triangle is similar to the one defined by $\mathcal Z_+$, $\mathcal Z_+[1]$, and part of $y=-\frac{\pi}{2}$.
		Inside each of the smaller triangles, we repeat the construction of $\mathcal T_1$, except scaled and reflected about the vertical axis of symmetry in the triangle.
		This produces $\mathcal T_2$, with its own set of smaller triangles.
		
		We repeat the above process on the smaller and smaller triangles obtained.
		The limit of this process, which we denote by $\mathcal T_\infty$, contains all the discrete points from $\mathcal T_n$, for all $n\geq 1$.
		In the limit, the line segments between the $E_i$'s and $F_i$'s vanish
		and we are left with a totally disconnected (indeed discrete) set.
		
		Finally, note that the construction guarantees that each $\mathcal T_n$ is a maximally compatible set in $\mathcal{C}_{\mathcal{Z}_+}$. 
		\begin{proposition}\label{prop:disconnected is a T-cluster}
			With the same notation as above, the set $\mathcal T_\infty$ is a $\TT$-cluster.
		\end{proposition}
		\begin{proof}
			Let $W$ be an indecomposable in $\mathcal{C}_{\mathcal{Z}_+}$ such that $W$ is not in $\mathcal T_\infty$.
			It is straightforward to check that if $W$ is not in $\mathcal T_n$ for all $n\geq 0$, then there exists $X\in\mathcal T_\infty$ such that $W$ and $X$ are incompatible.
			
			Suppose $W$ is in $\mathcal T_n$ for some $n\geq 0$.
			Without loss of generality, suppose $W$ is in $\mathcal T_n$ but not $\mathcal T_{n+1}$.
			Then $W$ was on a line segment obtained from the construction of $\mathcal T_n$.
			Further, there is a sub-segment of a line segment in $\mathcal T_{n+1}$, whose complement is also a sub-segment, such that $W$ is incompatible with each $X$ in the sub-segment (see Figure~\ref{fig:incompatible with W}).
			Up to symmetry, we have the following picture which contains only the relevant part of $\mathcal T_{n+1}$.
			\begin{figure}[h]
			\begin{center}
				\begin{tikzpicture}
					\draw[dotted] (-4,-2) -- (4,-2);
					\draw[dashed, draw opacity=.4] (-4,-2) -- (0,2) -- (4,-2);
					
					\foreach \x in {-2,2}
						\filldraw[fill=white](\x-2,\x) circle[radius=.4mm];
					\foreach \x in {-1.95, -1.9, -1.8, -1.6, -1.2, -0.6, 0, 0.6, 1.2, 1.6, 1.8, 1.9, 1.95}
					{
						\filldraw (\x-2, \x) circle[radius=.4mm];
					}
					
					\coordinate (A) at (-3.85,-1.95);
					\coordinate (B) at (-3.8,-2);
					\coordinate (C) at (-3.9,-2);
					\draw[dashed] (C) -- (A);
					\draw (A) -- (B);
					\coordinate (A) at (-3.7,-1.9);
					\coordinate (B) at (-3.6,-2);
					\coordinate (C) at (-3.8,-2);
					\draw[dashed] (C) -- (A);
					\draw (A) -- (B);
					\coordinate (A) at (-3.4,-1.8);
					\coordinate (B) at (-3.2,-2);
					\coordinate (C) at (-3.6,-2);
					\draw[dashed] (C) -- (A);
					\draw (A) -- (B);
					\coordinate (A) at (-2.8,-1.6);
					\coordinate (B) at (-2.4,-2);
					\coordinate (C) at (-3.2,-2);
					\draw[dashed] (C) -- (A);
					\draw (A) -- (B);
					\coordinate (A) at (-1.8,-1.4);
					\coordinate (B) at (-1.2,-2);
					\coordinate (C) at (-2.4,-2);
					\draw[dashed] (C) -- (A);
					\draw (A) -- (B);
					\coordinate (A) at (-0.6,-1.4);
					\coordinate (B) at (0,-2);
					\coordinate (C) at (-1.2,-2);
					\draw[dashed] (C) -- (A);
					\draw (A) -- (B);
					\coordinate (A) at (0.6,-1.4);
					\coordinate (B) at (1.2,-2);
					\coordinate (C) at (0,-2);
					\draw[dashed] (C) -- (A);
					\draw (A) -- (B);
					\coordinate (A) at (1.8,-1.4);
					\coordinate (B) at (2.4,-2);
					\coordinate (C) at (1.2,-2);
					\draw[dashed] (C) -- (A);
					\draw (A) -- (B);
					\coordinate (A) at (2.8,-1.6);
					\coordinate (B) at (3.2,-2);
					\coordinate (C) at (2.4,-2);
					\draw[dashed] (C) -- (A);
					\draw (A) -- (B);
					\coordinate (A) at (3.4,-1.8);
					\coordinate (B) at (3.6,-2);
					\coordinate (C) at (3.2,-2);
					\draw[dashed] (C) -- (A);
					\draw (A) -- (B);
					\coordinate (A) at (3.7,-1.9);
					\coordinate (B) at (3.8,-2);
					\coordinate (C) at (3.6,-2);
					\draw[dashed] (C) -- (A);
					\draw (A) -- (B);
					\coordinate (A) at (3.85,-1.95);
					\coordinate (B) at (3.9,-2);
					\coordinate (C) at (3.8,-2);
					\draw[dashed] (C) -- (A);
					\draw (A) -- (B);
					
					\filldraw[fill=blue, draw=blue] (-2.3,-.3) circle[radius=.4mm];
					\draw (-2.3,-.3) node [anchor=south east] {\textcolor{blue}{$W$}};
					
					\filldraw[fill=blue, fill opacity=.3, draw opacity=0] (-2.3,-.3) -- (-.6, -2) -- (1.7,.3) -- (0,2) -- (-2.3,-.3);
					\draw (-.5,.2) node [rotate=45] {\textcolor{blue}{Incompatible}};
					\draw (-.1, -.2) node [rotate=45] {\textcolor{blue}{with $W$}};
					\draw[red, thick] (-.3,-1.7) -- (-.6,-1.4);
					
					\draw (0,-2) node[anchor=north] {};
				\end{tikzpicture}
				\caption{A visual aid to the proof of Proposition~\ref{prop:disconnected is a T-cluster}. There are infinitely many objects in both the red line segment and $\mathcal T_\infty$.}\label{fig:incompatible with W}
				\end{center}
			\end{figure}
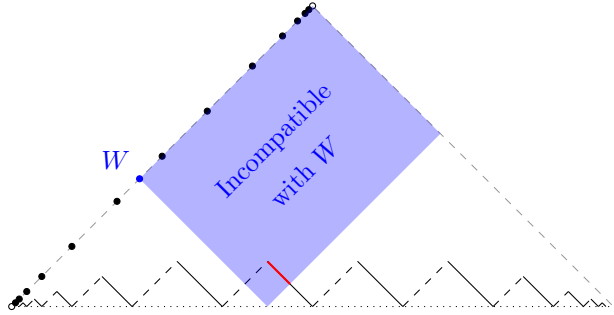
			An infinite sequence of points used to construct $\mathcal T_{n+2}$ approaches the upper end of the red line segment.
			Then there exist infinitely many such $X_i$ in the sequence that are incompatible with $W$.
			Each of the $X_i$ in the sequence is in $\mathcal T_{n+2}$ and in $\mathcal T_\infty$.
			This concludes the proof.
		\end{proof}
		
		\subsection{Mutation of \textbf{T}-clusters}\label{sec:algebraic mutation}
		We now describe the notion of mutation for $\TT$-clusters, which we call $\TT$-mutation (Definition \ref{def:algebra mutation}).
		Again, the reason for the `$\TT$-' prefix is to indicate that our notion of mutation is based on the compatibility condition, which itself relies on tilting rectangles.
		Before we provide further details, we warn the reader that unlike in the standard setting of cluster algebras, for an arbitrary object $X$ in a $\TT$-cluster $\mathcal T$ there may not be a mutation at $X$.
		
		To describe the notion of $\TT$-mutation, we need the following technical lemma.
		\begin{lemma}\label{lem:minimal tilting rectangle}
			Let $\mathcal T$ be a $\TT$-cluster.
			Suppose $X\in\mathcal T$ and $Y\notin \mathcal T$ such that $(\mathcal T\setminus\{X\})\cup\{Y\}$ is a compatible set.
			Then there exists a tilting rectangle in $\ClosedHeart$ whose left and right corners are $X$ and $Y$ (or $Y$ and $X$) and whose top and bottom corners are in $\mathcal T$.
		Moreover, for an indecomposable $Z$ in the interior or on a side of this tilting rectangle, if $Z$ is not a corner vertex then it is not in $\mathcal T$.
		\end{lemma}
		\begin{proof}
			Because $\mathcal T$ is a $\TT$-cluster and further $(\mathcal T\setminus\{X\})\cup\{Y\}$ is also a compatible set, it follows that $X$ and $Y$ cannot be compatible.
			Hence, there exists a tilting rectangle in $\ClosedHeart$ whose left and right corners are respectively $X$ and $Y$, or dually $Y$ and $X$.
			Since the cases are symmetric, without loss of generality, we only treat the former case.

			Let $T_1$ and $T_2$ respectively denote the top and bottom vertices of the rectangle. We aim to show $T_1$ and $T_2$ belong to $\mathcal{T}$.
			Since $X\not\cong Y$, we must have $T_1\not\cong T_2$.
			For the sake of contradiction, suppose $T_1\notin \mathcal T$ and further $T_1\neq 0$. (The case for $T_2$ is similar.)
			
			Since $T_1\notin \mathcal T$, there must exist $Z$ in $\mathcal T$ such that $T_1$ and $Z$ are incompatible.
			Since $X$ and $T_1$ are joined by a diagonal, we conclude $Z\not\cong X$.
			Similarly, $Z\not\cong Y$.
			In the following diagram, consider the regions labelled $1$, $2$, $3$. (If $T_2=0$, we only have regions $1$ and $3$.)
			
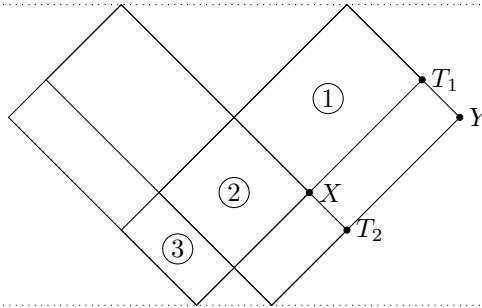
\begin{displaymath}
				\begin{tikzpicture}
					\draw [dotted] (-2,2) -- (6,2);
					\draw [dotted] (-2,-2) -- (6,-2);
					
					\filldraw (3,-0.5) circle[radius=.4mm];
					\draw (3,-0.5) -- (0.5,2) -- (-1,0.5) -- (1.5,-2) -- (3,-0.5);
					\draw (3,-0.5) node[anchor=west] {$X$};
					
					\filldraw (5,0.5) circle[radius=.4mm];
					\draw (5,0.5) -- (3.5,2) -- (1,-0.5) -- (2.5,-2) -- (5,0.5);
					\draw (5,0.5) node[anchor=west] {$Y$};
					
					\filldraw (3.5,-1) circle[radius=.4mm];
					\draw (3.5,-1) -- (0.5,2) -- (-0.5,1) -- (2.5,-2) -- (3.5,-1);
					\draw (3.5,-1) node[anchor=west] {$T_2$};
					
					\filldraw (4.5,1) circle[radius=.4mm];
					\draw (1.5,-2) -- (4.5,1) -- (3.5,2) -- (0.5,-1) -- (1.5,-2);
					\draw (4.5,1) node[anchor=west] {$T_1$};
					
					\coordinate (A) at (3.25,0.75);
					\coordinate (B) at (2,-0.5);
					\coordinate (C) at (1.25,-1.25);
					\foreach \x/\xtext in {A/1, B/2, C/3}
					{
						\draw (\x) node {\xtext};
						\draw (\x) circle[radius=.2];
					}
					
				\end{tikzpicture}
			\end{displaymath}
			
		We show that $Z$ cannot be to the left of $T_1$. This is because if $Z$ is to the left of $T_1$, it must be in regions $1$, $2$, or $3$. But, $Y$ is compatible with all objects in $\mathcal T \setminus \{X\}$, and therefore $Z$ cannot be in regions $1$ or $2$.
		Further, $Z$ cannot be in regions $2$ or $3$, because $Z$ is compatible with $X$.
		
		Similarly, one can show that $Z$ cannot be to the right of $T_1$. This gives the desired contradiction and implies that $T_1$ belongs to $\mathcal T$.
		
		To prove the last assertion of the lemma, suppose $Z$ is on a side of the tilting rectangle but is not one of the corners. 
		Then one may check $Z$ is incompatible with either $X$ or $Y$.
		If $Z$ is in the interior of the tilting rectangle, it is compatible neither with $X$ nor with $Y$.
		Thus, $Z\notin\mathcal T$.
		\end{proof}
		
		Intuitively, if we may take out $X$ and replace it with $Y$, then there is a tilting rectangle whose top and bottom corners are in $\mathcal T$ and nothing ``between" the corners may be in $\mathcal T$.
		We use the preceding lemma and this intuition to prove the following proposition. In particular, the proposition implies that for a given $\TT$-cluster $\mathcal{T}$ and a given indecomposable object $X$ in it, if the mutation at $X$ is possible, then it is unique.
		
		\begin{proposition}\label{prop:mutation yields a unique cluster}
			Let $\mathcal T$ be a $\TT$-cluster, and let $X$ be in $\mathcal T$ such that there exists $Y\notin \mathcal T$ where $(\mathcal T\setminus\{X\})\cup\{Y\}$ is a compatible set.
			Then the following are true.
			\begin{enumerate}
				\item The set $\mathcal T':= (\mathcal T\setminus\{X\})\cup\{Y\}$ is a $\TT$-cluster.
				\item If there exists $Y'$ such that $(\mathcal T\setminus\{X\})\cup\{Y'\}$ is also a compatible set, then $Y= Y'$.
			\end{enumerate}
		\end{proposition}
		\begin{proof}
			We start with the first statement.
			If the statement fails, there is an object $Z$ in $\ClosedHeart$ such that $Z$ is compatible with $\mathcal T'$ but not with $\mathcal T$.
			This implies that $Z$ is incompatible with $X$ but it is compatible with $Y$.
			By Lemma \ref{lem:minimal tilting rectangle}, there exists a tilting rectangle $X T_1 Y T_2$ (or $Y T_1 X T_2$) where $T_1$ and $T_2$ are in $\mathcal T$.
			By the same lemma, if $M$ is a point in the tilting rectangle which is not on a corner, then $M$ is incompatible either with $X$ or with $Y$.
			
			Without loss of generality (by symmetry), assume the tilting rectangle is $X T_1 Y T_2$.
			Consider the regions labelled $1$ through $7$ in the following picture.
			
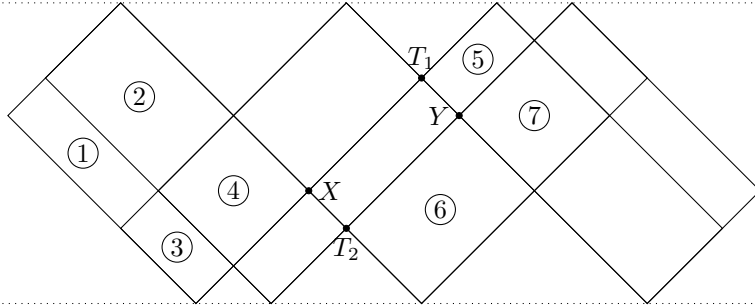
\begin{displaymath}
				\begin{tikzpicture}
					\draw [dotted] (-2,2) -- (10,2);
					\draw [dotted] (-2,-2) -- (10,-2);
					
					\filldraw (3,-0.5) circle[radius=.4mm];
					\draw (3,-0.5) -- (0.5,2) -- (-1,0.5) -- (1.5,-2) -- (3,-0.5) -- (5.5,2) -- (7,0.5) -- (4.5,-2) -- (3,-0.5);
					\draw (3,-0.5) node[anchor=west] {$X$};
					
					\filldraw (5,0.5) circle[radius=.4mm];
					\draw (5,0.5) -- (3.5,2) -- (1,-0.5) -- (2.5,-2) -- (5,0.5) -- (6.5,2) -- (9,-0.5) -- (7.5,-2) -- (5,0.5);
					\draw (5,0.5) node[anchor=east] {$Y$};
					
					\filldraw (3.5,-1) circle[radius=.4mm];
					\draw (3.5,-1) -- (0.5, 2) -- (-0.5,1) -- (2.5,-2) -- (3.5,-1) -- (6.5,2) -- (7.5,1) -- (4.5,-2) -- (3.5,-1);
					\draw (3.5,-1) node[anchor=north] {$T_2$};
					
					\filldraw (4.5,1) circle[radius=.4mm];
					\draw (4.5,1) -- (1.5,-2) -- (0.5,-1) -- (3.5,2) -- (4.5,1) -- (7.5,-2) -- (8.5,-1) -- (5.5,2) -- (4.5,1);
					\draw (4.5,1) node[anchor=south] {$T_1$};
					
					\coordinate (A) at (1.25,-1.25);
					\coordinate (O) at (0,0);
					\coordinate (B) at (.75,0.75);
					\coordinate (C) at (2,-0.5);
					\coordinate (F) at (4.75,-0.75);
					\coordinate (E) at (6,0.5);
					\coordinate (D) at (5.25,1.25);
					\foreach \x/\xtext in {A/3, O/1, B/2, C/4, F/6, E/7, D/5}
					{
						\draw (\x) node {\xtext};
						\draw (\x) circle[radius=.2];
					}
				\end{tikzpicture}
			\end{displaymath}
			Note that $Z$ must be in one of the labelled regions, or else $Z$ is compatible with $X$ or is incompatible with $Y$.
			First, observe that $Z$ cannot be in region $1$, because then $Z$ would be to the left of $Y[-1]$ and thus not in $\ClosedHeart$.
			Next, $Z$ cannot be in regions $2$ or $3$, because $Z$ would be incompatible with $T_2$ or $T_1$, respectively.
			Further, $Z$ cannot be in region $4$, because otherwise $Z$ would be incompatible with both $T_1$ and $T_2$.
			
			Now on the right, $Z$ cannot be in regions 5 or 6, since then $Z$ would be incompatible with $T_1$ or $T_2$, respectively.
			Finally, $Z$ cannot be in region $7$, since $Z$ would not be compatible with both $T_1$ and $T_2$.
			Thus, if $Z$ is not in $\mathcal T$ and $Z\neq Y$, then $Z$ is not in $\mathcal T'$.
			
			For (2), consider a $Y'$ such that $(\mathcal T\setminus\{X\})\cup\{Y'\}$ is a compatible set. 
			We see that $Y'$ cannot be in any of the labelled regions but must be incompatible with $X$.
			By Lemma \ref{lem:minimal tilting rectangle} again, $Y'$ cannot be on the interior or sides (without corners) of the tilting rectangle $XT_1YT_2$.
			Thus $Y'=Y$.
			\end{proof}
			
		\begin{definition}\label{def:algebra mutation}
			Let $\mathcal T$ be a $\TT$-cluster. Suppose $X\in\mathcal T$ and $Y\notin\mathcal T$ such that $\mathcal{T}':=(\mathcal T\setminus\{X\})\cup\{Y\}$ is also a $\TT$-cluster.
			Define $\mu:\mathcal T\to\mathcal T'$ by
			\begin{displaymath}
				\mu(T)= \begin{cases}
 					Y & T=X \\
 					T & \text{otherwise},
 				\end{cases}
			\end{displaymath}
			and call it the $\TT$\textdef{-mutation} of $\mathcal{T}$ at $X$.
		\end{definition}
		
		\begin{remark}\label{rmk:cluster theory}
			Proposition \ref{prop:mutation yields a unique cluster} asserts that the $\TT$-clusters and $\TT$-mutations yield a cluster theory in the sense of \cite[Definition 5.1.1]{IRT22}.
			The primary difference between a cluster structure as in \cite{BIRS09} and a cluster theory in \cite{IRT22} is that, for a cluster theory, we do not require that every object be mutable.
			The uniqueness of the mutation, if it exists, is required by both cluster structures and cluster theories.
		\end{remark}
		
		\begin{example}\label{xmp:mutation}
			Suppose $\mathcal Z$ has at least $3$ line segments, and let $\mathcal T=\mathcal Z$.
			Also, let $\ell_i$ be a line segment of $\mathcal{Z}$ and $X$ be an interior point of $\ell_i$. 
			Then any nondegenerate rectangle $\Diamond$ in $\mathcal C_\mathcal{Z}$ whose left corner is $X$ has a side that intersects $\mathcal T=\mathcal Z$ at infinitely many points.
			By Lemma \ref{lem:minimal tilting rectangle},
			there is no $\TT$-mutation $\mathcal T\to \mathcal{T}'$ at $X$ with $\mathcal T'\neq \mathcal T$.
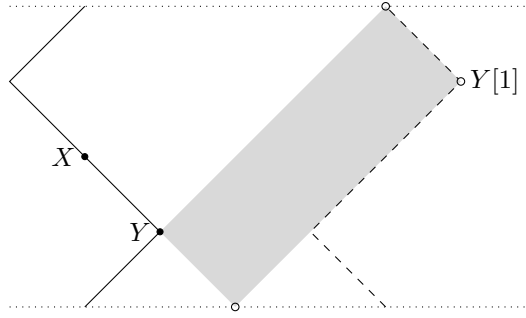
			\begin{displaymath}
				\begin{tikzpicture}
					\draw[dotted] (-1,2) -- (6,2);
					\draw[dotted] (-1,-2) -- (6,-2);
					\draw (0,-2) -- (1,-1) -- (-1,1) -- (0,2);
					\draw[dashed] (4,-2) -- (3,-1) -- (5,1) -- (4,2);
					\filldraw (1,-1) circle[radius=.4mm];
					\filldraw (0,0) circle[radius=.4mm];
					\filldraw[fill=white] (5,1) circle[radius=.5mm];
					\filldraw[fill=white] (2,-2) circle[radius=.5mm];
					\filldraw[fill=white] (4,2) circle[radius=.5mm];
					\draw (0,0) node[anchor=east] {$X$};
					\draw (1,-1) node[anchor=east] {$Y$};
					\draw (5,1) node[anchor=west] {$Y[1]$};
					\filldraw[fill opacity=.15, draw opacity=0] (1,-1) -- (2,-2) -- (5,1) -- (4,2) -- (1,-1);
				\end{tikzpicture}
			\end{displaymath}
			In contrast, choose $Y$ to be the intersection of $\ell_i$ and $\ell_{i+1}$, where $\ell_i$ has slope $-1$.
			(By our assumption on $\mathcal Z$, there exists such a $Y$.)
			Then, we can find a $\TT$-mutation of $\mathcal T$ at $Y$, given by $\mathcal{T}':=(\mathcal T\setminus\{Y\})\cup\{Y[1]\}$.
			For a given zigzag $\mathcal{Z}$, all the possible ways to $\TT$-mutate $\mathcal Z$ are discussed in Section \ref{sec:mutation on associahedron}.
		\end{example}
		
	\section{The associahedron}\label{sec:associahedron}
		Throughout this section, $\mathcal{Z}$ denotes a zigzag in $\mathcal{D}$ (Definition \ref{def:zigzag}), and by $\ClosedHeart$ we denote the associated full subcategory in $\mathcal D$ introduced in Notation \ref{note:closed heart}.
		Moreover, we fix a permissible function $\intc:\Ind(\mathcal D)\sqcup\{0\}\to \R$ (Definition \ref{def:permissible function}) such that $\intc(X)>0$ for all indecomposable objects $X$ in $\ClosedHeart$.  Some of the fundamental steps of our construction of the continuous associahedron in this section are inspired by an analogous approach developed in \cite{B-MDMTY24}, where the problem is treated for simply laced finite quivers in the discrete setting. Hence, to observe the similarities and differences, we encourage the reader to consult Section \ref{sec:Amplituhedron for Dynkin Quivers} and the references therein, where some key ingredients of \cite{B-MDMTY24} and some work preceding that reference are summarized.
		\subsection{Solutions}\label{sec:solutions}
	    As recalled in Section \ref{sec:Amplituhedron for Dynkin Quivers}, in the finite setting there is a correspondence between the clusters of the associated cluster algebra and the solutions of certain systems 
	    of equations arising from the mesh relations, as discussed in  \cite{B-MDMTY24}. 
		To employ this idea in the continuous setting, we first generalize the notion of a solution with respect to a system of linear equations induced by the continuous deformed mesh relations.
		Then, we connect $\TT$-clusters to such solutions and show that for a solution $\Phi$ and each indecomposable $X$ in $\ClosedHeart$, the possible values of $\Phi(X)$ are bounded.
		
		\begin{definition}\label{def:solution}
			With the notations as above, a function $\Phi:\Ind(\mathcal D)\sqcup\{0\}\to \R$ is called a \textdef{solution with respect to $\intc$} in $\ClosedHeart$ if it satisfies the following conditions.
			\begin{itemize}
				\item $\Phi$ satisfies the continuous deformed mesh relations over $\intc$ (Definition \ref{def:contintuous deformed mesh relations}),
				and
				\item $\Phi(X)\geq 0$ for each indecomposable $X$ in $\ClosedHeart$.
			\end{itemize}
		\end{definition}
		Since our ultimate goal is to relate these solutions to $\TT$-clusters,
		we are only interested in $\Phi|_{\Ind(\ClosedHeart)}$.
		However, by Proposition \ref{prop:all the values}, one can extend any such restricted function to all of $\Ind(\mathcal D)\sqcup\{0\}$.
		We note that this extension is
		unique, thus two solutions $\Phi_1$ and $\Phi_2$ are the same if and only if they are the same on $\ClosedHeart$.
		
		\begin{proposition}\label{prop:cluster solutions}
		Let $\Phi:\Ind(\mathcal{D})\sqcup\{0\}\to \R$ be a solution with respect to $\intc$ in $\ClosedHeart$.
		If there exists a $\TT$-cluster $\mathcal T$ such that $\Phi(X)=0$ for all $X\in\mathcal T$, then $\Phi(X)>0$ for all $X\notin \mathcal T$.
		\end{proposition}
		\begin{proof}
			Let $W$ be an indecomposable in $\ClosedHeart$ such that $W$ is not in $\mathcal T$.
			Using the same notation as in Definition \ref{def:tilting rectangle}, there exists $X\in \mathcal T$ and a tilting rectangle $\Diamond=XYWZ$ or $\Diamond=WYXZ$ contained in $\ClosedHeart$.
			In particular, since $Y$ and $Z$ are respectively the top and bottom vertices in either rectangle,
			and because $\Phi$ satisfies the continuous deformed mesh relations over $\intc$, we have
			\begin{displaymath}
				\Phi(X) + \Phi(W) = \Phi(Y) + \Phi(Z) + \int_{\Diamond} \intc.
			\end{displaymath}
			The right side of the equation is positive since $\intc$ takes positive values and $\Diamond$ is nondegenerate.
		\end{proof}
		
		Recall that any zigzag contained in $\ClosedHeart$ is a $\TT$-cluster (Example \ref{xmp:universal examples}).
		By Proposition \ref{prop:all the values}, this means that all such clusters have a unique solution with respect to $\intc$ in $\ClosedHeart$.
		The study of arbitrary $\TT$-clusters is more complex.
		In particular, there are totally disconnected $\TT$-clusters, such as the example in Section \ref{sec:totally disconnected}.
		In $\mathcal T_1$ (Figure \ref{fig:T1}) from the same section, one may instead fill the smaller triangles on the bottom with a scaled down version of any cluster that does not intersect the top point of the triangle or the left side.
		Furthermore, if one begins with a different zigzag, perhaps with many line segments, there are even more possible constructions.

		\begin{question}\label{quest:open problem}
		    For each $\TT$-cluster $\mathcal T$, does there exist a solution $\Phi$ with respect to $\intc$ in $\ClosedHeart$ such that $\Phi(X)=0$ for all $X \in\mathcal T$?
		    For which $\TT$-clusters is such a solution unique (when it exists)?
		\end{question}
		
		Let $\gZ$-vectors (Definition \ref{def:g-vector}) and the rays $\mathfrak u$ and $\mathfrak d$ be as defined in Section \ref{sec:g-vectors}.
		Let $X$ be an indecomposable object in $\ClosedHeart$ which is neither in $\mathcal Z$  nor in $\mathcal Z[1]$. Suppose $GX$ is the region in $\ClosedHeart$ bounded by $\mathcal Z$, $\mathfrak u$, and $\mathfrak d$.
		The reader may refer back to Figure \ref{fig:g-vector} for a depiction of such a $GX$.
		
		\begin{theorem}\label{thm:bounded}
			Let $X$ be an indecomposable object in $\ClosedHeart$.
			There exists a bound $\xi_X\in \R_{>0}$ such that $\Phi(X) \leq \xi_X$, for all solutions $\Phi$ with respect to $\intc$ in $\ClosedHeart$.
		\end{theorem}
		\begin{proof}
			Let $\Phi$ be a solution with respect to $\intc$ in $\ClosedHeart$.
			We first show that the possible values of the objects in $\mathcal Z$ and $\mathcal Z[1]$ are bounded.
			
			Let $X\in\mathcal Z$.
			Let $\Diamond$ be the tilting rectangle in $\ClosedHeart$ whose left and right corners are $X$ and $X[1]$, respectively.
			Note that $\Phi(X[1]) \geq 0$, and therefore the maximum possible value of $\Phi(X)$ is $\int_{\Diamond}\intc$.
			Similarly, $\Phi(X[1]) \leq \int_{\Diamond}\intc$.
			
			Now let $X$ be an indecomposable not in $\mathcal Z$ and not in $\mathcal Z[1]$.
			Let $Z_0,\ldots, Z_n$ be the objects in $\mathcal Z$ corresponding to the nonnegative coordinates of $g_{\mathcal Z}(X)$.
			Using a similar argument to that in the proof of Proposition \ref{prop:g-vectors}, we see that 
			\begin{displaymath}
				\Phi(X) = \begin{cases}
 					\sum\limits_{\text{odd } i} \Phi(Z_i) - \sum\limits_{\text{even } i} \Phi(Z_i) + GX & Z_0 \text{ is left of } Z_1 \\
 					\sum\limits_{\text{even } i} \Phi(Z_i) - \sum\limits_{\text{odd } i} \Phi(Z_i) + GX & Z_0 \text{ is right of } Z_1.
 				\end{cases}
			\end{displaymath}
			Suppose $Z_0$ is left of $Z_1$.
			The other case is similar.
			Let $\Diamond_i$ be the rectangle in $\ClosedHeart$ whose left and right corners are $Z_i$ and $Z_i[1]$, respectively.
			By using the minimum values of the even $\Phi(Z_i)$'s and the maximum values of the odd $\Phi(Z_i)$'s, it follows that
			\begin{displaymath}
				\Phi (X) \leq GX + \sum\limits_{\text{odd } i} \left(\int_{\Diamond_i} \intc \right).
			\end{displaymath}
            Now, choose $\xi_X$ to be the right side of the displayed inequality.
			This completes the proof.
		\end{proof}

		\subsection{Associahedron}\label{sec:associahedron (subsection)}
		Now we are equipped with the required tools to describe the titular object of the paper: a \textdef{continuous associahedron}.
		First, let us fix the ambient space where this associahedron will be realized. 
		By $\prod_{\Ind(\ClosedHeart)}\R$, we  denote the real vector space whose coordinates are indexed by indecomposable objects in $\ClosedHeart$.
		
		In this subsection, we introduce the continuous associahedron and show that it is a convex object in the ambient space $\prod_{\Ind(\ClosedHeart)}\R$. Further, we prove that the solutions corresponding to $\TT$-clusters are extremal points (that is, on the boundary). 
		To each function $\Phi:\Ind(\mathcal{D})\sqcup\{0\}\to \R$, one can naturally associate a vector in $\prod_{\Ind(\ClosedHeart)} \R$. In particular, for a fixed $\intc$, we are interested in those vectors corresponding to the solutions with respect to $\intc$ in $\ClosedHeart$. 
	
		\begin{definition}\label{def:continuous associahedron}
		For each fixed $\intc$, the \textdef{continuous associahedron} $\UZc$ is the subset of $\prod_{\Ind(\ClosedHeart)}\R$ consisting of those vectors corresponding to the solutions with respect to $\intc$ in $\ClosedHeart$. 
		\end{definition}
		
	\begin{remark}\label{rmk:projection to g-vectors}
	By Proposition~\ref{prop:all the values}, each solution $\Phi$ with respect to $\intc$ in $\ClosedHeart$ is  determined by the values $\Phi$ takes on any zigzag in $\ClosedHeart$.
	In particular, if we know the values of $\Phi$ on $\mathcal Z[1]$ then we may recover the values of $\Phi$ on the rest of $\ClosedHeart$.
	So, one may consider the projection of $\UZc$ onto $\prod_{\Ind(\mathcal Z[1])} \R$ without losing any information.

    As recalled in Section~\ref{sec:Amplituhedron for Dynkin Quivers}, in \cite{B-MDMTY24} the authors give another characterization of $\mathbb A_{\intc}$ in terms of the $g$-vectors of indecomposable modules associated to elements of $\mathcal I$. (See Theorem \ref{th-one} and the paragraphs preceding it.) 
	That is to say, $\mathbb A_{\intc}$ is obtained as the projection of $\mathbb U_{\intc}$ via a map determined by the $g$-vectors of those indecomposable objects associated to elements of $\mathcal I^{[1]}$.
	Consequently, in the setting of Dynkin quivers, they also recover $\mathbb A_{\intc}$ as a polytopal realization of the $g$-vector fans. This approach has been further studied and developed in \cite{PP+23}. In particular, starting from an arbitrary initial seed in any finite type cluster algebra, in \cite[Section 3.4]{PP+23} the authors showed that the mesh mutations are the minimal relations among the $g$-vectors with respect to the chosen seed. Consequently, they elegantly generalized the earlier results of \cite{ABHY18} and \cite{B-MDMTY24} which were shown for acyclic initial seeds. With this observation in mind, we remark that our construction in this work is the continuous analogue of the one given in \cite{B-MDMTY24}, and it is not in the full generality of \cite{PP+23}. 
    Our treatment of zigzags, introduced in Section~\ref{subsec:Quilting}, may be viewed as the continuous analogue of acyclic seeds. It is an interesting problem to study similar problems for continuous cyclic initial seeds.  
    Inspired by such results, it is natural to hope for an analogous realization of the associahedron in the continuous setting. Namely, to study the projection of $\UZc$ onto $\prod_{\Ind(\mathcal Z[1])}\R$
	determined by the $g_{\mathcal Z}$-vectors corresponding to each $Z$ in $\Ind(\mathcal Z[1])$. This direction of work requires further investigation.
	\end{remark}
		
		We recall that a subset $\mathbb X$ of a vector space $\mathbb V$ is said to be \textdef{convex} if for each pair of points $A,B\in\mathbb X$ and every $t\in[0,1]$, the linear combination $t A + (1-t)B$ is in $\mathbb X$.
		
		\begin{theorem}\label{thm:convex}
			The set $\UZc$ is convex in $\prod_{\Ind(\ClosedHeart)}\R$. 
		\end{theorem}
		\begin{proof}
			Let $\Phi_0$ and $\Phi_1$ be solutions to $\intc$ in $\ClosedHeart$.
			For each $t\in(0,1)$, define $\Phi_t$ as
			\begin{displaymath}
				\Phi_t(X) := t \cdot \Phi_1(X) + (1-t)\cdot \Phi_0(X).
			\end{displaymath}
			For each tilting rectangle $\Diamond=XYWZ$ in $\ClosedHeart$, we have the following equations from $\Phi_0$ and $\Phi_1$:
			\begin{align*}
				\Phi_0(X) + \Phi_0(W)&= \Phi_0(Y)+\Phi_0(Z)+\int_{\Diamond} \intc \\
				\Phi_1(X) + \Phi_1(W)&= \Phi_1(Y)+\Phi_1(Z)+\int_{\Diamond} \intc.
			\end{align*}
			Since $\int_{\Diamond} \intc$ is fixed, we have
			\begin{displaymath}
				\Phi_t(X) + \Phi_t(W) = \Phi_t(Y)+\Phi_t(Z) + \int_{\Diamond} \intc.
			\end{displaymath}
			Thus, any line segment in $\prod_{\Ind(\ClosedHeart)}\R$ connecting two points in $\UZc$ is entirely contained in $\UZc$.
		\end{proof}
		
		The previous theorem and its proof allow us to make the following definition.
		
		\begin{definition}\label{def:boundary}
		    Let $\Phi\in\UZc$ be a solution with respect to $\intc$ in $\ClosedHeart$.
		    We say $\Phi$ is \textdef{on the boundary} of $\UZc$ if there exists a line segment $\ell$ parameterized by $t\in[0,1]$ satisfying the following conditions. 
		    \begin{itemize}
		        \item $\Phi=\Phi_1$ and $\Phi_0\in \UZc$ are the distinct endpoints of $\ell$ (and so $\ell\subsetneq \UZc$).
		        \item If $\ell'$ is a parameterized line segment that contains $\ell$ such that $\Phi_1$ is not an endpoint of $\ell'$, then $\ell'$ is not contained in $\UZc$.
		    \end{itemize}
		    We call the set of all such $\Phi$ the \textdef{boundary} of $\UZc$.
		\end{definition}
		
		\begin{theorem}\label{thm:clusters are extremal}
			Let $\mathcal T$ be a $\TT$-cluster and $\Phi$ a solution with respect to $\intc$ in $\ClosedHeart$ such that $\Phi(X)=0$ for all $X\in\mathcal T$.
			Then $\Phi$ is on the boundary of $\UZc$.
		\end{theorem}
		\begin{proof}
		    Let $X\in \mathcal T$, $\Phi_1=\Phi$, and $\mathcal T'$ be any zigzag contained in $\ClosedHeart$ that does not contain $X$.
		    By Proposition \ref{prop:all the values}, there is a unique solution $\Phi_0$ such that $\Phi_0(X')=0$ for all $X'\in \mathcal T'$.
			For all $t\in\R$ and indecomposables $X$ in $\ClosedHeart$, let
			\begin{displaymath}
				\Phi_t(X) := t\cdot \Phi_1(X) + (1-t)\cdot \Phi_0(X).
			\end{displaymath}
			Let $\e>0$ and note $\Phi_{1+\e}(X) = - \e\cdot \Phi_0(X)$.
			Since $X\notin \mathcal T'$, Proposition \ref{prop:cluster solutions} asserts that $\Phi_0(X)>0$.
			Thus, $\Phi_{1+\e}(X)<0$ and so $\Phi_{1+\e}\notin\UZc$.
			Therefore, $\Phi=\Phi_1$ is on the boundary of $\UZc$.
		\end{proof}
	
		\subsection{\textbf{T}-mutation a continuous associahedron}\label{sec:mutation on associahedron}
		Recall that in the finite
		setting (Section \ref{sec:Amplituhedron for Dynkin Quivers}), the vertices of the ABHY associahedron correspond to the clusters, the facets are 
		in bijection with cluster variables, and the mutation of clusters corresponds to the edges which connect two vertices of the associahedron. These edges are given by the intersection of the hyperplanes associated to the mutable cluster variables.
		(For full details, see \cite{B-MDMTY24}.)
		
		In this subsection, we discuss analogous phenomena in the continuous setting. In particular, we describe the connections between $\TT$-mutations and the continuous associahedron $\UZc$.
		Namely, we show that for two different $\TT$-clusters, each of which corresponds to a unique solution in $\UZc$, the $\TT$-mutation can be seen via $\UZc$.
		
		\begin{notation}\label{note:hyperplane} 
		For an indecomposable object $X$ in $\ClosedHeart$, by $\HZc(X)$ we denote the \textdef{hyperplane} in $\prod_{\Ind(\ClosedHeart)}\R$ associated to $X$, which is the set of functions $\Phi:\Ind(\ClosedHeart)\sqcup\{0\}\to \R$ such that $\Phi(X)=0$. 
		\end{notation}
		
		\begin{theorem}\label{thm:mutation on associahedron}
		Let $\mathcal T$ and $\mathcal T'$ be $\TT$-clusters with respective unique solutions $\Phi$ and $\Phi'$ in $\UZc$. 
		The following are equivalent.
			\begin{enumerate}
				\item There is a $\TT$-mutation $\mu:\mathcal T\to\mathcal T'$, where $\mathcal T'=(\mathcal T\setminus \{T_0\})\cup\{T_1\}$, with $T_0\in\mathcal T$ and $T_1\in\mathcal T'$.
				\item There is a line segment $\{\Phi_t \mid t\in[0,1]\}$ 
				which connects $\Phi$ to $\Phi'$ in $\UZc$ such that $\Phi_0=\Phi$, $\Phi_1=\Phi'$, and $\Phi_t(X)=0$, for all $X\in\mathcal T\cap \mathcal T'$. 
			\end{enumerate}
			The line segment in the second part is given by
					\begin{displaymath}
					\left(\bigcap_{X\in\mathcal T\cap \mathcal T'} \HZc(X)\right) \cap \UZc,
					\end{displaymath}
					where
					$\mathcal T\cap \mathcal T'=\mathcal T\setminus \{T_0\}=\mathcal T'\setminus \{T_1\}$. 
		\end{theorem}
		\begin{proof}
		If (2) holds, to get (1) apply Proposition \ref{prop:mutation yields a unique cluster} to the equation $\mathcal T\cap \mathcal T'=\mathcal T\setminus \{T_0\}=\mathcal T'\setminus \{T_1\}$. 
			
			Now assume (1).
			Let $\Phi_0=\Phi$ and $\Phi_1=\Phi'$. Moreover, for all indecomposable $X$ in $\ClosedHeart$ define
			\begin{displaymath}
				\Phi_t(X) = t\cdot \Phi_1(X) + (1-t)\cdot \Phi_0(X),
			\end{displaymath}
			for $t \in [0,1]$.
			Then, for any $X$ in $\mathcal T\cap \mathcal T'$ we have $\Phi_t(X)=0$.
			Since $\UZc$ is convex (Theorem \ref{thm:convex}), we have the desired line segment.
			We know $\Phi$ and $\Phi'$ are unique, so we have the singleton sets:
			\begin{equation*}
				\{\Phi\} = \left(\bigcap_{X\in\mathcal T} \HZc(X)\right) \cap \UZc \qquad \text{and} \qquad
				\{\Phi'\} = \left(\bigcap_{X\in\mathcal T'} \HZc(X)\right) \cap \UZc.
			\end{equation*}
			By Proposition \ref{prop:mutation yields a unique cluster}, the line segment is the desired intersection. 
		\end{proof}
		
		For the next proposition, recall the definition of a right vertex (Definition \ref{def:right vertex}).
		An example of the setting and statement of the proposition can be seen in Figure~\ref{fig:kaveh's picture}.
		\begin{proposition}\label{prop:n dimensional hypercube}
		Suppose $\mathcal Z$ has $n$ right vertices, and let $\mathscr{M}_{\mathcal Z}$ be the set of $\TT$-clusters obtained from $\mathcal Z$ via finitely many $\TT$-mutations. 
		Then $\mathscr M_{\mathcal Z}$ and the edges in $\UZc$ corresponding to the $\TT$-mutations form the $1$-skeleton of an $n$-dimensional hypercube in $\UZc$.
		In particular, $\mathscr M_{\mathcal Z}$ is finite.
		\end{proposition}
		\begin{proof}
			We first note that the only $\TT$-mutable objects in $\mathcal Z$ are the right vertices (see Example \ref{xmp:mutation}).
			For a vertex $X$, consider the mutation $\mathcal Z\to (\mathcal Z\setminus\{X\})\cup\{X[1]\}$.
			Choose a different right vertex $Y$ of $\mathcal Z$.
			We may also mutate
			\begin{displaymath}
				(\mathcal Z\setminus\{X\})\cup\{X[1]\} \to
				\left(\left(\left(\mathcal Z\setminus\{X\}\right)\cup\{X[1]\}\right)\setminus \{Y\}\right)\cup \{Y[1]\}.
			\end{displaymath}
			We note that the mutation of $X$ and $Y$ are independent of each other.
			Further, observe that $X$ and $Y$ were chosen arbitrarily, therefore this argument holds for all $n$ right vertices of $\mathcal Z$.
			
			We then have a bijection from the set of $\TT$-clusters that can be obtained from $\mathcal Z$ in finitely many $\TT$-mutations	to the set $\{0,1\}^n$.
			This is in particular the number of vertices of 
			an $n$-dimensional hypercube.
			The commutativity of the mutations yields the $n$-dimensional hypercube structure.
		\end{proof}
	
	\begin{figure}
	    \centering
	    \begin{tikzpicture}
	       \filldraw (0,0) circle[radius=.4mm];
	       \draw[dashed] (0,2) -- (0,-2);
	       \draw (0,0) -- (1.414,1.414);
	       \draw (0,0) -- (2,0);
	       \draw (0,0) -- (1.414,-1.414);
	       \filldraw[draw opacity=0, fill opacity=.2] (0,2) arc (225:270:2) arc (180:225:2) arc (135:180:2) arc (90:135:2);
	       \draw (0,0) node[anchor=east] {$\mathcal Z$};
	       \draw (1.414,1.414) node[anchor=south west] {$(\mathcal Z\setminus \{X_1\})\cup \{X_1[1]\}$};
	       \draw (2,0) node[anchor=west] {$(\mathcal Z\setminus \{X_2\})\cup \{X_2[1]\}$};
	       \draw (1.414,-1.414) node[anchor=north west] {$(\mathcal Z\setminus \{X_3\})\cup \{X_3[1]\}$};
	    \end{tikzpicture}
	    \qquad
	    \begin{tikzpicture}
	       \draw (1,1) -- (3,1) -- (3,3) -- (1,3) -- (1,1);
	       \draw (0,0) -- (1,1);
	       \draw (0,2) -- (1,3);
	       \draw (2,0) -- (3,1);
	       \draw (2,2) -- (3,3);
	       \draw (0,2) node[anchor=east] {$X_1,X_2,X_3$};
	       \draw (0,0) node[anchor=east] {$X_1,X_2,X_3[1]$};
	       \draw (2,2) node[anchor=west] {$X_1,X_2[1],X_3$};
	       \draw (2,0) node[anchor=west] {$X_1,X_2[1],X_3[1]$};
	       \draw (1,3) node[anchor=east] {$X_1[1],X_2,X_3$};
	       \draw (1,1) node[anchor=east] {$X_1[1],X_2,X_3[1]$};
	       \draw (3,1) node[anchor=west] {$X_1[1],X_2[1],X_3[1]$};
	       \draw (3,3) node[anchor=west] {$X_1[1],X_2[1],X_3$};
	       \filldraw[fill=white, fill opacity=.5] (0,0) -- (0,2) -- (2,2) -- (2,0) -- (0,0);
	    \end{tikzpicture}
	    \caption{We illustrate possible mutations where $\mathcal T=\mathcal Z$ has $3$ right vertices: $X_1, X_2, X_3$ (Definition \ref{def:right vertex}).
	    On the left, we have the web-like effect for mutation at the unique $\Phi$ in $\UZc$ such that $\Phi(X)=0$ for all $X\in \mathcal Z$.
	    The vertices represented by the webbing between the lines are not $\TT$-mutable.
	    On the right, we have the mutation structure starting with a zigzag $\mathcal Z$ with $3$ right vertices.
	    Each facet corresponds to one of $X_1$, $X_2$, $X_3$, $X_1[1]$, $X_2[1]$, and $X_3[1]$, and the vertices are labelled by the facets they belong to.
	    Thus, the vertices are the corresponding $\TT$-clusters.}
	    \label{fig:kaveh's picture}
	\end{figure}
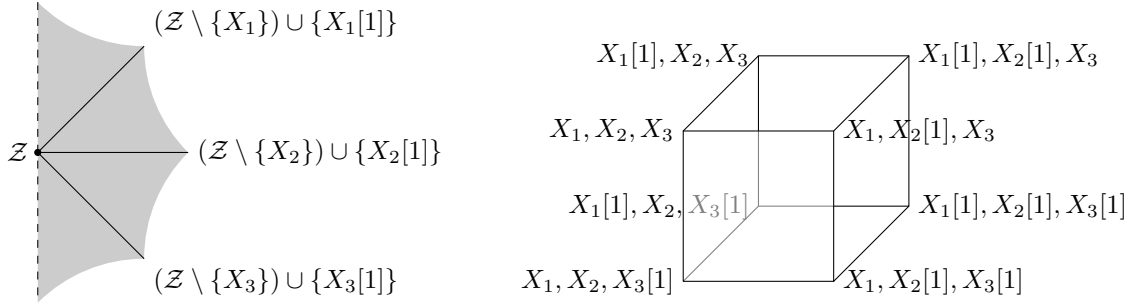

\subsection{Finite embeddings}\label{sec:finite embeddings}
    The main goal of this section is to study the relationship between the cluster structures of type $A_n$ and the $\TT$-clusters. In particular, in Subsection~\ref{sec:continuingxmp} we return to Example~\ref{xmp:universal examples}~(\ref{xmp:universal examples:vertical curve}) to prove a technical lemma (Lemma~\ref{lem:vertical solution}) that we need to deduce Theorem \ref{thm:big finite embedding}.
    In Subsection~\ref{sec:finite embeddings construction},
    we use the results of the first subsection to complete our argument.
	
\subsubsection{Using Example~\ref{xmp:universal examples}~(\ref{xmp:universal examples:vertical curve})}\label{sec:continuingxmp}
    We now show that each $\TT$-cluster $\mathcal T$ as in Example \ref{xmp:universal examples} (\ref{xmp:universal examples:vertical curve}) has a unique solution $\Phi$ in $\UZc$ such that $\Phi(X)=0$ for all $X\in\mathcal T$.
	
    Let $\ell$ be a curve between $\mathcal Z$ and $\mathcal Z[1]$ such that the slope of $\ell$ at each point is greater than $1$, less than $-1$, or equal to $\infty$.
    Additionally, suppose that for all $a\in(-\frac{\pi}{2},\frac{\pi}{2})$, there is an $A\in\ell$ such that the $y$-coordinate of $A$ is $a$.
    Assign $\Phi(A)=0$ for all $A$ on $\ell$.
    
    Now we construct a tilting rectangle that contains part of $\ell$. It may be helpful for the reader to refer to Figure~\ref{fig:vertical computation} while reading this construction.
    Let $X$ be an indecomposable in $\ClosedHeart$ such that $X$ is not on $\ell$ and $X$ is to the left of $\ell$.
    Without loss of generality, we assume that the rays emanating from $X$ at $\pm 45^\circ$ intersect $\ell$.
    Due to our conditions on $\ell$, there is a unique $X'$ not on $\ell$ such that the tilting rectangle with left and right corners $X$ and $X'$ has top and bottom corners on $\ell$.
    We create a sequence of regions $R_1\subsetneq R_2\subsetneq \cdots$ such that \[\Phi(X) = \int_{\lim_{n\to\infty}R_n}\intc.\]
	
    We now choose $X$ such that $X$ and $X'$ above are both in $\ClosedHeart$. 
    Let $\Diamond$ be the tilting rectangle whose left and right corners are $X$ and $X'$, respectively, and whose top and bottom corners are on $\ell$.
    Then
    \begin{displaymath}
	    \Phi(X) + \Phi(X') = \int_{\Diamond} \intc.
    \end{displaymath}
    We subdivide the rectangle with objects $Y,Y',Z,Z'$, each distinct from each other and from $X, X'$, such that each of the four smaller rectangles share a corner on $\ell$ (illustrated in Figure \ref{fig:vertical computation}).
    Let $R_1$ be the tilting rectangle whose left corner is $X$, top corner is $Y$, bottom corner is $Z$, and right corner is on $\ell$.
    Then we have
    \begin{displaymath}
	    \Phi(X) = \Phi(Y)+\Phi(Z) +\int_{R_1} \intc.
    \end{displaymath}
    In order for $\Phi$ to be a solution with respect to $\intc$ in $\UZc$,
    we must have $\Phi(X)\geq 0$, $\Phi(Y)\geq 0$, and $\Phi(Z)\geq 0$.
    Thus, we must have $\Phi(X) \geq \int_{R_1} \intc$.
    By a similar process, we create two smaller tilting rectangles whose left corners are $Y$ and $Z$, respectively, and whose right corners are on $\ell$.
    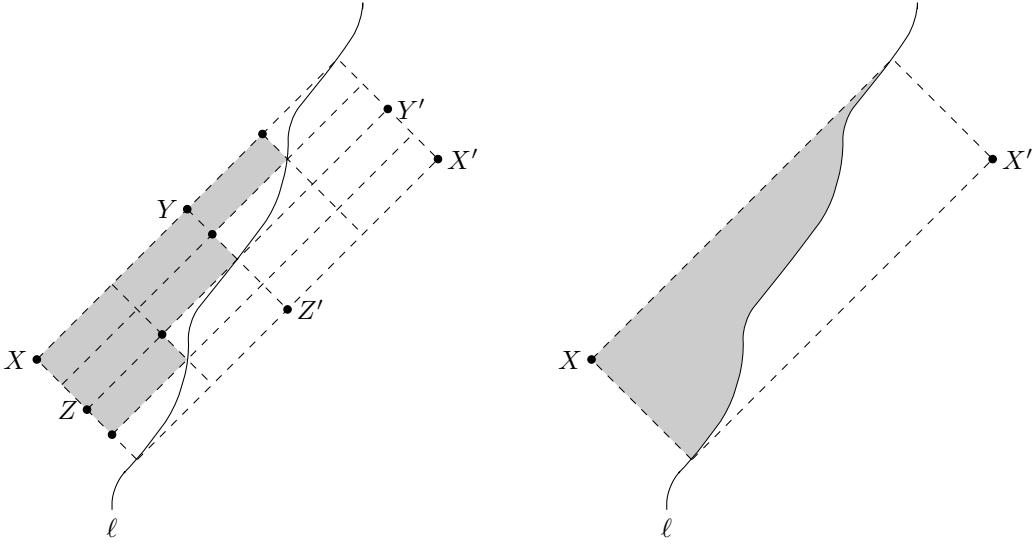
\begin{figure}
    \begin{center}
	\begin{tikzpicture}
		\filldraw[draw opacity=0, fill opacity=.2] (-1,2) -- (1,4) -- (1.666,3.333) -- (-0.333,1.333) -- (-1,2);
		\filldraw[draw opacity=0, fill opacity=.2] (1,4) -- (2,5) -- (2.333,4.666) -- (1.333,3.666) -- (1,4);
		\filldraw[draw opacity=0, fill opacity = .2] (-0.333, 1.333) -- (.666,2.333) -- (1, 2) -- (0,1) -- (-0.333, 1.333);
		\draw[rounded corners = 7] (0,0) -- (0,.333) -- (.333,.666) -- (.833,1.333) -- (1.03,2)--(1,2.5)--(1.666,3.333)--(2.166,4)--(2.35,4.666)--(2.333,5.166)--(3,6)--(3.333,6.5)--(3.333,6.666);
		\filldraw (-1,2) circle[radius=.5mm];
		\filldraw (4.333,4.666) circle[radius=.5mm];
		\filldraw (3.666,5.333) circle[radius=.5mm];
		\filldraw (2.333,2.666) circle[radius=.5mm];
		\filldraw (1,4) circle[radius=.5mm];
		\filldraw (-0.333,1.333) circle[radius=.5mm];
		\filldraw (1.333,3.666) circle[radius=.5mm];
		\filldraw (2,5) circle[radius=.5mm];
		\filldraw (.666,2.333) circle[radius=.5mm];
		\filldraw (0,1) circle[radius=.5mm];
		\draw[dashed] (.333,.666) -- (-1,2) -- (3,6) -- (4.333,4.666) -- (.333,.666);
		\draw[dashed] (1,4) -- (2.333,2.666);
		\draw[dashed] (0,3) -- (1.333,1.666);
		\draw[dashed] (2,5) -- (3.333,3.666);
		\draw[dashed] (-0.666,1.666) -- (3.333,5.666);
		\draw[dashed] (-0.333,1.333) -- (3.666,5.333);
		\draw[dashed] (0,1) -- (4,5);
		\draw (-1,2) node[anchor=east] {$X$};
		\draw (4.333,4.666) node[anchor=west] {$X'$};
		\draw (1,4) node[anchor=east] {$Y$};
		\draw (3.666,5.333) node[anchor=west] {$Y'$};
		\draw (-0.333,1.333) node[anchor=east] {$Z$};
		\draw (2.333,2.666) node[anchor=west] {$Z'$};
		\draw(0,0) node[anchor=north] {$\ell$};
	\end{tikzpicture}
	\qquad
	\begin{tikzpicture}
		\filldraw[draw opacity=0, fill opacity=.2] (.333,.666) -- (-1,2) -- (3,6) -- (4,5) -- (1.333,1.666) -- (.333,.666);
		\filldraw[fill=white, draw=white, rounded corners = 7](0,0) -- (0,.333) -- (.333,.666) -- (.833,1.333) -- (1.03,2)--(1,2.5)--(1.666,3.333)--(2.166,4)--(2.35,4.666)--(2.333,5.166)--(3,6)--(3.333,6.5)--(3.333,6.666) -- (5,6) -- (2,1.666) -- (.333,.666);
		\draw[rounded corners = 7] (0,0) -- (0,.333) -- (.333,.666) -- (.833,1.333) -- (1.03,2)--(1,2.5)--(1.666,3.333)--(2.166,4)--(2.35,4.666)--(2.333,5.166)--(3,6)--(3.333,6.5)--(3.333,6.666);
		\filldraw (-1,2) circle[radius=.5mm];
		\filldraw (4.333,4.666) circle[radius=.5mm];
		\draw[dashed] (.333,.666) -- (-1,2) -- (3,6) -- (4.333,4.666) -- (.333,.666);
		\draw (-1,2) node[anchor=east] {$X$};
		\draw (4.333,4.666) node[anchor=west] {$X'$};
		\draw(0,0) node[anchor=north] {$\ell$};
	\end{tikzpicture}
	\caption{On the left, we have the region $R_2$ obtained in the processs of finding a lower bound of $\Phi(X)$, where $\Phi$ is $0$ on $\ell$.
	On the right, we have the region $R$ such that $\Phi(X)=\int_R \intc $.}\label{fig:vertical computation}
    \end{center}
    \end{figure}
    Let $R_2$ be the union of $R_1$ and the two smaller rectangles.
    If we want to extend $\Phi$ to a solution with respect to $\intc$ in $\ClosedHeart$, we must have $\Phi(X)\geq \int_{R_2}\intc$.

    We continue defining successively larger regions $R_3 \subsetneq R_4 \subsetneq \cdots$ similarly.
    The limit $R=\lim_{n\to\infty} R_n$ is the region inside $\Diamond$ on the left of $\ell$, and we see $\Phi(X)\geq \int_{R}\intc$.
    Denote by $R'$ the region in 
    $\Diamond$ on the right of $\ell$ so that $R\cup R'\cup\ell=\Diamond$.
    Then $\Phi(X')\geq \int_{R'}\intc$ by the same argument. Since $\Phi(X)+\Phi(X')=\int_{\Diamond}\intc$, we must have $\Phi(X)=\int_{R}\intc$ and $\Phi(X')=\int_{R'}\intc$.
    Repeating this argument for each point in $R$ shows that $\Phi$ extends uniquely to all of $\Diamond$. 
    
    \begin{lemma}\label{lem:vertical solution}
    Let $\ell$ be a curve in $\ClosedHeart$ as in
    Example \ref{xmp:universal examples} (\ref{xmp:universal examples:vertical curve}). 
    Then there is a unique solution $\Phi$ with respect to $\intc$
    in $\ClosedHeart$ such that $\Phi(A)=0$, for all $A\in\ell$.
    \end{lemma}
    \begin{proof}
        Note that $\ell\subsetneq\Ind(\ClosedHeart)$.
        Since $\mathcal Z$ and $\mathcal Z[1]$ have finitely many line segments, we find finitely many tilting rectangles $\{\Diamond_i\}_{i=1}^n$ in $\Ind(\ClosedHeart)$ with the following three properties.
        First, $\ell\subsetneq\bigcup_{i=1}^n\Diamond_i$.
        Second, for $1\leq i<n$, the bottom corner of $\Diamond_i$ is the top corner of $\Diamond_{i+1}$. Third, top corner of $\Diamond_1$ and bottom corner of $\Diamond_n$ have $y$-coordinates $\frac{\pi}{2}$ and $-\frac{\pi}{2}$, respectively.
        By the construction preceding the lemma, we extend $\Phi$ uniquely on each of these $\Diamond_i$'s.
        The left sides of all $\Diamond_i$'s form a zigzag in $\Ind(\ClosedHeart)$.
        By Proposition \ref{prop:all the values}, there is a unique extension of $\Phi$ as stated in the proposition.
    \end{proof}

\subsubsection{Construction of the embeddings}\label{sec:finite embeddings construction}
In this subsection we employ our argument in Section~\ref{sec:continuingxmp} to prove Theorem~\ref{thm:big finite embedding}.
Throughout the section, we use our notation from Section \ref{sec:Amplituhedron for Dynkin Quivers}.
For the theorem, we consider the ascending linear orientation of $A_n$, for $n\geq 2$, meaning that the $A_n$ quiver is linearly ordered and $P_i\hookrightarrow P_j$ if $i>j$.
    There is an exact embedding $\rep(A_2)\to \rep(A_3)$ determined by sending $P_1\mapsto P_1$ and $P_2\mapsto P_2$.
	In general, there is an exact embedding $\rep (A_n)\to \rep(A_{n+1})$ by sending $P_i\mapsto P_i$, for $1\leq i\leq n$.
    Also, for simplicity, we work with the zigzag $\mathcal Z_+$ (initially introduced in Section~\ref{sec:totally disconnected}) consisting of exactly one line segment, which has slope $+1$.
	Otherwise, the computations become exceedingly involved.
		
	We construct a $\TT$-cluster $\mathcal T_U$ (Figure \ref{fig:embedding cluster}) similar to $\mathcal T_1$ from Section \ref{sec:totally disconnected} (see Figure \ref{fig:T1}).  Before beginning, we note here that although the construction of $\mathcal{T}_U$ depends on $\mathcal{Z}_+$, we avoid cumbersome notation by suppressing mention of $\mathcal{Z}_+$ in the notation $\mathcal{T}_U$.
	We consider $\{X_i=(x_i,y_i)\mid i\in\Z_{<0}\}$,
	as a sequence of indecomposables in $\mathcal Z_+$ such that $j<i$ implies $y_j<y_i$ and $\lim_{j\to-\infty} y_j=-\frac{\pi}{2}$.
	Let $(x_0,\frac{\pi}{2})$ be the top boundary point of $\mathcal Z_+$.
	    For $i<0$, we define $E_i$ as in Section \ref{sec:totally disconnected} and a vertical line segment $\ell_i$:
	    \begin{align*}
	        E_i &:= \left( x_{i+1} + (x_i+\frac{\pi}{2}),\, y_{i+1}-(y_i+\frac{\pi}{2})\right), \\
	        \ell_i &:= \left\{(x,y) \mid x=x_{i+1} + (x_i+\frac{\pi}{2}),\, y<y_{i+1}-(y_i+\frac{\pi}{2})\right\}.
	    \end{align*}
	    Note that $\ell_i$ does not include $E_i$.
	    Let
	    \begin{displaymath}
	        \mathcal T_U = \{X_i \mid i<0\} \cup \left(\bigcup_{i<0} \ell_i \right).
	    \end{displaymath}
	    		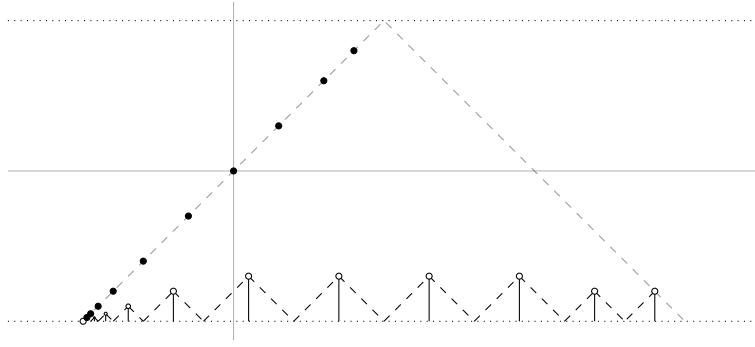
\begin{figure}
		\begin{center}
			\begin{tikzpicture}
				\draw[dotted] (-5,2) -- (5,2);
				\draw[dotted] (-5,-2) -- (5,-2);
				\draw[dashed, draw opacity=.4] (-4,-2) -- (0,2) -- (4,-2);
				\draw[draw opacity = .3] (-5,0) -- (5,0);
				\draw[draw opacity = .3] (-2,2.25) -- (-2,-2.25);
				
				\filldraw[fill=white](-4,-2) circle[radius=.4mm];
				\foreach \x in {-1.95, -1.9, -1.8, -1.6, -1.2, -0.6, 0, 0.6, 1.2, 1.6}
				{
					\filldraw (\x-2, \x) circle[radius=.4mm];
				}
				
				\coordinate (A) at (-3.85,-1.95);
				\coordinate (B) at (-3.8,-2);
				\coordinate (C) at (-3.9,-2);
				\coordinate (D) at (-3.85,-2);
				\draw[dashed] (C) -- (A);
				\draw[dashed] (A) -- (B);
				\draw (A) -- (D);
				\filldraw[fill=white] (A) circle[radius=.1mm];
				\coordinate (A) at (-3.7,-1.9);
				\coordinate (B) at (-3.6,-2);
				\coordinate (C) at (-3.8,-2);
				\coordinate (D) at (-3.7,-2);
				\draw[dashed] (C) -- (A);
				\draw[dashed] (A) -- (B);
				\draw (A) -- (D);
				\filldraw[fill=white] (A) circle[radius=.2mm];
				\coordinate (A) at (-3.4,-1.8);
				\coordinate (B) at (-3.2,-2);
				\coordinate (C) at (-3.6,-2);
				\coordinate (D) at (-3.4,-2);
				\draw[dashed] (C) -- (A);
				\draw[dashed] (A) -- (B);
				\draw (A) -- (D);
				\filldraw[fill=white] (A) circle[radius=.3mm];
				\coordinate (A) at (-2.8,-1.6);
				\coordinate (B) at (-2.4,-2);
				\coordinate (C) at (-3.2,-2);
				\coordinate (D) at (-2.8,-2);
				\draw[dashed] (C) -- (A);
				\draw[dashed] (A) -- (B);
				\draw (A) -- (D);
				\filldraw[fill=white] (A) circle[radius=.4mm];
				\coordinate (A) at (-1.8,-1.4);
				\coordinate (B) at (-1.2,-2);
				\coordinate (C) at (-2.4,-2);
				\coordinate (D) at (-1.8,-2);
				\draw[dashed] (C) -- (A);
				\draw[dashed] (A) -- (B);
				\draw (A) -- (D);
				\filldraw[fill=white] (A) circle[radius=.4mm];
				\coordinate (A) at (-0.6,-1.4);
				\coordinate (B) at (0,-2);
				\coordinate (C) at (-1.2,-2);
				\coordinate (D) at (-0.6,-2);
				\draw[dashed] (C) -- (A);
				\draw[dashed] (A) -- (B);
				\draw (A) -- (D);
				\filldraw[fill=white] (A) circle[radius=.4mm];
				\coordinate (A) at (0.6,-1.4);
				\coordinate (B) at (1.2,-2);
				\coordinate (C) at (0,-2);
				\coordinate (D) at (0.6,-2);
				\draw[dashed] (C) -- (A);
				\draw[dashed] (A) -- (B);
				\draw (A) -- (D);
				\filldraw[fill=white] (A) circle[radius=.4mm];
				\coordinate (A) at (1.8,-1.4);
				\coordinate (B) at (2.4,-2);
				\coordinate (C) at (1.2,-2);
				\coordinate (D) at (1.8,-2);
				\draw[dashed] (C) -- (A);
				\draw[dashed] (A) -- (B);
				\draw (A) -- (D);
				\filldraw[fill=white] (A) circle[radius=.4mm];
				\coordinate (A) at (2.8,-1.6);
				\coordinate (B) at (3.2,-2);
				\coordinate (C) at (2.4,-2);
				\coordinate (D) at (2.8,-2);
				\draw[dashed] (C) -- (A);
				\draw[dashed] (A) -- (B);
				\draw (A) -- (D);
				\filldraw[fill=white] (A) circle[radius=.4mm];
				\coordinate (A) at (3.6,-1.6);
				\coordinate (B) at (3.2,-2);
				\coordinate (C) at (4,-2);
				\coordinate (D) at (3.6,-2);
				\draw[dashed] (A) -- (B);
				\draw (A) -- (D);
				\filldraw[fill=white] (A) circle[radius=.4mm];
			\end{tikzpicture}
			\caption{The $\TT$-cluster $\mathcal T_U$ used to construct the finite embeddings.}\label{fig:embedding cluster}
		\end{center}
		\end{figure}
		As in $\mathcal T_1$, we have the smaller triangles along the bottom.
		However, $\mathcal T_U$ has a rightmost triangle on the bottom.
		One may verify that $\mathcal T_U$ is a $\TT$-cluster.
		
		By Lemma \ref{lem:minimal tilting rectangle}, only the discrete set of points on $\mathcal Z_+$ are $\TT$-mutable.
		As in Section \ref{sec:t-structures}, $\mathcal D^\heartsuit$ denotes the heart of the $t$-structure induced by $\mathcal{Z}_+$ in the category $\mathcal{D}$.
		
		\begin{proposition}\label{prop:big finite representation embedding}
	        For $n\geq 2$, there is an exact embedding $\rep(A_n)\to \mathcal D^\heartsuit$ of abelian categories determined by $P_{-i}\mapsto X_i$, for $-n\leq i\leq -1$.
	        The embedding factors as $\rep(A_n)\to \rep(A_{n+1})\to \mathcal D^\heartsuit$.
		\end{proposition}
		\begin{proof}
		    From the paragraph preceding the proposition, it follows that all pairwise distinct indecomposable projective objects in $\rep(A_n)$ are sent to pairwise distinct indecomposable projective objects in $\mathcal D^\heartsuit$.
		    Furthermore, for any two indecomposables $E, F$ in $\rep(A_n)$, the space of morphisms from $E$ to $F$ in $\rep(A_n)$ is isomorphic to the space of morphisms of the corresponding indecomposables in $\mathcal D^\heartsuit$.
		    From this, one may check that the embedding is exact.
		    The factorization follows from the definitions of the embeddings.
		\end{proof}paperbiblio
		
		The indecomposable objects of $\mathcal D^b(A_n)$ and $\mathcal D$ are shifts of copies of $\rep(A_n)$ and $\mathcal D^\heartsuit$, respectively.
		The following proposition then follows from straightforward computations.
		
		\begin{proposition}\label{prop:big finite derived embedding}
		    With the same notation as above, for each $n\geq 2$, there is a triangulated embedding $\Theta_n:\mathcal D^b(A_n)\to \mathcal D$ determined by sending $P_{-i}[m]\to X_i[m]$, for all $m\in\mathbb Z$.
	        Furthermore, the embedding factors as $\mathcal D^b(A_n)\to \mathcal D^b(A_{n+1})\to \mathcal D$.
		\end{proposition}
		
		Consider again the $\TT$-cluster $\mathcal T_U$ defined at the beginning of this section (depicted in Figure \ref{fig:embedding cluster}).
		The following notation will be of use in the proof of Theorem \ref{thm:big finite embedding}. Recall $\mathcal I^+$ as defined paperbiblioin Section \ref{sec:Amplituhedron for Dynkin Quivers}, page \pageref{tag:I +}.
		\begin{notation}\label{note:theta of cluster}
		Let $\mathcal S$ be a cluster of type $A_n$ consisting of indecomposables in $\repPlus$.
		Using $\Theta_n$ in Proposition \ref{prop:big finite derived embedding}, set $\widetilde{\mathcal S} := \left(\mathcal T_U\setminus \{X_i\}_{i\geq -n}\right)\cup \Theta(\mathcal S)$, where  $\{X_i \mid i\in\mathbb Z_{<0}\}$ is as before.
		\end{notation}
		
		\begin{proposition}\label{prop:embedding T computable}
		   With the same notation $\mathcal T_U$ as above, there is a unique solution $\Phi$ with respect to $\intc$
		   in $\mathcal{C}_{\mathcal{Z}_+}$ such that $\Phi(X)=0$, for all $X\in\mathcal T_U$.
		\end{proposition}
		\begin{proof}
		    Let $\Phi(X)=0$ for all $X\in\mathcal T_U$.
		    Then $\Phi$ is defined
		    on a countable number of points on $\mathcal Z_+$, but not on all of $\mathcal Z_+$.
		    By Lemma~\ref{lem:vertical solution}, if we consider each of the small triangles in the bottom of Figure~\ref{fig:T1} as a triangular patch, 
		    this $\Phi$ uniquely extends to all of them.
		    From here it is straightforward to check that we may uniquely extend $\Phi$ to all of $\mathcal Z_+$.
		    Then, Proposition \ref{prop:all the values} implies that $\Phi$ uniquely extends to all of $\mathcal{C}_{\mathcal{Z}_+}$.
		    Thus, we have a unique $\Phi$ as in the statement of the proposition.
		\end{proof}
		
		For each $A_n$, by Proposition \ref{prop:big finite derived embedding}, we may consider $\repPlus$ as a subcategory of $\mathcal{C}_{\mathcal{Z}_+}$.
		The almost split triangles in $\repPlus$, inherited from $\mathcal D^b(A_n)$, determine tilting rectangles in $\mathcal{C}_{\mathcal{Z}_+}$.
		Assign to each indecomposable $E$ in $\rep (A_n)$ the value of $\int \intc$ over the corresponding tilting rectangle in $\mathcal{C}_{\mathcal{Z}_+}$ determined by the almost split triangle starting with $E$.
		While not integer values, we may  still find values for the deformed mesh relations in $\repPlus$.paperbiblio
		Denote by $\mathbb U_{n,\intc}$ the associahedron obtained from $A_n$ using $\intc$ in this way.
	
		\begin{theorem}\label{thm:big finite embedding}
		    There is an infinite sequence of embeddings
		    \begin{displaymath}
			\mathbb U_{2,\intc}\hookrightarrow\mathbb U_{3,\intc}\hookrightarrow\cdots\hookrightarrow\mathbb U_{n,\intc}\hookrightarrow\mathbb U_{n+1,\intc}\hookrightarrow\cdots  \mathbb{U}_{\mathcal{Z}_+,\underline{c}}.
		\end{displaymath}
		For $n\geq 2$, the composition of embeddings $\mathbb U_{n,\intc}\to \mathbb{U}_{\mathcal{Z}_+,\underline{c}}$ takes the point corresponding to any cluster $\mathcal S$ to the unique solution $\Phi$ with respect to $\intc$ 
		in $\mathcal{C}_{\mathcal{Z}_+}$ such that $\Phi(X)=0$ for all $X\in\widetilde{\mathcal S}$ (Notation \ref{note:theta of cluster}).
		Furthermore, the composition $\mathbb U_{n,\intc}\to \mathbb{U}_{\mathcal{Z}_+,\underline{c}}$ takes a mutation edge to a $\TT$-mutation edge.
		\end{theorem}
		\begin{proof}
		    We first show that each $\Theta_n|_{\repPlus}$ in Proposition \ref{prop:big finite derived embedding} takes clusters in $\repPlus$ to $\TT$-clusters in $\mathcal{C}_{\mathcal{Z}_+}$.
		    Then we show that $\Theta_n|_{\repPlus}$ takes mutations to $\TT$-mutations.
		    
		    Let $\mathcal S$ be a cluster in $\repPlus$ and $\widetilde{\mathcal S}$ be as in Notation~\ref{note:theta of cluster}.
		    Suppose $Y\in\mathcal{C}_{\mathcal{Z}_+}$ but $Y\notin\widetilde{\mathcal S}$.
		    We show that there must exist a $T\in\widetilde{\mathcal S}$ such that $T$ and $Y$ are incompatible.
		    If there exists $T\in\ell_i$, fpaperbiblioor some $i<0$, such that $T$ and $Y$ are incompatible we are done.
		    So, for every $i<0$, suppose $Y$ is compatible with all $X\in\ell_i$.
		    Using Figure \ref{fig:embedding cluster}, we see the rays extending from $Y$ in the negative $y$-direction with slope $-1$ and $+1$ must intersect some $E_j$ and $E_i$, respectively.
		    Otherwise, one of the rays intersects one of the $\ell_i$'s, which contradicts our assumption. This is because we took $A_n$ to be the linearly ordered quiver with $P_i\hookrightarrow P_j$ if $i>j$.
		    
		    If $E_i=E_{-1}$ then $Y=X_j[1]$ because $Y\in\Ind(\mathcal Z_+[1])$, and there is a distinguished triangle $X_j\to 0 \to Y\stackrel{\cong}{\to} X_j[1]$.
		    If $E_i\neq E_{-1}$, there is a distinguished triangle $X_j\to X_{i+1}\to Y\to$ in $\mathcal D$ with all terms in $\mathcal{C}_{\mathcal{Z}_+}$.
		    In either case, we use the fact that $\Theta_n$ is a triangulated embedding by Proposition~\ref{prop:big finite derived embedding}.
		    This implies there is a distinguished triangle $P_j\to W\to F\to$ in $\mathcal D^b(A_n)$ whose terms are in $\repPlus$, such that $Y=\Theta_n(F)$.
		    Since $Y\notin \widetilde{\mathcal S}$, we have $F\notin \mathcal S$.
		    Hence, there is $S\in\mathcal S$ such that $S$ and $F$ are incompatible.
		    Using Proposition \ref{prop:big finite derived embedding} again, we have $\Theta_n(S)$ and $Y$ are incompatible. Now $\Theta_n(S)$ is the desired $T\in\widetilde{\mathcal S}$ such that $T$ and $Y$ are incompatible. This shows that each $\Theta_n|_{\repPlus}$ in Proposition \ref{prop:big finite derived embedding} takes clusters in $\repPlus$ to $\TT$-clusters in $\mathcal{C}_{\mathcal{Z}_+}$.
		    
		    Now, we let $\mu:\mathcal S\to\mathcal S'$ be a mutation of clusters in $\repPlus$.
		    Then $\widetilde{\mathcal S}$ and $\widetilde{\mathcal S'}$ differ by one element.
		    Thus, by Proposition \ref{prop:mutation yields a unique cluster}, $\widetilde{\mathcal S}\to \widetilde{\mathcal S'}$ is a $\TT$-mutation.
		    
		    We finish the proof by showing that each $\Theta_n$ induces a geometric embedding $\mathbb U_{n,\intc}\to \mathbb{U}_{\mathcal{Z}_+,\underline{c}}$.
		    We index the coordinates of the ambient space $\prod_{\mathcal I^+}\mathbb R$ of $\mathbb U_{n,\intc}$ by the indecomposables in $\mathcal I^+$.
		    Then $\Theta_n|_{\repPlus}$ induces a geometric embedding $\prod_{\mathcal I^+}\mathbb R\hookrightarrow \prod_{\Ind(\mathcal{C}_{\mathcal{Z}_+})}\R$.
		    Note that the deformed mesh relations in $\repPlus$ (Section \ref{sec:Amplituhedron for Dynkin Quivers}, Equation~\eqref{defmesh}) satisfy the continuous deformed mesh relations in $\mathcal{C}_{\mathcal{Z}_+}$ (Definition \ref{def:contintuous deformed mesh relations}).
		    Thus, the embedding takes nonnegative solutions to nonnegative solutions.
		    Given our embeddings, for each $n\geq 2$ we may consider an $A_n$ cluster and complete it to an $A_{n+1}$ cluster by including the projective $P_{n+1}$.
            By this assumption and the paragraph before the theorem, the deformed mesh relations for $A_n$ can be viewed as a set of deformed mesh relations for $A_{n+1}$.
	
		    Furthermore, for each $A_n$ cluster $\mathcal S$, there exists a sequence of finitely many mutations $\{\mu_i\}_{i=1}^m$ such that $\mu_m\cdots\mu_1$ takes $\mathcal S$ to the cluster of projective indecomposables.
		    Therefore, by our previous argument, the corresponding sequence of $\TT$-mutations of $\widetilde{\mathcal S}$ forms a finite sequence that takes $\widetilde{\mathcal S}$ to $\mathcal T_U$.
		    Then by Proposition \ref{prop:embedding T computable}, for each $\widetilde{\mathcal S}$ above, there is a unique solution $\Phi$ with respect to $\intc$ in $\mathcal{C}_{\mathcal{Z}_+}$ such that $\Phi(X)=0$, for all $X\in\widetilde{\mathcal S}$.
		    Finally, since each $\Theta_n$ takes $A_n$ clusters to $\TT$-clusters and mutations to $\TT$-mutations, Theorem \ref{thm:mutation on associahedron} implies the final statement of the theorem.
		\end{proof}
		We remark that in Theorem \ref{thm:big finite embedding} there is no $\mathbb U_{n,\intc}$ (or indeed any discrete associahedron) that immediately precedes $\mathbb{U}_{\mathcal{Z}_+,\underline{c}}$.

        We finish this section by pointing out that an interesting uniform treatment of all generalized associahedra of a given finite Dynkin type has appeared in \cite{HPS18}, in terms of what the authors called the ``universal associahedron". The universal associahedron is a different generalization from the continuous associahedron introduced in our work. 
        In particular, in Theorem~\ref{thm:big finite embedding}, we vary $n$ but maintain a given zigzag (which can be viewed as an initial acyclic seed), whereas in the study of the universal associahedron the authors work with different initial seeds for a fixed $n$. For more details, see \cite{HPS18} and the references therein.

\end{document}